\documentclass[a4paper,10pt]{amsart}
\usepackage[T1]{fontenc}
\usepackage{lmodern}
\usepackage{amsmath}
\usepackage{amsfonts}
\usepackage{amssymb}
\usepackage{amsthm}
\usepackage{thmtools}
\usepackage{mathtools}
\usepackage{enumitem}
\usepackage{placeins}
\usepackage{caption}
\usepackage{tikz-cd}
\usepackage[pdftex,
            pdfauthor={},
            pdftitle={Character Identities Between Affine and Virasoro Vertex Operator Algebra Modules},
            pdfsubject={},
            pdfkeywords={},
            pdfproducer={},
            pdfcreator={}]{hyperref}

\theoremstyle{plain}
\newtheorem{thm}{Theorem}[section]
\newtheorem{prop}[thm]{Proposition}
\newtheorem{cor}[thm]{Corollary}

\newtheorem{conj}[thm]{Conjecture}

\newtheorem*{thm*}{Theorem}
\newtheorem*{conj*}{Conjecture}
\newtheorem*{claim*}{Claim}
\newtheorem*{prop*}{Proposition}

\theoremstyle{definition}

\newtheorem*{nota*}{Notation}
\newtheorem{rem}[thm]{Remark}
\newtheorem{prob}[thm]{Problem}

\newtheoremstyle{introTheorems}% name of the style to be used
  {}% measure of space to leave above the theorem. E.g.: 3pt
  {}% measure of space to leave below the theorem. E.g.: 3pt
  {\itshape}% name of font to use in the body of the theorem
  {}% measure of space to indent
  {\bfseries}% name of head font
  {}
  { }% space after theorem head; " " = normal interword space
  {\thmname{#1}
  \textnormal{\bf \thmnote{#3}.$\hspace{-.1cm}$}
  }
\theoremstyle{introTheorems}

\newtheorem{introProp}{Proposition}
\newtheorem{introCor}{Corollary}
\newtheorem{introConj}{Conjecture}

\newcommand{\Q}{\mathbb{Q}}
\newcommand{\Z}{\mathbb{Z}}
\newcommand{\Ns}{\mathbb{Z}_{>0}}
\newcommand{\N}{\mathbb{Z}_{\geq0}}
\newcommand{\C}{\mathbb{C}}

\newcommand{\wt}{\operatorname{wt}}
\newcommand{\tr}{\operatorname{tr}}
\renewcommand{\i}{i}% imaginary unit
\newcommand{\e}{e}% Euler's number
\newcommand{\End}{\operatorname{End}}
\newcommand{\Aut}{\operatorname{Aut}}

\newcommand{\voa}{vertex operator algebra}
\newcommand{\Voa}{Vertex operator algebra}
\newcommand{\VOA}{Vertex Operator Algebra}

\newcommand{\svoa}{vertex operator superalgebra}

\newcommand{\aia}{abelian intertwining algebra}

\newcommand{\vac}{\textbf{1}}
\newcommand{\ch}{\operatorname{ch}}
\newcommand{\sch}{\operatorname{sch}}

\newcommand{\id}{\operatorname{id}}

\newcommand{\g}{\mathfrak{g}}
\renewcommand{\sl}{\mathfrak{sl}}

\newcommand{\T}{\mathcal{T}}
\newcommand{\I}{\mathcal{I}}

\newcommand{\strat}{strongly rational}

\newcommand{\no}{\,{\raise0.25em\hbox{$\mathop{\hphantom{\cdot}}\limits^{_{\circ}}_{^{\circ}}$}}\,}

\newcommand{\Rep}{\operatorname{Rep}}

\newcommand{\SU}{\mathrm{SU}}

\begin{document}

\title[{Character Identities Between Affine and Virasoro VOA Modules}]{Character Identities Between Affine and Virasoro Vertex Operator Algebra Modules}
\author[Dražen Adamović and Sven Möller]{Dražen Adamović\textsuperscript{\lowercase{a}} and Sven Möller\textsuperscript{\lowercase{b}}}
\thanks{\textsuperscript{a}{Department of Mathematics, Faculty of Science, University of Zagreb, Bijeni\v{c}ka cesta 30, 10000 Zagreb, Croatia}}
\thanks{\textsuperscript{b}{Fachbereich Mathematik, Universität Hamburg, Bundesstraße 55, 20146 Hamburg, Germany}}
\thanks{Email: \href{mailto:adamovic@math.hr}{\nolinkurl{adamovic@math.hr}}, \href{mailto:math@moeller-sven.de}{\nolinkurl{math@moeller-sven.de}}}

\begin{abstract}
The affine \voa{}s for $\sl_2$ and the Virasoro minimal models are related by Drinfeld-Sokolov reduction and by the Goddard-Kent-Olive coset construction. In this work, we propose another connection based on certain character identities between these \voa{}s and their modules. This relates the simple affine \voa{}s $L_k(\sl_2)$ at admissible levels $k=-2+q/p$ to the rational $(q,3p)$-minimal models $L_\mathrm{Vir}(c_{q,3p},0)$, and also extends to the nonadmissible levels with $q=1$.

Several special cases are particularly interesting. In the nonadmissible case $q=1$, the character identities extend to certain \aia{}s, specifically $\mathcal{V}^{(p)}$ and the doublet $\smash{\mathcal{A}^{(3p)}}$. Specialising further to $p=2$, where $\smash{\mathcal{V}^{(2)}}$ is the simple small $\mathcal{N}=4$ superconformal algebra of central charge $\smash{c=-9}$, this recovers, via the 4d/2d-correspondence, a known identity between the Schur indices of the 4d $\mathcal{N}=4$ supersymmetric Yang-Mills theory for $\SU(2)$ and the 4d $\mathcal{N}=2$ $(3,2)$ Argyres-Douglas theory.

In the boundary admissible case $q=2$, in a similar vein, we obtain an identity between the Schur indices of 4d $\mathcal{N}=2$ Argyres-Douglas theories of types $(A_1,D_{2n+1})$ and $(A_1,A_{6n})$.

On the other hand, for integral levels, $p=1$, where both involved \voa{}s are \strat{}, our character identity induces a Galois conjugation between the representation categories $\smash{\Rep(L_{-2+q}(\sl_2))}$ and $\smash{\Rep(L_\mathrm{Vir}(c_{q,3},0))}$; and for small values of $q$, the characters are related by the action of certain Hecke operators.

Finally, we also sketch how to extend the results of this paper to relaxed highest-weight and Whittaker modules.
\end{abstract}

\maketitle

\setcounter{tocdepth}{1}
\tableofcontents
\setcounter{tocdepth}{3}

\pagebreak

%%%%%%%%%%%%%%%%%%%%%%%%%%%%%%%
%%%%%%%%%%%%%%%%%%%%%%%%%%%%%%%

\section{Introduction}

\Voa{}s and their representation categories axiomatise two-dimensional conformal field theories in physics (see, e.g., \cite{FLM88,Kac98,LL04,FBZ04}). More recently, \voa{}s have also been shown to capture important aspects of higher-dimensional superconformal field theories.

\medskip

\noindent\emph{Background.} Important \voa{}s are those associated with infinite-dimensional Lie algebras. Among the best-studied examples are affine \voa{}s $V^k(\g)$, which are constructed from affinisations $\hat\g=\g\otimes\C[t,t^{-1}]\oplus\C K$ of finite-dimensional, simple Lie algebras $\g$, and the Virasoro \voa{} $V_\mathrm{Vir}(c,0)$, which is associated with the Virasoro Lie algebra $\mathcal{L}=\bigoplus_{s\in\Z}\C L_s\oplus\C C$. In both cases, the \voa{}s and their modules are constructed on certain induced modules of these infinite-dimensional Lie algebras.

Specialising to the case $\g=\sl_2=\langle e,h,f\rangle_\C$, it turns out that there are several interesting and well-studied connections between the representation theory of the affine \voa{}s $V^k(\sl_2)$ and the Virasoro \voa{}s $V_\mathrm{Vir}(c,0)$. Quantum Drinfeld-Sokolov reduction (or quantum Hamiltonian reduction) $\smash{H_\mathrm{DS}^0}$ \cite{FF90c,FF92} relates the universal affine \voa{} $V^k(\sl_2)$ at level $k\in\C\setminus\{-2\}$ to the universal Virasoro \voa{} $\smash{V_\mathrm{Vir}(c,0)=H_\mathrm{DS}^0(V^k(\sl_2))}$ at central charge $c=c(k)=1-6(k+1)^2/(k+2)$. It lifts to an exact functor between suitable representation categories.

For the admissible levels $k=k_{q,p}\coloneqq-2+q/p$ with $q,p\in\Z$, $q\geq2$, $p\geq1$ and $(p,q)=1$, the universal affine \voa{} $V^{-2+q/p}(\sl_2)$ is not simple, and we can consider the simple quotient $L_{-2+q/p}(\sl_2)$. In a similar way, for $c=c_{q,p}\coloneqq 1-6(p-q)^2/(pq)$ with $p,q\in\Z$, $p,q\geq2$ and $(p,q)=1$, the universal Virasoro \voa{} $V_\mathrm{Vir}(c_{q,p},0)$ is not simple. In that case, the simple quotient $L_\mathrm{Vir}(c_{q,p},0)$ is (strongly) rational. Quantum Drinfeld-Sokolov reduction also relates the simple quotients $L_{-2+q/p}(\sl_2)$ and $L_\mathrm{Vir}(c_{q,p},0)$, with the same values of $q$ and $p$, i.e.\ $c_{q,p}=c(k_{q,p})$.

Another well-known relation between \voa{}s for $\smash{\hat\sl_2}$ and for the Virasoro algebra $\mathcal{L}$ is given by the (generalised) Goddard-Kent-Olive coset construction, which realises the simple quotient $L_\mathrm{Vir}(c_{k+2,k+3},0)$ as a coset (or commutant) of $L_{k+1}(\sl_2)$ inside $L_k(\sl_2)\otimes L_1(\sl_2)$ \cite{GKO86}.

\medskip

\noindent\emph{Motivation.} In this paper, based on character identities, we investigate a new relation between $V^k(\sl_2)$ and $V_\mathrm{Vir}(c,0)$, between their simple quotients and between their modules. A special case of this appeared in \cite{BN22}. The starting point of this correspondence is the seemingly innocent observation (see \autoref{prop:gvsiuniversal}) that, by the Poincaré-Birkhoff-Witt theorem, there is a vector-space isomorphism
\begin{equation*}
\psi^\pm\colon V^k(\sl_2)\overset{\sim}{\longrightarrow}V_\mathrm{Vir}(c,0)
\end{equation*}
between the universal \voa{}s for $k\in\C\setminus\{-2\}$ and $c\in\C$ defined by mapping
\begin{equation}\label{eq:intro1}
e_{-i}\mapsto L_{-3i\pm1},\quad h_{-i}\mapsto L_{-3i},\quad f_{-i}\mapsto L_{-3i\mp1}
\end{equation}
for $i\in\Ns$ and $\vac\mapsto\vac$. This isomorphism $\psi^\pm$ maps a homogeneous vector of $h_0$-weight $f$ and $L_0$-weight $n$ to a vector of $L_0$-weight $n'=\mp f/2+3n$. Hence, it induces an identity of formal (Jacobi) $q$-characters (see \autoref{sec:reps} for the definition)
\begin{equation*}
\ch^*_{V^k(\sl_2)}(q^{\mp1/2},q^3)=\ch^*_{V_\mathrm{Vir}(c,0)}(q),
\end{equation*}
see \autoref{prop:chariduniversal}.

\medskip

\noindent\emph{Verma Modules.} It turns out that the above graded vector-space isomorphism can be vastly generalised, first of all, by also mapping $f_0\mapsto L_{-1}$, to a correspondence between Verma modules for $V^k(\sl_2)$ and $V_\mathrm{Vir}(c,0)$,
\begin{equation*}
\varphi^+\colon M_k(\mu)\overset{\sim}{\longrightarrow}M_\mathrm{Vir}(c,h);
\end{equation*}
see \autoref{prop:vermacorr}. An analogous version $\varphi^-$ for lowest-weight Verma modules $M^-_k(-\mu)$ for $V^k(\sl_2)$, mapping $e_0\mapsto L_{-1}$, exists as well; see \autoref{prop:vermacorr2}.

Now, returning to the correspondence for the universal \voa{}s, we recall that both \voa{}s can be constructed as quotients of Verma modules by singular vectors. Indeed, the universal affine \voa{} $V^k(\sl_2)$, which is nothing but the trivial Weyl module $V^k(0)$, is a proper quotient of the Verma module $M_k(0)$ by a submodule generated by one singular vector of $h_0$-weight $-2$ and $L_0$-weight $0$. Similarly, the universal Virasoro \voa{} $V_\mathrm{Vir}(c,0)$ is formed from the Verma module $M_\mathrm{Vir}(c,0)$ by taking the quotient corresponding to the singular vector $L_{-1}\vac$ of $L_0$-weight~$1$. We denote both these singular vectors by $\raisebox{.5pt}{\textcircled{\raisebox{-.9pt}{1}}}$.

Importantly, we notice that the graded vector-space isomorphism $\varphi^+$ respects the weights of these singular vectors, i.e.\ it maps the choice of singular vector in $M_k(0)$ to a vector of $L_0$-weight $1$ in $M_\mathrm{Vir}(c,0)$. This then implies the existence (nonuniquely) of a graded vector-space isomorphism $V^k(\sl_2)\overset{\sim}{\longrightarrow}V_\mathrm{Vir}(c,0)$, like $\psi^+$, between the universal \voa{}s (as quotients of Verma modules) and the corresponding character identity, and similarly for $\varphi^-$ and $\psi^-$.

As a warning, we point out that $\varphi^\pm$ does not map the singular vector exactly to the singular vector $L_{-1}\vac$, so that we cannot say that the map $\psi^\pm$ is induced by $\varphi^\pm$ on the quotients. However, $\psi^\pm$ and the linear map induced by $\varphi^\pm$ act in the same way on the level of graded components. We plan to investigate this in more detail in future work, which may also help to prove the classical freeness of certain \voa{}s \cite{EH21,AEH23}; see \autoref{rem:classicallyfree} and below.

\medskip

\noindent\emph{Singular Vectors.} The essential observation is that the compatibility of the linear isomorphism $\varphi^+\colon M_k(\mu)\overset{\sim}{\longrightarrow}M_\mathrm{Vir}(c,h)$ between the Verma modules with the weights of the singular vectors $\raisebox{.5pt}{\textcircled{\raisebox{-.9pt}{1}}}$ extends to nonzero $\mu\in\C$ and $h\in\C$ and to a second family $\raisebox{.5pt}{\textcircled{\raisebox{-.9pt}{2}}}$ of singular vectors (though, depending on the level $k\in\C\setminus\{-2\}$ and the central charge $c\in\C$, these singular vectors may or may not be present). This then allows us to obtain graded vector-space isomorphisms, in the same way as before, between irreducible modules for $V^k(\sl_2)$ and $V_\mathrm{Vir}(c,0)$. Indeed, as we explain in \autoref{sec:affine} and \autoref{sec:virasoro}, one sometimes needs two singular vectors to get from Verma modules to their irreducible quotients.

In this work, we restrict our study to two cases of levels for $V^k(\sl_2)$, namely the admissible levels $k=k_{q,p}=-2+q/p$ with $q\geq2$ and $p\geq1$ and certain nonadmissible levels $k=k_{1,p}=-2+1/p$ with $p\geq1$, which we call near-admissible. In both cases, in order to match the weights of the singular vectors $\raisebox{.5pt}{\textcircled{\raisebox{-.9pt}{1}}}$ and $\raisebox{.5pt}{\textcircled{\raisebox{-.9pt}{2}}}$ in the Verma modules for $V^k(\sl_2)$ with those in the Verma modules for $V_\mathrm{Vir}(c,0)$, we relate, assuming $3\nmid q$, modules for $k=k_{q,p}$ to modules for $c=c_{q,3p}$,
\begin{equation*}
(q,p)\mapsto(q,3p).
\end{equation*}
This relates the admissible levels $k_{q,p}$ to the rational cases with central charge $c_{q,3p}$ and, extending both formulae to $q=1$, the near-admissible levels $k_{1,p}$ to the logarithmic cases $c_{1,3p}$. We contrast this to the well-known Drinfeld-Sokolov reduction, which relates them for $k=k_{q,p}$ and $c=c_{q,p}$ with the same $p$ and~$q$.

\medskip

\noindent\emph{Admissible Levels.} First, in the admissible case, we consider irreducible modules $L_{-2+q/p}(\mu_{r,s})$ and $L_\mathrm{Vir}(c_{q,p},h_{r,s}^{q,p})$ for the simple quotient \voa{}s $L_{-2+q/p}(\sl_2)$ and $L_\mathrm{Vir}(c_{q,3p},0)$, as summarised in \autoref{fig:admrep} and \autoref{fig:ratrep}, respectively, recalling that only some modules for the universal \voa{}s are also modules for the simple quotients. The irreducible modules are obtained by taking quotients of Verma modules by two singular vectors.
\begin{introProp}[\ref{prop:vermasingadm}]
Let $q,p\in\Z$ with $q\geq2$, $p\geq1$ and $(p,q)=1$. Moreover, assume that $3\nmid q$. Let $r,s\in\Z$ with $0\leq r\leq q-2$ and $0\leq s\leq p-1$. Then there is a vector-space isomorphism
\begin{equation*}
\varphi^+\colon M_{-2+q/p}(\mu_{r,s})\overset{\sim}{\longrightarrow}M_\mathrm{Vir}(c_{q,3p},h_{r+1,3s+1}^{q,3p})
\end{equation*}
between Verma modules for $L_{-2+q/p}(\sl_2)$ and $L_\mathrm{Vir}(c_{q,3p},0)$ that maps a vector of $(h_0,L_0)$-weight $(f,n)$ to a vector of $L_0$-weight $n'=-f/2+3n$. The singular vectors $\raisebox{.5pt}{\textcircled{\raisebox{-.9pt}{1}}}$ and $\raisebox{.5pt}{\textcircled{\raisebox{-.9pt}{2}}}$ in $\smash{M_{-2+q/p}(\mu_{r,s})}$ are mapped to vectors in $\smash{M_\mathrm{Vir}(c_{q,3p},h_{r+1,3s+1}^{q,3p})}$ of the same $L_0$-weights as those of the two singular vectors $\raisebox{.5pt}{\textcircled{\raisebox{-.9pt}{1}}}$ and $\raisebox{.5pt}{\textcircled{\raisebox{-.9pt}{2}}}$, respectively, in $\smash{M_\mathrm{Vir}(c_{q,3p},h_{r+1,3s+1}^{q,3p})}$.
\end{introProp}
Because the weights of the singular vectors are preserved (although not necessarily the singular vectors themselves), the above character identity for the Verma modules descends to a character identity for the irreducible modules.
\begin{introProp}[\ref{cor:charid} (and \autoref{rem:inducedGVSI})]
Let $q,p\in\Z$ with $q\geq2$, $p\geq1$ and $(p,q)=1$. Moreover, assume that $3\nmid q$. Let $r,s\in\Z$ with $0\leq r\leq q-2$ and $0\leq s\leq p-1$. Then the character identity
\begin{equation*}
\ch^*_{L_{-2+q/p}(\mu_{r,s})}(q^{-1/2},q^3)=\ch^*_{L_\mathrm{Vir}(c_{q,3p},h_{r+1,3s+1}^{q,3p})}(q)
\end{equation*}
for irreducible modules for $L_{-2+q/p}(\sl_2)$ and $L_\mathrm{Vir}(c_{q,3p},0)$ holds. Equivalently, there exists a graded vector-space isomorphism
\begin{equation*}
\chi^+\colon L_{-2+q/p}(\mu_{r,s})\overset{\sim}{\longrightarrow}L_\mathrm{Vir}(c_{q,3p},h_{r+1,3s+1}^{q,3p}).
\end{equation*}
\end{introProp}
We point out (see \autoref{rem:inducedGVSI}) that such a vector-space isomorphism $\chi^+$, besides not being unique, would be difficult to write down explicitly, in general, though we do so in the special case of $(q,p)=(4,1)$, as we explain below.

Analogous statements for lowest-weight Verma modules for $\smash{L_{-2+q/p}(\sl_2)}$ based on the map $\varphi^-$ (and $\chi^-$) hold; see \autoref{prop:lowestvermasingadm} and \autoref{cor:lowestcharid}.

In the special case of $r=s=0$, we obtain graded vector-space isomorphisms $\chi^\pm$ and character identities between the simple \voa{}s themselves. They are related to the ones above for the universal \voa{}s by taking the quotient by the ideal generated by the singular vector $\raisebox{.5pt}{\textcircled{\raisebox{-.9pt}{2}}}$.

\medskip

\noindent\emph{Nonadmissible Levels.} In the near-admissible case, we relate irreducible modules for $\smash{L_{-2+1/p}(\sl_2)=V^{-2+1/p}(\sl_2)}$ with $p\in\Ns$ to those of the logarithmic $(1,3p)$-minimal models $\smash{L_\mathrm{Vir}(c_{1,3p},0)=V_\mathrm{Vir}(c_{1,3p},0)}$ \cite{FFT11}, as summarised in \autoref{fig:nearadmrep} and \autoref{fig:logrep}. In contrast to the admissible case, there is now only (at most) one family~$\raisebox{.5pt}{\textcircled{\raisebox{-.9pt}{1}}}$ of singular vectors in the Verma modules on both sides.
\begin{introProp}[\ref{prop:vermasingnear}]
Let $p\in\Ns$. Let $\mu\in\C$. There is a vector-space isomorphism
\begin{equation*}
\varphi^+\colon M_{-2+1/p}(\mu)\overset{\sim}{\longrightarrow}M_\mathrm{Vir}(c_{1,3p},\mu(-2+3p(\mu+2))/4)
\end{equation*}
between Verma modules for $V^{-2+1/p}(\sl_2)$ and $V_\mathrm{Vir}(c_{1,3p},0)$ that maps a vector of $(h_0,L_0)$-weight $(f,n)$ to a vector of $L_0$-weight $n'=-f/2+3n$. If $\mu\notin\N$, both sides are already irreducible.

If $\mu=\mu_{r,0}=r$ for $r\in\N$, this becomes a vector-space isomorphism
\begin{equation*}
\varphi^+\colon M_{-2+1/p}(\mu_{r,0})\overset{\sim}{\longrightarrow}M_\mathrm{Vir}(c_{1,3p},h_{r+1,1}^{1,3p})
\end{equation*}
that maps the singular vector $\raisebox{.5pt}{\textcircled{\raisebox{-.9pt}{1}}}$ in $\smash{M_{-2+1/p}(\mu_{r,0})}$ to a vector in $\smash{M_\mathrm{Vir}(c_{1,3p},h_{r+1,1}^{1,3p})}$ of the same $L_0$-weight as the singular vector $\raisebox{.5pt}{\textcircled{\raisebox{-.9pt}{1}}}$ in $M_\mathrm{Vir}(c_{1,3p},h_{r+1,1}^{1,3p})$.
\end{introProp}
Again, because the weights of the singular vectors are preserved, this implies a character identity for the irreducible modules.
\begin{introProp}[\ref{cor:charid2} (and \autoref{prop:gvsi})]
Let $p\in\Ns$. Let $\mu\in\C\setminus\N$. Then
\begin{equation*}
\ch^*_{L_{-2+1/p}(\mu)}(q^{-1/2},q^3)=\ch^*_{L_\mathrm{Vir}(c_{1,3p},\mu(-2+3p(\mu+2))/4)}(q)
\end{equation*}
where $L_\mathrm{Vir}(c_{1,3p},\mu(-2+3p(\mu+2))/4)=M_\mathrm{Vir}(c_{1,3p},\mu(-2+3p(\mu+2))/4)$ and $L_{-2+1/p}(\mu)=M_{-2+1/p}(\mu)$ are irreducible Verma modules for $V^{-2+1/p}(\sl_2)$ and $V_\mathrm{Vir}(c_{1,3p},0)$, respectively.

On the other hand, let $r\in\N$. Then
\begin{equation*}
\ch^*_{L_{-2+1/p}(\mu_{r,0})}(q^{-1/2},q^3)=\ch^*_{M_{r+1,1;3p}}(q),
\end{equation*}
where $\smash{L_{-2+1/p}(\mu_{r,0})=V^{-2+1/p}(r)}$ is an irreducible Weyl module for $V^{-2+1/p}(\sl_2)$ and $\smash{M_{r+1,1;3p}=L_\mathrm{Vir}(c_{1,3p},h_{r+1,1}^{1,3p})}$ an irreducible module for $V_\mathrm{Vir}(c_{1,3p},0)$. Equivalently, there is a graded vector-space isomorphism
\begin{equation*}
\psi^+\colon L_{-2+1/p}(\mu_{r,0})\overset{\sim}{\longrightarrow}M_{r+1,1;3p}.
\end{equation*}
\end{introProp}
Analogous results for the map $\varphi^-$ (and $\psi^-$) from lowest-weight Verma modules for $V^{-2+1/p}(\sl_2)$ hold; see \autoref{prop:lowestvermasingnear}, \autoref{cor:lowestcharid2} and \autoref{prop:gvsi}.

For $r=0$ we obtain the graded vector-space isomorphisms $\psi^\pm$ and the character identities between the \voa{}s themselves. They coincide with the ones described above for universal \voa{}s (which are simple here).

In the near-admissible case (in contrast to the admissible case, where there is more than one singular vector), one can construct a graded vector-space isomorphism $\psi^\pm$ explicitly; see \autoref{prop:gvsi}. This was done in \cite{BN22} in the special case of $p=2$, but their proof generalises straightforwardly. This result partially served as motivation for this paper.

\medskip

\noindent\emph{Generalisations.} We also propose how to extend the vector-space isomorphisms to modules not necessarily having finite-dimensional weight spaces or being graded at all, in particular relaxed highest-weight modules and Whittaker modules.

The relaxed modules $\smash{\mathcal{E}^{\sl_2}_{\lambda,\chi}}$ \cite{AM95,RW15,CMY24}, defined in \eqref{eq:real-relaxed}, are $\N$-gradable $\smash{V^k(\sl_2)}$-modules whose top component is a weight $\sl_2$-module that is neither a lowest- nor highest-weight module. In \eqref{eq:Vir-relaxed}, we also introduce the analogous weight Virasoro modules $\mathcal{E}_{\lambda,\chi}^\mathrm{Vir}$ as certain $\N$-gradable $V_\mathrm{Vir}(c,0)$-modules, now with unbounded $L_0$-weights (and infinite-dimensional weight-spaces). While the former admit formal (Jacobi) characters, the latter do not.
\begin{introProp}[\ref{prop:relaxed}]
Let $k\in\C\setminus\{-2\}$ and $c\in\C$. Assume that $(\lambda,\chi)\in\C^2$. Then there is a graded vector-space isomorphism $\varphi^+\colon\mathcal{E}_{\lambda,\chi}^{\sl_2}\overset{\sim}{\longrightarrow}\mathcal{E}_{\lambda,\chi}^\mathrm{Vir}$.
\end{introProp}

We also find a vector-space isomorphism between certain Whittaker modules $\operatorname{Wh}^k_{\hat{\sl}_2}(\lambda,\mu)$ for $V^k(\sl_2)$ and $\operatorname{Wh}^c_\mathrm{Vir}(a,b)$ for $V_\mathrm{Vir}(c,0)$. These are weak vertex algebra modules and admit neither $L_0$- nor $h_0$-gradings nor characters in the usual way. Note that both modules are irreducible if all parameters $\lambda,\mu,a,b\in\C$ are nonzero \cite{ALZ16,OW09}.
\begin{introProp}[\ref{prop:isomorphism-whittaker}]
Let $k\in\C$ and $c\in\C$. Let $\lambda,\mu\in\C\setminus\{0\}$ and $a,b\in\C\setminus\{0\}$. Then the map $\varphi^+\colon\operatorname{Wh}^k_{\hat{\sl}_2}(\lambda,\mu)\overset{\sim}{\longrightarrow}\operatorname{Wh}^c_\mathrm{Vir}(a,b)$ is a vector-space isomorphism.
\end{introProp}

\medskip

\noindent\emph{Integral Levels.} Returning to highest- and lowest weight modules, we mention a few interesting special cases. The nicest admissible levels for $V^{-2+q/p}(\sl_2)$ are the integral levels $k=-2+q$ with $q\geq2$, i.e.\ the case $p=1$. Here, both \voa{}s $L_{-2+q}(\sl_2)$ and $L_\mathrm{Vir}(c_{q,3},0)$ are \strat{} and the correspondence in \autoref{prop:vermasingadm} and \autoref{cor:charid} (or \autoref{prop:lowestvermasingadm} and \autoref{cor:lowestcharid}) defines a bijection between the irreducible modules that respects the fusion rules.
\begin{introProp}[\ref{prop:fusionrings}]
Let $q\in\Z$ with $q\geq 2$ and $3\nmid q$. Then the correspondence $L_{-2+q}(\mu_{r,0})\mapsto L_\mathrm{Vir}(c_{q,3},h_{r+1,1}^{q,3})$ for $r\in\Z$ with $0\leq r\leq q-2$ defines an isomorphism of fusion rings of the modular tensor categories $\Rep(L_{-2+q}(\sl_2))$ and $\Rep(L_\mathrm{Vir}(c_{q,3},0))$.
\end{introProp}
The modular tensor categories $\Rep(L_{-2+q}(\sl_2))$ and $\Rep(L_\mathrm{Vir}(c_{q,3},0))$ are in general not equivalent, not even as fusion categories. However, we suspect that they are related by a certain Galois symmetry \cite{BG91,CG94}.
\begin{introConj}[\ref{conj:galois}]
For $q\in\Z$ with $q\geq 2$ and $3\nmid q$, the modular tensor categories $\Rep(L_{-2+q}(\sl_2))$ and $\Rep(L_\mathrm{Vir}(c_{q,3},0))$ are Galois conjugates.
\end{introConj}
The assertion is known to be true when $q$ is odd \cite{Gan03}. Moreover, for small values of $q$, the vector-valued characters of $L_{-2+q}(\sl_2)$ and $L_\mathrm{Vir}(c_{q,3},0)$ are related by certain Hecke operators \cite{HHW20,Wu20,DLS22}.

\medskip

\noindent\emph{Free Fermions.} We further specialise to the integral level $k=2$, i.e.\ the case $(q,p)=(4,1)$. Then both the simple affine \voa{} $L_2(\sl_2)$ and the simple Virasoro \voa{} $L_\mathrm{Vir}(c_{4,3},0)$, as well as their irreducible modules with lowest $L_0$-weight $1/2$, can be embedded into the \svoa{} of three and one free fermion, respectively:
\begin{equation*}
F^{3/2}\cong L_2(\mu_{0,0})\oplus L_2(\mu_{2,0}),\quad F^{1/2}\cong L_\mathrm{Vir}(c_{4,3},h_{1,1}^{4,3})\oplus L_\mathrm{Vir}(c_{4,3},h_{3,1}^{4,3}).
\end{equation*}
The character identities in \autoref{cor:charid} and \autoref{cor:lowestcharid} then naturally extend to the character identities (see \autoref{prop:charidFF})
\begin{equation*}
\ch^*_{F^{3/2}}(q^{\mp1/2},q^3)=\ch^*_{F^{1/2}}(q),
\end{equation*}
which imply the existence (nonuniquely) of graded vector-space isomorphisms
\begin{equation*}
\chi^\pm\colon F^{3/2}\overset{\sim}{\longrightarrow}F^{1/2}
\end{equation*}
that map a homogeneous vector of $h_0$-weight $f$ and $L_0$-weight $n$ to a vector of $L_0$-weight $n'=\mp f/2+3n$.

Because $\smash{F^{3/2}=\langle\Psi^{(0)},\Psi^+,\Psi^-\rangle}$ and $\smash{F^{1/2}=\langle\Phi\rangle}$ are free-field \svoa{}s with three and one strong generator, respectively, we are able to explicitly construct the graded vector-space isomorphisms $\chi^\pm$, in contrast to our general remark above (see \autoref{rem:inducedGVSI}). They are defined by mapping
\begin{equation}\label{eq:intro2}
\Psi^+_{-i}\mapsto\Phi_{-3i+1\pm1},\quad\Psi^-_{-i}\mapsto\Phi_{-3i+1\mp1},\quad\Phi^{(0)}_{-i}\mapsto\Phi_{-3i+1}
\end{equation}
for $i\in\Ns$ and $\vac\mapsto\vac$, similarly to \eqref{eq:intro1}. See also \autoref{rem:twisted} for an interpretation in terms of twisted modules for permutation automorphisms. The map \eqref{eq:intro2} also endows the \svoa{}s $F^{3/2}$ and $F^{1/2}$ with a new type of duality that was first observed in \cite{AW23} in a different context (see \autoref{rem:AWduality}):
\begin{introProp}[\ref{prop:FFcorrespondence}]
The \svoa{} $F^{1/2}$ is an irreducible $F^{3/2}$-module and \eqref{eq:intro2} is an isomorphism of $F^{3/2}$-modules $\smash{\chi^\pm\colon F^{3/2}\overset{\sim}{\longrightarrow}F^{1/2}}$.

Conversely, the \svoa{} $F^{3/2}$ is an irreducible $F^{1/2}$-module and \eqref{eq:intro2} is an isomorphism of $F^{1/2}$-modules $\smash{(\chi^\pm)^{-1}\colon F^{1/2}\overset{\sim}{\longrightarrow}F^{3/2}}$.
\end{introProp}
By restriction to the even parts of these \svoa{}s, we then obtain analogous isomorphisms of $L_2(\sl_2)$-modules $\smash{\chi^\pm\colon L_2(\sl_2)\overset{\sim}{\longrightarrow}L_\mathrm{Vir}(c_{4,3},0)}$ and of $L_\mathrm{Vir}(c_{4,3},0)$-modules $\smash{(\chi^\pm)^{-1}\colon L_\mathrm{Vir}(c_{4,3},0)\overset{\sim}{\longrightarrow}L_2(\sl_2)}$, constituting again a duality in the sense of \cite{AW23}; see \autoref{cor:dual}.

\medskip

\noindent\emph{Schur Indices.} In the following, we discuss two applications of the graded vector-space isomorphisms in this text to vertex operator (super)algebras appearing in the context of the 4d/2d-correspondence \cite{BLLPRR15} in physics. The idea is that the character identities can be interpreted as identities of Schur indices of 4d $\mathcal{N}=2$ superconformal field theories. First, we consider the (nonadmissible) near-admissible case $q=1$ and then the (boundary) admissible case $q=2$.

Let $p\geq1$. The vector-space isomorphisms $\psi^\pm\colon V^{-2+1/p}(\mu_{r,0})\overset{\sim}{\longrightarrow}M_{r+1,1;3p}$ for $r\in\N$ from \autoref{prop:gvsi} extend to the infinite direct sums
\begin{equation*}
\psi^\pm\colon\mathcal{V}^{(p)}=\bigoplus_{r=0}^\infty(r+1)V^{-2+1/p}(\mu_{r,0})\overset{\sim}{\longrightarrow}\mathcal{A}^{(3p)}=\bigoplus_{r=0}^\infty(r+1)M_{r+1,1;3p}
\end{equation*}
of these irreducible modules (see \autoref{prop:gvsispecial}), which admit well-known structures of generalised \voa{}s; see \cite{Ada16,ACGY21} for $\mathcal{V}^{(p)}$ and \cite{FFT11,FT11,AM13} for the doublet $\mathcal{A}^{(p)}$. On the level of the characters:
\begin{introProp}[\ref{prop:charid2special}]
Let $p\in\Ns$. Then the following character identities hold:
\begin{equation*}
\ch^*_{\mathcal{V}^{(p)}}(q^{\mp1/2},q^3)=\ch^*_{\mathcal{A}^{(3p)}}(q).
\end{equation*}
Moreover, if $p\equiv2\pmod4$, then both sides are \svoa{}s and also the supercharacters satisfy
\begin{equation*}
\sch^*_{\mathcal{V}^{(p)}}(q^{\mp1/2},q^3)=\sch^*_{\mathcal{A}^{(3p)}}(q).
\end{equation*}
\end{introProp}
We remark that the \aia{}s $\mathcal{V}^{(p)}$ and $\mathcal{A}^{(p)}$ for the same $p$ are related by quantum Drinfeld-Sokolov reduction \cite{ACGY21}, but here we relate them for $p$ and $3p$.

For $p=2$ \cite{BN22}, the \svoa{} $\mathcal{V}^{(2)}$ is the simple quotient of the small $\mathcal{N}=4$ superconformal algebra of central charge $c=-9$ \cite{Kac98,Ada16}. In the context of the 4d/2d-correspondence \cite{BMR19}, it is the \svoa{} $\smash{\mathcal{V}^{(2)}\cong\mathbb{V}(\T_{\SU(2)})}$ for the 4d $\mathcal{N}=4$ supersymmetric Yang-Mills theory $\T_{\SU(2)}$ with gauge group $\SU(2)$ \cite{AKM23}. On the other hand, the doublet \svoa{} $\smash{\mathcal{A}^{(6)}\cong\mathbb{V}(\T_{(3,2)})}$ is the \svoa{} of a certain 4d $\mathcal{N}=2$ theory $\T_{(3,2)}$ obtained as marginal diagonal $\SU(2)$ gauging of three $D_3(\SU(2))\cong(A_1,A_3)$ Argyres-Douglas theories \cite{BN16,DVX15}.

The character identity in \autoref{prop:charid2special} then implies the identity \cite{BN22}
\begin{equation*}
\I_{\SU(2)}(q^{\mp1/2},q^3)=\I_{(3,2)}(q)
\end{equation*}
of the Schur indices of these 4d $\mathcal{N}=2$ superconformal field theories. This Schur index identity can be generalised considerably to further 4d $\mathcal{N}=2$ superconformal field theories, as shown in \cite{KLS21}. It would be interesting to also find an underlying vertex algebraic mechanism explaining this (see \autoref{prob:generalschur}). Moreover, one can ask if the correspondence between $\mathcal{V}^{(p)}$ and $\mathcal{A}^{(3p)}$ for $p>2$ has any physical interpretation (see \autoref{prob:VpA3p}).

\medskip

\noindent\emph{Boundary Admissible Levels.} Finally, we consider the case $q=2$. Let $p\in\Ns$ be odd. In \autoref{cor:charid} and \autoref{cor:lowestcharid} we established the character identities $\smash{\ch^*_{L_{-2+2/p}(\sl_2)}(q^{\mp1/2},q^3)=\ch^*_{L_\mathrm{Vir}(c_{2,3p},0)}(q)}$ between the simple \voa{}s. It is known (see, e.g., \cite{BR18}) that these \voa{}s correspond to 4d $\mathcal{N}=2$ Argyres-Douglas theories of type $(A_1,D_{2n+1})$ and $(A_1,A_{2n})$, respectively, i.e.\ $\smash{L_{-2+2/(2n+1)}(\sl_2)\cong\mathbb{V}(\T_{(A_1,D_{2n+1})})}$ and $\smash{L_\mathrm{Vir}(c_{2,2n+3},0)\cong\mathbb{V}(\T_{(A_1,A_{2n})})}$ for $n\in\Ns$. On the level of Schur indices, it then follows:
\begin{introCor}[\ref{cor:schur}]
For $n\in\Ns$, the following identity of Schur indices holds:
\begin{equation*}
\I_{(A_1,D_{2n+1})}(q^{\mp1/2},q^3)=\I_{(A_1,A_{6n})}(q).
\end{equation*}
\end{introCor}
We ask if there is a physical interpretation of this Schur index identity, similar to \cite{BN22,KLS21} (see \autoref{prob:argyresdouglas})?

\medskip

\noindent\emph{Classical Freeness.} The boundary admissible cases are also interesting from a mathematical point of view. First of all, we note that it was shown in \cite{KW17} that the characters of the irreducible modules $L_{-2+2/p}(\mu_{0,s})$ for $L_{-2+2/p}(\sl_2)$ satisfy a nice product formula involving the Jacobi theta function $\vartheta_{11}(w,q)$. With our character identities in \autoref{cor:charid} and \autoref{cor:lowestcharid}, it follows immediately that this is also the case for the Virasoro \voa{} $L_\mathrm{Vir}(c_{2,3p},0)$.

Furthermore, it is expected that the \voa{}s $L_{-2+2/p}(\sl_2)$ for odd $p\in\Ns$ are classically free in the sense of \cite{EH21,AEH23}. On the other hand, it was proved in \cite{EH21} that the \voa{}s $L_\mathrm{Vir}(c_{2,p},0)$ are classically free. We believe that (some choice of) our vector-space isomorphism between the \voa{}s $\smash{L_{-2+2/p}(\sl_2)}$ and $\smash{L_\mathrm{Vir}(c_{2,3p},0)}$ can be used to prove the classical freeness of these affine \voa{}s (see \autoref{rem:classicallyfree}).

%%%%%%%%%%%%%%%%%%%%%%%%%%%%%%%

\subsection*{Notation}

All vector spaces, Lie algebras, vertex (operator) algebras, etc.\ are over the base field $\C$.

Denote by $(a;q)_\infty=\prod_{n=0}^\infty(1-aq^n)$ the infinite $q$-Pochhammer symbol and by $\eta(q)=\smash{q^{1/24}(q;q)_\infty}$ the Dedekind eta function. We shall also be using the standard Jacobi theta functions
{\allowdisplaybreaks
\begin{align*}
\vartheta_{00}(w,q)&=\vartheta(w,q)=\sum_{n\in\Z}w^nq^{n^2/2}\\
&=(q;q)_\infty(-wq^{1/2};q)_\infty(-w^{-1}q^{1/2};q)_\infty,\\
\vartheta_{01}(w,q)&=\sum_{n\in\Z}(-1)^nw^nq^{n^2/2}\\
&=(q;q)_\infty(wq^{1/2};q)_\infty(w^{-1}q^{1/2};q)_\infty,\\
\vartheta_{11}(w,q)&=\i\sum_{n\in\Z}(-1)^nw^{n+1/2}q^{(n+1/2)^2/2}\\
&=\i w^{1/2}q^{1/8}(q;q)_\infty(wq;q)_\infty(w^{-1};q)_\infty,
\end{align*}
}%
as well as the theta functions
\begin{equation*}
\theta_{n,m}(w,q)\coloneqq\sum_{j\in n/m+\Z}q^{mj^2/2}w^{mj}
\end{equation*}
for $m,n\in\Z$ from \cite{KW88}. Throughout, we shall treat $w$ and $q$ as formal variables, ignoring questions of convergence.

%%%%%%%%%%%%%%%%%%%%%%%%%%%%%%%

\subsection*{Acknowledgements}

We are happy to thank Thomas Creutzig, Jeffrey Harvey, Reimundo Heluani, Victor Kac, Ching Hung Lam, Antun Milas, Brandon Rayhaun, Kaiwen Sun and Qing Wang for helpful comments and discussions. We also thank the referee for their helpful remarks.

Sven Möller acknowledges support from the German Research Foundation (DFG) through the Emmy Noether Programme and the CRC 1624 ``Higher Structures, Moduli Spaces and Integrability'', project nos. 460925688 and 506632645.

Dra\v{z}en Adamovi\'{c} is partially supported by the Croatian Science Foundation under the project IP-2022-10-9006 and by the project﻿ ``Implementation of
cutting-edge research and its application as part of the Scientific Center of Excellence for Quantum and Complex Systems, and Representations of Lie Algebras'', grant no. PK.1.1.10.0004, co-financed by the European Union through the European Regional Development Fund -- Competitiveness and Cohesion Programme 2021-2027.

Sven Möller thanks the Department of Mathematics at the University of Zagreb for its hospitality and support during his research stay, during which part of this work was carried out. Dra\v{z}en Adamovi\'{c} thanks the Department of Mathematics at the University of Hamburg for its hospitality during several recent visits.

%%%%%%%%%%%%%%%%%%%%%%%%%%%%%%%
%%%%%%%%%%%%%%%%%%%%%%%%%%%%%%%

\section{Vertex Operator Algebras and Their Representations}\label{sec:reps}

In this work, we shall compare the formal characters (i.e.\ graded dimensions) of modules of the simple affine \voa{}s $L_k(\sl_2)$ of level $k\in\C\setminus\{-2\}$ with those of modules of the simple Virasoro \voa{}s $L_\mathrm{Vir}(c,0)$ of central charge $c\in\C$.

For the latter, we look at ordinary modules~$M$ (with $L_0$ acting semisimply, $L_0$-eigenvalues bounded from below and finite-dimensional $L_0$-eigenspaces). For these, it makes sense to consider the usual formal $q$-characters $\ch^*_M(q)\coloneqq\tr_Mq^{L_0}$, here, for convenience, defined without the usual modular correction factor $q^{-c/24}$.

For the simple affine \voa{}s $L_k(\sl_2)$, we consider modules $M$ for which $L_0$ and $h_0$ (fixing a choice of $\sl_2$-triple $\{e,f,h\}$ in $\sl_2$) act semisimply on $M$ (and commute), with finite-dimensional simultaneous eigenspaces, and the $L_0$-eigenvalues again bounded from below. For these modules, we can consider the formal ``Jacobi'' character $\ch^*_M(w,q)\coloneqq\tr_Mq^{L_0}w^{h_0}$. As these modules are allowed to have infinite-dimensional $L_0$-eigenspaces, they are not necessarily ordinary modules anymore, but they are $\N$-gradable.\footnote{Recall that for \voa{}s, especially in the context of defining (strong) rationality, one typically considers weak, $\N$-gradable (also called $\N$-graded or admissible, which we avoid in order to avoid confusion) and ordinary modules; see, e.g., \cite{DLM97}.}

In \autoref{sec:relaxed}, we shall briefly also consider ($\N$-gradable or weak) modules for the universal vertex (operator) algebras $V^k(\sl_2)$ and $V_\mathrm{Vir}(c,0)$ that do not necessarily admit formal characters anymore in the above sense.

\subsection{Affine \VOA{}s for \texorpdfstring{$\sl_2$}{sl\_2}}\label{sec:affine}

We consider the universal affine \voa{} $V^k(\sl_2)$ of level $k\in\C\setminus\{-2\}$, which has central charge $c=3k/(k+2)$ and is neither rational nor $C_2$-cofinite. For a general introduction to the representation theory of affine \voa{}s, we refer the reader to \cite{LL04}. We denote by $L_k(\sl_2)$ the unique simple quotient of $V^k(\sl_2)$. Generically, $L_k(\sl_2)=V^k(\sl_2)$. However, $V^k(\sl_2)$ is not simple if and only if
\begin{equation*}
k=k_{q,p}\coloneqq-2+q/p\quad\text{with}\quad q,p\in\Z, q\geq2, p\geq1\text{ and }(p,q)=1
\end{equation*}
\cite{GK07}, in which case the central charge is $c=3-6p/q$. These levels $k=-2+q/p$ are called \emph{admissible levels}. Equivalently, $L_k(\sl_2)$ is an admissible module for $\smash{\hat\sl_2}$ \cite{KW88,KW89,KW08}. Indeed, for an \emph{admissible level} $k=-2+q/p$, the universal affine \voa{} $V^k(\sl_2)$ possesses a nontrivial singular vector $\raisebox{.5pt}{\textcircled{\raisebox{-.9pt}{2}}}$ (see below) of $h_0$-weight $2(q-1)$ and $L_0$-weight $p(q-1)$.

The admissible cases can be further split into the \emph{integral} cases, where $p=1$, and the remaining ones. In the integral case, meaning exactly that $k\in\N$, the simple affine \voa{} $L_k(\sl_2)$ is \strat{} and called \emph{Wess-Zumino-Witten model}. On the other hand, if $p>1$, or equivalently if $k=k_{q,p}\notin\N$, then $L_k(\sl_2)$ is neither rational nor $C_2$-cofinite.

Of the nonadmissible cases, we shall only study the \emph{near-admissible} cases of level $k=k_{1,p}\coloneqq-2+1/p$ for $p\in\Ns$, which still have some nice properties. As stated above, $\smash{L_{-2+1/p}(\sl_2)=V^{-2+1/p}(\sl_2)}$ is neither rational nor $C_2$-cofinite.

\medskip

In the following, we review the representation theory of $V^k(\sl_2)$ and $L_k(\sl_2)$, depending on the level $k$. To explain the notation, let $\Lambda_0$ and $\Lambda_1$ denote the fundamental weights of $\smash{\hat\sl_2}=\sl_2\otimes\C[t,t^{-1}]\oplus\C K$, so that any weight $\lambda$ for $\smash{\hat{\sl}_2}$ can be written as a complex linear combination of $\Lambda_0$ and $\Lambda_1$. As we fix the level $k$ of the $\smash{\hat{\sl}_2}$-modules, we can write $\lambda=(k-\mu)\Lambda_0+\mu\Lambda_1$ for some $\mu\in\C$. It is common to only use $\mu$ (often also called~$j$) rather than $\lambda$ to label (certain) modules of $\smash{\hat{\sl}_2}$ and $V^k(\sl_2)$.

In general, the representation theory of $V^k(\sl_2)$ and $L_k(\sl_2)$ is quite complicated; see, e.g., \cite{AM95,DLM97c,RW15,KR19}. Any module for $L_k(\sl_2)$ is naturally also a module for $V^k(\sl_2)$, but the converse is usually not the case if $V^k(\sl_2)\neq L_k(\sl_2)$, i.e.\ for an admissible level. We shall discuss the admissible case in more detail below, in particular which $V^k(\sl_2)$-modules survive as $L_k(\sl_2)$-modules.

A common way to obtain $V^k(\sl_2)$-modules is to start with an irreducible $\smash{\sl_2}$-module~$U$, consider it as a module for the Lie subalgebra $P=\smash{\smash{\hat{\sl}}_2^{\geq0}}=\sl_2\otimes\C[t]\oplus\C K$ of $\smash{\hat\sl_2}=\sl_2\otimes\C[t,t^{-1}]\oplus\C K$ such that $\sl_2\otimes t\C[t]$ acts trivially and $K$ as $k\id$ and induce it up to a (smooth) $\smash{\hat\sl_2}$-module $\smash{V^k(U)=U(\hat{\sl}_2)\otimes_{U(P)}U}$, which will then be a module for $V^k(\sl_2)$ as well. (There are also other types of induction to obtain $\smash{\hat\sl_2}$-modules, which we mention briefly in the following as well.) $V^k(U)$ may or may not be irreducible, but we can form the irreducible quotient $L_k(U)$ of $V^k(U)$. Restricting to weight modules of $\sl_2$ \cite{Fer90,Mat00} (see also \cite{Maz10}), we obtain the following, denoting by $\Lambda_1$ the fundamental weight of $\sl_2$.
\begin{enumerate}[wide]
\item\label{item:1} Finite-dimensional modules: Suppose $U$ is the $(r+1)$-dimensional irreducible module of $\sl_2$ for $r\in\N$. Then $U$ has highest weight~$\mu\Lambda_1=r\Lambda_1$ (and lowest weight~$-r\Lambda_1$). We denote the induced modules by $V^k(r)$. They are sometimes called \emph{Weyl modules}. These are ordinary $V^k(\sl_2)$-modules, i.e.\ they have finite-dimensional $L_0$-weight spaces in particular. For $r=0$, $V^k(0)=V^k(\sl_2)$ is the vacuum module, i.e.\ recovers the universal affine \voa{} itself.

We can form the irreducible quotients, which we denote by $L_k(r)$. The modules $L_k(r)$ for $r\in\N$ exhaust all irreducible ordinary $V^k(\sl_2)$-modules up to isomorphism. $L_k(0)=L_k(\sl_2)$ is the simple affine \voa{} itself.

\item Highest-weight modules: Suppose that $U$ is an infinite-dimensional highest-weight Verma module with highest weight $\mu\Lambda_1$ for $\mu\notin\N$. We denote the induced $\smash{\hat\sl_2}$-modules by $M_k(\mu)$. They are usually called \emph{Verma modules} as well. They can also be constructed by inducing up a one-dimensional module $\C_\mu$ of the standard Borel subalgebra of $\smash{\hat\sl_2}$, on which $h_0$ acts as $\mu$. Using the latter construction, we obtain Verma modules $M_k(\mu)$ for all $\mu\in\C$. Again, we can consider the irreducible quotients, which we denote by $L_k(\mu)$.

The Weyl module $V^k(r)$ for $r\in\N$ considered in \eqref{item:1} is a proper quotient of $M_k(r)$, by a submodule generated by one singular vector \raisebox{.5pt}{\textcircled{\raisebox{-.9pt}{1}}}, which we describe below. In particular, the irreducible quotient $L_k(r)$ considered in \eqref{item:1} coincides with the one here, justifying our notation.

The Verma modules $M_k(\mu)$ are not ordinary modules of $V^k(\sl_2)$ but they are $\N$-gradable, and neither are their irreducible quotients (unless $\mu\in\N$ in the admissible case; see below).

\item Lowest-weight modules: Suppose that $U$ is an infinite-dimensional lowest-weight Verma module with lowest weight $-\mu\Lambda_1$ for $\mu\notin\N$. Then we can consider the induced $\smash{\hat\sl_2}$-module $M^-_k(-\mu)$, which we call \emph{lowest-weight Verma module} for $V^k(\sl_2)$. It can also be constructed by inducing up a one-dimensional module of the Borel subalgebra of $\smash{\hat\sl_2}$ that is obtained by applying the Weyl reflection to the standard Borel subalgebra. Again, we can also consider the irreducible quotients $L^-_k(-\mu)$. As the Verma modules above, these modules (and their irreducible quotients unless $\mu\in\N$ in the admissible case) are not ordinary, but only $\N$-gradable $V^k(\sl_2)$-modules.

The Weyl module $V^k(r)$ for $r\in\N$ considered in \eqref{item:1} is a proper quotient of both $M_k(r)$ and $M^-_k(-r)$. Hence, the irreducible quotients satisfy $L_k(r)=L^-_k(-r)$.

\item Dense modules: further infinite-dimensional weight modules that are neither highest nor lowest weight, called dense modules. The induced modules for $V^k(\sl_2)$ (or their irreducible quotients) are called \emph{relaxed highest-weight modules}. In this work, we shall only treat them briefly in \autoref{sec:relaxed}.
\end{enumerate}

For later reference, we note that the highest-weight vector of the Verma module $M_k(\mu)$ for $\mu\in\C$ has $(h_0,L_0)$-weight $(\mu,\mu(\mu+2)/(4(k+2)))$ and the character is
{\allowdisplaybreaks
\begin{align*}
&\ch^*_{M_k(\mu)}(w,q)=\tr_{M_k(\mu)}q^{L_0}w^{h_0}\\
&=\frac{w^\mu}{1-w^{-2}}\frac{q^{\mu(\mu+2)/(4(k+2))}}{\prod_{n=1}^\infty(1-w^2q^n)(1-q^n)(1-w^{-2}q^n)}\\
&=\frac{w^\mu q^{\mu(\mu+2)/(4(k+2))}}{\prod_{n=1}^\infty(1-w^2q^n)(1-q^n)(1-w^{-2}q^{n-1})}\\
&=\frac{w^\mu q^{\mu(\mu+2)/(4(k+2))}}{(w^2q;q)_\infty(q;q)_\infty(w^{-2};q)_\infty}.
\end{align*}
}%
Similarly, the lowest-weight vector of the Verma module $M^-_k(-\mu)$ for $\mu\in\C$ has $(h_0,L_0)$-weight $(-\mu,\mu(\mu+2)/(4(k+2)))$ so that the character is
\begin{equation*}
\ch^*_{M^-_k(-\mu)}(w,q)=\frac{w^{-\mu}q^{\mu(\mu+2)/(4(k+2))}}{(w^2;q)_\infty(q;q)_\infty(w^{-2}q;q)_\infty}.
\end{equation*}
We note that $M^-_k(-\mu)$ can be obtained from $M_k(\mu)$ by composing the action of $V^k(\sl_2)$ on $M_k(\mu)$ with the inner automorphism of $V^k(\sl_2)$ lifted from the inner automorphism $e\mapsto f$, $f\mapsto e$, $h\mapsto -h$ of $\sl_2$. On the level of the characters, this amounts to replacing $w\mapsto w^{-1}$ and $q\mapsto q$.

%%%%%%%%%%%%%%%%%%%%%%%%%%%%%%%

\subsubsection{Admissible Case}

We now consider the simple affine \voa{}s $L_{-2+q/p}(\sl_2)$ of admissible levels $k=-2+q/p$ with $q,p\in\Z$, $q\geq2$, $p\geq1$ and $(p,q)=1$, i.e.\ when $V^k(\sl_2)\neq L_k(\sl_2)$. The central charge is $c=3-6p/q$.

Of the irreducible $V^k(\sl_2)$-modules discussed above, the ones that are $L_k(\sl_2)$-modules are the following \cite{AM95,RW15}, ignoring from now on the relaxed highest-weight modules, which we will return to briefly in \autoref{sec:relaxed}.

\begin{enumerate}
\item\label{item:adm1} The irreducible quotients $L_{-2+q/p}(\mu_{r,0})$ of the Weyl modules, both highest- and lowest-weight modules,
\item\label{item:adm2} The highest-weight modules $L_{-2+q/p}(\mu_{r,s})$,
\item\label{item:adm3} The lowest-weight modules $L^-_{-2+q/p}(-\mu_{r,s})$
\end{enumerate}
for $r,s\in\Z$ with $0\leq r\leq q-2$ and $1\leq s\leq p-1$ where $\mu_{r,s}=r-qs/p$ or equivalently the full $\smash{\hat\sl_2}$-weight
\begin{equation*}
\lambda_{r,s}=\lambda_r^\text{int}-\frac{q}{p}\lambda_s^\text{fr}
\end{equation*}
with
\begin{equation*}
\lambda_r^\text{int}=(q-r-2)\Lambda_0+r\Lambda_1,\quad\lambda_s^\text{fr}=(p-s-1)\Lambda_0+s\Lambda_1,
\end{equation*}
recalling that $\Lambda_0$ and $\Lambda_1$ are the fundamental weights of $\smash{\hat\sl_2}$.

The irreducible highest-weight modules $L_{-2+q/p}(\mu_{r,s})$ in \eqref{item:adm1} and \eqref{item:adm2} are exactly the irreducible admissible (or equivalently modular invariant) modules of $\smash{\hat{\sl}_2}$, which are considered in \cite{KW89}. They are indexed by the highest weights $\lambda_{r,s}$ with $0\leq r\leq q-2$ and $0\leq s\leq p-1$. We mostly focus on these modules in the following, noting that the lowest-weight modules in \eqref{item:adm3} can be described completely analogously.

The $L_0$-grading of these modules is bounded from below. Indeed, the lowest $L_0$-eigenvalue of $L_{-2+q/p}(\mu_{r,s})$ is
\begin{equation*}
\ell_{r,s}^{q,p}\coloneqq\frac{(p(r+1)-qs)^2-p^2}{4pq}.
\end{equation*}
However, the $L_0$-eigenspaces are usually not finite-dimensional, meaning that these irreducible modules are in general not ordinary modules (but they are all $\N$-gradable). As we explained above, $L_{-2+q/p}(\mu_{r,s})$ is an ordinary $L_{-2+q/p}(\sl_2)$-module if and only if $s=0$, i.e.\ if $\mu_{r,s}$, the coefficient of $\Lambda_1$ in $\lambda_{r,s}$, is an integer. As a special case, for $r=s=0$, i.e.\ $\mu_{r,s}=0$, we recover the vacuum module $L_{-2+q/p}(0)=L_{-2+q/p}(\sl_2)$.

\medskip

The irreducible module $L_{-2+q/p}(\mu_{r,s})$ can be obtained from the highest-weight Verma module $M_{-2+q/p}(\mu_{r,s})$ by taking the quotient by the maximal proper submodule of $M_{-2+q/p}(\mu_{r,s})$, which is generated by two singular vectors of $(h_0,L_0)$-weights $(r-qs/p,\ell_{r,s}^{q,p})$ plus
\begin{equation*}
\raisebox{.5pt}{\textcircled{\raisebox{-.9pt}{1}}}~(-2(r+1),s(r+1))\quad\text{and}\quad\raisebox{.5pt}{\textcircled{\raisebox{-.9pt}{2}}}~(2(q-r-1),(p-s)(q-r-1)).
\end{equation*}
For the ordinary modules $L_{-2+q/p}(\mu_{r,0})=L_{-2+q/p}(r)$, i.e.\ for $s=0$, this becomes $(r,\ell_{r,0}^{q,p})$ plus
\begin{equation*}
\raisebox{.5pt}{\textcircled{\raisebox{-.9pt}{1}}}~(-2(r+1),0)\quad\text{and}\quad\raisebox{.5pt}{\textcircled{\raisebox{-.9pt}{2}}}~(2(q-r-1),p(q-r-1)).
\end{equation*}
Here, taking the quotient by the submodule in $\smash{M_{-2+q/p}(\mu_{r,0})=M_{-2+q/p}(r)}$ generated by the vector~\raisebox{.5pt}{\textcircled{\raisebox{-.9pt}{1}}} yields the Weyl module $\smash{V^{-2+q/p}(r)}$ and taking the quotient of this again by \raisebox{.5pt}{\textcircled{\raisebox{-.9pt}{2}}} yields $\smash{L_{-2+q/p}(\mu_{r,0})=L_{-2+q/p}(r)}$.

In particular, for $r=0$ we recover the statement from above about the singular vector~\raisebox{.5pt}{\textcircled{\raisebox{-.9pt}{2}}} of weight $(2(q-1),p(q-1))$ in the universal affine \voa{} $V^{-2+q/p}(\sl_2)$ that yields $L_{-2+q/p}(\sl_2)$ upon taking the quotient by the submodule it generates.

\medskip

We briefly comment on the lowest-weight modules $\smash{L^-_{-2+q/p}(-\mu_{r,s})}$ as well. They are obtained from the lowest-weight Verma modules $\smash{M^-_{-2+q/p}(-\mu_{r,s})}$ by taking the quotient by two singular vectors of $(h_0,L_0)$-weights $(-r+qs/p,\ell_{r,s}^{q,p})$ plus
\begin{equation*}
\raisebox{.5pt}{\textcircled{\raisebox{-.9pt}{1}}}~(2(r+1),s(r+1))\quad\text{and}\quad\raisebox{.5pt}{\textcircled{\raisebox{-.9pt}{2}}}~(-2(q-r-1),(p-s)(q-r-1)).
\end{equation*}
For the ordinary modules $\smash{L^-_{-2+q/p}(-\mu_{r,0})=L_{-2+q/p}(\mu_{r,0})}$ this is $(-r,\ell_{r,0}^{q,p})$ plus
\begin{equation*}
\raisebox{.5pt}{\textcircled{\raisebox{-.9pt}{1}}}~(2(r+1),0)\quad\text{and}\quad\raisebox{.5pt}{\textcircled{\raisebox{-.9pt}{2}}}~(-2(q-r-1),p(q-r-1)).
\end{equation*}
Taking the quotient by the submodule in $\smash{M^-_{-2+q/p}(-\mu_{r,0})}$ generated by the singular vector~\raisebox{.5pt}{\textcircled{\raisebox{-.9pt}{1}}} yields the Weyl module $V^{-2+q/p}(r)$ and taking the quotient again by \raisebox{.5pt}{\textcircled{\raisebox{-.9pt}{2}}} yields $\smash{L^-_{-2+q/p}(-\mu_{r,0})=L_{-2+q/p}(\mu_{r,0})}$.

\medskip

We summarise the various $L_{-2+q/p}(\sl_2)$-modules in \autoref{fig:admrep}. Here, single arrows correspond to taking the quotient by a single singular vector, double arrows to taking the quotient by two.
\begin{figure}[ht]
\begin{tikzcd}[column sep=-2pt]
\{M_{-2+q/p}(\mu_{r,s})\}\arrow[Rightarrow]{dd}{\raisebox{.5pt}{\textcircled{\raisebox{-.9pt}{1}}},\,\raisebox{.5pt}{\textcircled{\raisebox{-.9pt}{2}}}} &\makebox[0pt]{$\supseteq$}& \{M_{-2+q/p}(\mu_{r,0})\}\arrow[shorten >= 10pt]{dr}{\raisebox{.5pt}{\textcircled{\raisebox{-.9pt}{1}}}} && \{M^-_{-2+q/p}(-\mu_{r,0})\}\arrow[shorten >= 10pt]{dl}[swap]{\raisebox{.5pt}{\textcircled{\raisebox{-.9pt}{1}}}} &\makebox[0pt]{$\subseteq$}& \{M^-_{-2+q/p}(-\mu_{r,s})\}\arrow[Rightarrow]{dd}{\raisebox{.5pt}{\textcircled{\raisebox{-.9pt}{1}}},\,\raisebox{.5pt}{\textcircled{\raisebox{-.9pt}{2}}}}\\
&&&\makebox[0pt]{$\{V^{-2+q/p}(r)\}$}\arrow{d}{\raisebox{.5pt}{\textcircled{\raisebox{-.9pt}{2}}}}\\
\{L_{-2+q/p}(\mu_{r,s})\} &\makebox[0pt]{$\supseteq$}&& \makebox[0pt]{$\hphantom{-}\{L_{-2+q/p}(\mu_{r,0})=L^-_{-2+q/p}(-\mu_{r,0})\}$} &&\makebox[0pt]{$\subseteq$}& \{L^-_{-2+q/p}(-\mu_{r,s})\}
\end{tikzcd}
\caption{Representations of admissible level $L_{-2+q/p}(\sl_2)$.}
\label{fig:admrep}
\end{figure}

\medskip

The formal characters of the irreducible highest-weight modules $L_{-2+q/p}(\mu_{r,s})$ for $L_{-2+q/p}(\sl_2)$ are of the form \cite{KW88}
{\allowdisplaybreaks
\begin{align*}
&\ch^*_{L_{-2+q/p}(\mu_{r,s})}(w,q)=\tr_{L_{-2+q/p}(\mu_{r,s})}q^{L_0}w^{h_0}\\
&=q^{1/8-p/4q}\frac{\theta_{b_+,2a}(w^{1/p},q)-\theta_{b_-,2a}(w^{1/p},q)}{\theta_{1,4}(w,q)-\theta_{-1,4}(w,q)}\\
&=q^{1/8-p/4q}\frac{\theta_{b_+,2a}(w^{1/p},q)-\theta_{b_-,2a}(w^{1/p},q)}{wq^{1/8}\prod_{n=1}^\infty(1-w^2q^n)(1-q^n)(1-w^{-2}q^{n-1})}\\
&=q^{1/8-p/4q}\frac{\theta_{b_+,2a}(w^{1/p},q)-\theta_{b_-,2a}(w^{1/p},q)}{wq^{1/8}(w^2q;q)_\infty(q;q)_\infty(w^{-2};q)_\infty}
\end{align*}
}%
with $a=pq$ and $b_\pm=\pm p(r+1)-qs$. On the other hand, the characters of the irreducible lowest-weight modules $\smash{L^-_{-2+q/p}(-\mu_{r,s})}$ for $\smash{L_{-2+q/p}(\sl_2)}$ are
\begin{align*}
\ch^*_{L^-_{-2+q/p}(-\mu_{r,s})}(w,q)&=q^{1/8-p/4q}\frac{\theta_{-b_+,2a}(w^{1/p},q)-\theta_{-b_-,2a}(w^{1/p},q)}{\theta_{-1,4}(w,q)-\theta_{1,4}(w,q)}
\end{align*}
with $a=pq$ and $b_\pm=\pm p(r+1)-qs$. We note that for $s=0$, this character coincides with that of $L_{-2+q/p}(\mu_{r,0})$, as it should.

\subsubsection*{Boundary Admissible Case \texorpdfstring{$q=2$}{q=2}}

The characters of the irreducible $L_{-2+q/p}(\sl_2)$-modules simplify in the \emph{boundary admissible} case $k=-2+2/p$, i.e.\ for $q=2$, the smallest possible admissible value of~$q$. Then $p\in\Ns$ must be odd. The central charge is $c=3-3p$. In that case, the $p$ irreducible modules are $L_{-2+2/p}(\mu_{0,s})$ for $0\leq s\leq p-1$. Their characters take the simpler form \cite{KW17}
\begin{align*}
\ch^*_{L_{-2+2/p}(\mu_{0,s})}(w,q)&=q^{(1-p)/8}w^{-2s/p}q^{s^2/2p}\frac{\vartheta_{11}(w^{-2}q^s,q^p)}{\vartheta_{11}(w^{-2},q)}\\
&=q^{(1-p)/8}w^{-2s/p}q^{s^2/2p}\frac{\vartheta_{11}(w^2q^{-s},q^p)}{\vartheta_{11}(w^2,q)}.
\end{align*}
The factor of $q^{c/24}=q^{(1-p)/8}$ is present compared to \cite{KW17} because we do not include the factor $q^{-c/24}$ to make the characters modular invariant. Similarly, the characters of the lowest-weight modules are
\begin{align*}
\ch^*_{L^-_{-2+2/p}(-\mu_{0,s})}(w,q)&=q^{(1-p)/8}w^{2s/p}q^{s^2/2p}\frac{\vartheta_{11}(w^2q^s,q^p)}{\vartheta_{11}(w^2,q)}\\
&=q^{(1-p)/8}w^{2s/p}q^{s^2/2p}\frac{\vartheta_{11}(w^{-2}q^{-s},q^p)}{\vartheta_{11}(w^{-2},q)}
\end{align*}
for $0\leq s\leq p-1$.

\subsubsection*{Integral Case \texorpdfstring{$p=1$}{p=1}}

We discuss the special (admissible) case of $p=1$, referred to as the integral case or the Wess-Zumino-Witten models. Here, $L_{-2+q}(\sl_2)$ for $q\geq2$ has level $k=-2+q$ and central charge $c=3-6/q$ and is \strat{}. This implies that all $\N$-gradable (in fact, all weak) $L_{-2+q}(\sl_2)$-modules are completely reducible with the only irreducible modules being ordinary modules. As we explained above, the $q-1$ many irreducible (ordinary) modules are exactly the $L_{-2+q}(\mu_{r,0})=L_{-2+q}(r)$ for $0\leq r\leq q-2$ with lowest $L_0$-weights
\begin{equation*}
\ell_{r,0}^{q,1}=\frac{r(r+2)}{4q}.
\end{equation*}
The representation category $\Rep(L_{-2+q}(\sl_2))$ is a modular tensor category, whose fusion rules are
\begin{equation*}
L_{-2+q}(\mu_{r_1,0})\boxtimes L_{-2+q}(\mu_{r_2,0})\cong\bigoplus_{\substack{r_3=|r_1-r_2|\\r_3\in|r_1-r_2|+2\Z}}^{\min(r_1+r_2,-4+2q-r_1-r_2)}L_{-2+q}(\mu_{r_3,0})
\end{equation*}
for $0\leq r_1\leq q-2$ and $0\leq r_2\leq q-2$.

%%%%%%%%%%%%%%%%%%%%%%%%%%%%%%%

\subsubsection{Near-Admissible Case \texorpdfstring{$q=1$}{q=1}}

Of the nonadmissible affine \voa{}s $V^k(\sl_2)=L_k(\sl_2)$, we only consider the case of $k=-2+1/p$ for $p\in\Ns$, which we call near-admissible levels. By abuse of notation, we can think of this as setting $q=1$ in the admissible levels $k_{q,p}=-2+q/p$.

We described all the irreducible weight modules for $V^{-2+1/p}(\sl_2)=L_{-2+1/p}(\sl_2)$ above. Ignoring (for now) again the relaxed highest-weight modules, these are:
\begin{enumerate}
\item The irreducible quotients $L_{-2+1/p}(r)$ of the Weyl modules for $r\in\N$, both highest- and lowest-weight modules,
\item The highest-weight modules $L_{-2+1/p}(\mu)$ for $\mu\notin\N$,
\item The lowest-weight modules $L^-_{-2+1/p}(-\mu)$ for $\mu\notin\N$.
\end{enumerate}

All the Weyl modules are irreducible, i.e.\ $V^{-2+1/p}(r)=L_{-2+1/p}(r)$ for $r\in\N$ \cite{Cre17}. For $r=0$, we recover the vacuum module $L_{-2+1/p}(0)=L_{-2+1/p}(\sl_2)$.

By extending the notation from the admissible case, these ordinary modules have the highest weight $\mu_{r,0}=r$ (or $\lambda_{r,0}=(-2+q/p-r)\Lambda_0+r\Lambda_1$) and smallest $L_0$-eigenvalue
\begin{equation*}
\ell_{r,0}^{1,p}=\frac{pr(r+2)}{4}=\frac{r(r+2)}{4(k+2)}
\end{equation*}
for $r\in\N$. They can be obtained from the Verma modules $M_{-2+1/p}(r)$ by taking the quotient by a singular vector of $(h_0,L_0)$-weight
\begin{equation*}
\raisebox{.5pt}{\textcircled{\raisebox{-.9pt}{1}}}~(r,\ell_{r,0}^{1,p})+(-2(r+1),0)=(-r-2,\ell_{r,0}^{1,p}).
\end{equation*}
This formally corresponds to the first singular vector~$\raisebox{.5pt}{\textcircled{\raisebox{-.9pt}{1}}}$ in the Verma module of highest weight $\mu_{r,s}$ (with $s=0$) in the admissible case discussed above.

On the other hand, the irreducible Weyl modules $\smash{V^{-2+1/p}(r)=L^-_{-2+1/p}(-r)}$ have lowest weight $-\mu_{r,0}=-r$ for $r\in\N$. They can also be obtained from the lowest-weight Verma modules $\smash{M^-_{-2+1/p}(-r)}$ by taking the quotient by the submodule generated by a singular vector of $(h_0,L_0)$-weight
\begin{equation*}
\raisebox{.5pt}{\textcircled{\raisebox{-.9pt}{1}}}~(-r,\ell_{r,0}^{1,p})+(2(r+1),0)=(r+2,\ell_{r,0}^{1,p}).
\end{equation*}

The Verma modules $M_{-2+1/p}(\mu)$ for generic $\mu\notin\N$ are already irreducible, i.e.\ $M_{-2+1/p}(\mu)=L_{-2+1/p}(\mu)$. The lowest-weight Verma modules $\smash{M^-_{-2+1/p}(-\mu)}$ for generic $\mu\notin\N$ are also irreducible, i.e.\ $\smash{M^-_{-2+1/p}(-\mu)=L^-_{-2+1/p}(-\mu)}$.

We summarise the various $V^{-2+1/p}(\sl_2)$-modules in \autoref{fig:nearadmrep}. Again, single arrows correspond to taking the quotient by one singular vector.
\begin{figure}[ht]
\begin{tikzcd}[column sep=-2pt]
\mu\notin\N{:}&&\makebox[0pt]{$r\in\N$:}&&\mu\notin\N{:}\\[-10pt]
&\makebox[0pt]{$M_{-2+1/p}(\mu_{r,0})$}\arrow[shorten <= 5pt]{dr}[swap]{\raisebox{.5pt}{\textcircled{\raisebox{-.9pt}{1}}}}&&\makebox[0pt]{$M^-_{-2+1/p}(-\mu_{r,0})$}\arrow[shorten <= 5pt]{dl}{\raisebox{.5pt}{\textcircled{\raisebox{-.9pt}{1}}}}\\
M_{-2+1/p}(\mu)=L_{-2+1/p}(\mu)&&\begin{array}{c}V^{-2+1/p}(r)\\=L_{-2+1/p}(r)\\=L^-_{-2+1/p}(-r)\end{array}&&M^-_{-2+1/p}(-\mu)=L^-_{-2+1/p}(-\mu)
\end{tikzcd}
\caption{Representations of near-admissible level $L_{-2+1/p}(\sl_2)$.}
\label{fig:nearadmrep}
\end{figure}

The characters of the irreducible Verma modules $M_{-2+1/p}(\mu)=L_{-2+1/p}(\mu)$ are
\begin{align*}
&\ch^*_{M_{-2+1/p}(\mu)}(w,q)\\
&=\frac{w^\mu}{1-w^{-2}}\frac{q^{\ell_{\mu,0}^{1,p}}}{\prod_{n=1}^\infty(1-w^2q^n)(1-q^n)(1-w^{-2}q^n)}\\
&=\frac{w^\mu q^{\ell_{\mu,0}^{1,p}}}{(w^2q;q)_\infty(q;q)_\infty(w^{-2};q)_\infty}
\end{align*}
for $\mu\in\C\setminus\N$. Furthermore, the characters of the irreducible Weyl modules $\smash{V^{-2+1/p}(r)=L_{-2+1/p}(r)=L^-_{-2+1/p}(-r)}$ for $r\in\N$ are
\begin{align*}
&\ch^*_{V^{-2+1/p}(r)}(w,q)\\
&=\frac{w^{r+1}-w^{-(r+1)}}{w-w^{-1}}\frac{q^{\ell_{r,0}^{1,p}}}{\prod_{n=1}^\infty(1-w^2q^n)(1-q^n)(1-w^{-2}q^n)}\\
&=(w^r-w^{-(r+2)})\frac{q^{\ell_{r,0}^{1,p}}}{(w^2q;q)_\infty(q;q)_\infty(w^{-2};q)_\infty}.
\end{align*}
For the lowest-weight Verma modules $\smash{M^-_{-2+1/p}(-\mu)=L^-_{-2+1/p}(-\mu)}$ they are
\begin{equation*}
\ch^*_{M^-_{-2+1/p}(-\mu)}(w,q)=\frac{w^{-\mu}q^{\ell_{\mu,0}^{1,p}}}{(w^2;q)_\infty(q;q)_\infty(w^{-2}q;q)_\infty}
\end{equation*}
for $\mu\in\C\setminus\N$.

%%%%%%%%%%%%%%%%%%%%%%%%%%%%%%%

\subsection{Virasoro \VOA{}s}\label{sec:virasoro}

We consider the universal Virasoro \voa{} $V_\mathrm{Vir}(c,0)$ of central charge $c\in\C$, which is neither rational nor $C_2$-cofinite. For a more thorough introduction, we refer the reader to, e.g., \cite{LL04,DFMS97}. Let $L_\mathrm{Vir}(c,0)$ denote the unique simple quotient. Generically, $V_\mathrm{Vir}(c,0)=L_\mathrm{Vir}(c,0)$. One can show that $V_\mathrm{Vir}(c,0)$ is not simple if and only if
\begin{equation*}
c=c_{q,p}\coloneqq 1-6(p-q)^2/(pq)\quad\text{with}\quad p,q\in\Z, p,q\geq2\text{ and }(p,q)=1.
\end{equation*}
Indeed, in that case, there is a nontrivial singular vector~\raisebox{.5pt}{\textcircled{\raisebox{-.9pt}{2}}} in $V_\mathrm{Vir}(c_{q,p},0)$ with $L_0$-weight $(p-1)(q-1)$. The simple Virasoro \voa{}s $L_\mathrm{Vir}(c_{q,p},0)$ for $p$ and $q$ as above are called \emph{minimal models}. They are \strat{}.

We also consider the case of ``$q=1$'' with central charge $c_{1,p}\coloneqq 13-6p-6/p$ for $p\in\Z_{\geq2}$. Then $L_\mathrm{Vir}(c_{1,p},0)=V_\mathrm{Vir}(c_{1,p},0)$ is called \emph{logarithmic $(1,p)$-minimal model}. As stated above, it is neither rational nor $C_2$-cofinite.

\medskip

Let $M_\mathrm{Vir}(c,h)\cong U(\mathcal{L}_{<0})$ be the Verma module of highest $L_0$-weight $h\in\C$ for the Virasoro algebra $\smash{\mathcal{L}=\bigoplus_{s\in\Z}\C L_s\oplus\C C}$. It is constructed from the $1$-dimensional module $\C_{c,h}$ for the subalgebra $\smash{\mathcal{L}_{\geq0}=\bigoplus_{s\in\N}\C L_s\oplus\C C}$ on which $L_0$ acts by $h$, $C$ by $c$ and $\smash{\mathcal{L}_{>0}=\bigoplus_{s\in\Ns}\C L_s}$ trivially and then forming the induced $\mathcal{L}$-module $\smash{M_\mathrm{Vir}(c,h)=U(\mathcal{L})\otimes_{U(\mathcal{L}_{\geq0})}\C_{c,h}}$.

The universal Virasoro \voa{} $V_\mathrm{Vir}(c,0)\cong U(\mathcal{L}_{\leq2})$ is formed from $M_\mathrm{Vir}(c,0)$ by taking the quotient $V_\mathrm{Vir}(c,0)=M_\mathrm{Vir}(c,0)/U(\mathcal{L}_{<0})L_{-1}$ by the right ideal $U(\mathcal{L}_{<0})L_{-1}$ of $U(\mathcal{L}_{<0})$, corresponding to the singular vector \raisebox{.5pt}{\textcircled{\raisebox{-.9pt}{1}}}, here $L_{-1}\vac$, of $L_0$-weight~$1$. Then $V_\mathrm{Vir}(c,0)$ can be equipped with the structure of a \voa{}, which may or may not be simple.

Let $L_\mathrm{Vir}(c,h)$ denote the quotient of $M_\mathrm{Vir}(c,h)$ by the (unique) maximal proper $\mathcal{L}$-submodule. Then $L_\mathrm{Vir}(c,h)$ naturally carries the structure of an irreducible module for the \voa{} $V_\mathrm{Vir}(c,0)$. Moreover, for $h\in\C$ these exhaust all irreducible $V_\mathrm{Vir}(c,0)$-modules up to equivalence.

Hence, if $c\neq c_{q,p}$ for $p,q\in\Z_{\geq2}$, this gives a classification of the irreducible modules of the simple Virasoro \voa{} $L_\mathrm{Vir}(c,0)=V_\mathrm{Vir}(c,0)$. As these are infinitely many, $L_\mathrm{Vir}(c,0)=V_\mathrm{Vir}(c,0)$ is neither rational nor $C_2$-cofinite.

If $c=c_{q,p}$, then not all the irreducible $V_\mathrm{Vir}(c_{q,p},0)$-modules $L_\mathrm{Vir}(c_{q,p},h)$ are (irreducible) modules for $L_\mathrm{Vir}(c_{q,p},0)$, but those (finitely many) that are, exhaust all irreducible $L_\mathrm{Vir}(c_{q,p},0)$-modules up to isomorphism. Indeed, every irreducible $L_\mathrm{Vir}(c_{q,p},0)$-module is also an irreducible $V_\mathrm{Vir}(c_{q,p},0)$-module. We describe these modules $L_\mathrm{Vir}(c_{q,p},h_{r,s}^{q,p})$ in more detail in the following. In this case, $L_\mathrm{Vir}(c_{q,p},0)$ is \strat{}. For $r=s=1$ (or $r=q-1$ and $s=p-1$) we recover the simple \voa{} $L_\mathrm{Vir}(c_{q,p},0)$ itself, i.e.\ the vacuum module.

\medskip

For later reference, we note that the character of the Verma module $M_\mathrm{Vir}(c,h)$ of highest weight $h\in\C$ is
\begin{equation*}
\ch^*_{M_\mathrm{Vir}(c,h)}(q)=\tr_{M_\mathrm{Vir}(c,h)}q^{L_0}=\frac{q^h}{\prod_{n=1}^\infty(1-q^{n})}=\frac{q^h}{(q;q)_\infty}.
\end{equation*}

%%%%%%%%%%%%%%%%%%%%%%%%%%%%%%%

\subsubsection{Rational Case}

We assume that $c=c_{q,p}$ with $p,q\in\Z_{\geq2}$ and $(p,q)=1$ so that the simple Virasoro \voa{} $L_\mathrm{Vir}(c_{q,p},0)$ is \strat{}. The irreducible modules of $L_\mathrm{Vir}(c_{q,p},0)$ are given by $L_\mathrm{Vir}(c_{q,p},h_{r,s}^{q,p})$ with lowest $L_0$-eigenvalue
\begin{equation*}
h_{r,s}^{q,p}\coloneqq\frac{(qs-pr)^2-(p-q)^2}{4pq}
\end{equation*}
for $r,s\in\Z$ with $1\leq r\leq q-1$ and $1\leq s\leq p-1$. But note that $h_{q-r,p-s}^{q,p}=h_{r,s}^{q,p}$ so that each irreducible module appears twice in this enumeration; and hence there are exactly $(p-1)(q-1)/2$ irreducible modules up to isomorphism. The modular tensor category $\Rep(L_\mathrm{Vir}(c_{q,p},0))$ is described further in \autoref{sec:integral} in the special case of $p=3$, in particular the fusion rules.

The module $\smash{L_\mathrm{Vir}(c_{q,p},h_{r,s}^{q,p})}$ is formed by taking the quotient of the Verma module $\smash{M_\mathrm{Vir}(c_{q,p},h_{r,s}^{q,p})}$ by the submodule generated by two singular vectors of $L_0$-weights
\begin{equation*}
\raisebox{.5pt}{\textcircled{\raisebox{-.9pt}{1}}}~h_{r,s}^{q,p}+rs\quad\text{and}\quad\raisebox{.5pt}{\textcircled{\raisebox{-.9pt}{2}}}~h_{r,s}^{q,p}+(q-r)(p-s);
\end{equation*}
see, e.g., \cite{DFMS97}.

We summarise the various $L_\mathrm{Vir}(c_{q,p},0)$-modules in \autoref{fig:ratrep} and point out the similarity to the representation theory of admissible level $L_{-2+q/p}(\sl_2)$; see \autoref{fig:admrep}. Single arrows correspond to taking the quotient by a single singular vector, double arrows to taking the quotient by two.
\begin{figure}[ht]
\begin{tikzcd}[column sep=0]
\{M_\mathrm{Vir}(c_{q,p},h_{r,s}^{q,p})\}\arrow[Rightarrow]{dd}{\raisebox{.5pt}{\textcircled{\raisebox{-.9pt}{1}}},\,\raisebox{.5pt}{\textcircled{\raisebox{-.9pt}{2}}}} &\makebox[0pt]{$\ni$}& M_\mathrm{Vir}(c_{q,p},0)\arrow{d}{\raisebox{.5pt}{\textcircled{\raisebox{-.9pt}{1}}}}\\
&&V_\mathrm{Vir}(c_{q,p},0)\arrow{d}{\raisebox{.5pt}{\textcircled{\raisebox{-.9pt}{2}}}}\\
\{L_\mathrm{Vir}(c_{q,p},h_{r,s}^{q,p})\} &\makebox[0pt]{$\ni$}& L_\mathrm{Vir}(c_{q,p},0)
\end{tikzcd}
\caption{Representations of rational $L_\mathrm{Vir}(c_{q,p},0)$.}
\label{fig:ratrep}
\end{figure}

The formal characters are
\begin{align*}
&\ch^*_{L_\mathrm{Vir}(c_{q,p},h_{r,s}^{q,p})}(q)=\tr_{L_\mathrm{Vir}(c_{q,p},h_{r,s}^{q,p})}q^{L_0}\\
&=q^{h_{r,s}^{q,p}}\frac{\sum_{j\in\Z}(q^{pqj^2+(qs-pr)j}-q^{pqj^2-j(qs+pr)+rs})}{\prod_{n=1}^\infty(1-q^n)}\\
&=q^{h_{r,s}^{q,p}}\frac{\sum_{j\in\Z}(q^{pqj^2+(qs-pr)j}-q^{pqj^2-j(qs+pr)+rs})}{(q;q)_\infty}
\end{align*}
for $r,s\in\Z$ with $1\leq r\leq q-1$ and $1\leq s\leq p-1$.

\subsubsection*{Boundary Case \texorpdfstring{$q=2$}{q=2}}

As for the affine \voa{}s for $\sl_2$ \cite{KW17}, the characters simplify in the boundary case of $q=2$, which is the smallest possible value of $q$. That is, we consider the simple Virasoro \voa{} $L_\mathrm{Vir}(c_{2,p},0)$ of central charge $c_{2,p}=13-12/p-3p$ for $p\geq3$ odd. In that case, the $(p-1)/2$ many irreducible modules are $L_\mathrm{Vir}(c_{2,p},h_{1,s}^{2,p})$ for $1\leq s\leq p-1$, with the symmetry $s\mapsto p-s$. The characters are
\begin{align*}
&\ch^*_{L_\mathrm{Vir}(c_{2,p},h_{1,s}^{2,p})}(q)=q^{c_{2,p}/24}\frac{q^{s^2/2p}\vartheta_{11}(q^s,q^p)}{q^{1/6}\vartheta_{11}(q,q^3)}\\
&=\i q^{c_{2,p}/24}\frac{q^{s^2/2p}\vartheta_{11}(q^s,q^p)}{q^{1/24}(q;q)_\infty}=\i q^{c_{2,p}/24}\frac{q^{s^2/2p}\vartheta_{11}(q^s,q^p)}{\eta(q)}.
\end{align*}
Here, the factor of $q^{c_{2,p}/24}$ appears because we considered the characters without the usual factor of $q^{-c_{2,p}/24}$ to make them modular invariant.

%%%%%%%%%%%%%%%%%%%%%%%%%%%%%%%

\subsubsection{Logarithmic Case \texorpdfstring{$q=1$}{q=1}}

In the following, we consider the case of central charge $c=c_{1,p}=13-6p-6/p$ for $p\in\Z_{\geq2}$, the logarithmic $(1,p)$-minimal models. As stated above, the $\smash{L_\mathrm{Vir}(c_{1,p},h)}$ for $h\in\C$ are all the irreducible modules for $\smash{L_\mathrm{Vir}(c_{1,p},0)=V_\mathrm{Vir}(c_{1,p},0)}$ up to equivalence. They are described in, e.g., \cite{FFT11}. For generic highest weight $h\in\C$, the Verma module $\smash{M_\mathrm{Vir}(c_{1,p},h)=L_\mathrm{Vir}(c_{1,p},h)}$ is already irreducible. The only irreducible modules that are not Verma modules are the
\begin{equation*}
M_{r,s;p}\coloneqq L_\mathrm{Vir}(c_{1,p},h_{r,s}^{1,p})
\end{equation*}
with
\begin{equation*}
h_{r,s}^{1,p}\coloneqq\frac{(s-pr)^2-(p-1)^2}{4p}=\frac{p}{4}(r^2-1)+\frac{1}{4p}(s^2-1)+\frac{1-rs}{2}
\end{equation*}
for $r\in\Ns$ and $1\leq s\leq p$. They are formed by taking a quotient of the Verma module $\smash{M_\mathrm{Vir}(c_{1,p},h_{r,s}^{1,p})}$ of highest weight $\smash{h_{r,s}^{1,p}}$ by the submodule generated by one singular vector~\raisebox{.5pt}{\textcircled{\raisebox{-.9pt}{1}}} of $L_0$-weight $\smash{h_{r,s}^{1,p}+rs}$. For $r=s=1$ we recover the vacuum module $\smash{M_{1,1;p}=L_\mathrm{Vir}(c_{1,p},0)=V_\mathrm{Vir}(c_{1,p},0)}$.

We summarise the modules of $\smash{L_\mathrm{Vir}(c_{1,p},0)=V_\mathrm{Vir}(c_{1,p},0)}$ in \autoref{fig:logrep}, pointing out the similarity to the modules of near-admissible level $L_{-2+1/p}(\sl_2)$; see \autoref{fig:nearadmrep}. Single arrows correspond to taking the quotient by a single singular vector.
\begin{figure}[ht]
\begin{tikzcd}[column sep=0]
h\neq h_{r,s}^{1,p}{:}&h_{r,s}^{1,p}{:}\\[-10pt]
&M_\mathrm{Vir}(c_{1,p},h_{r,s}^{1,p})\arrow{d}[swap]{\raisebox{.5pt}{\textcircled{\raisebox{-.9pt}{1}}}}\\
M_\mathrm{Vir}(c_{1,p},h)=L_\mathrm{Vir}(c_{1,p},h)&M_{r,s;p}=L_\mathrm{Vir}(c_{1,p},h_{r,s}^{1,p})
\end{tikzcd}
\caption{Representations of logarithmic $L_\mathrm{Vir}(c_{1,p})$.}
\label{fig:logrep}
\end{figure}

The characters of the irreducible Verma modules $M_\mathrm{Vir}(c_{1,p},h)=L_\mathrm{Vir}(c_{1,p},h)$ are simply
\begin{equation*}
\ch^*_{M_\mathrm{Vir}(c_{1,p},h)}(q)=\frac{q^h}{\prod_{n=1}^\infty(1-q^n)}=\frac{q^h}{(q;q)_\infty}
\end{equation*}
for $h\neq h_{r,s}^{1,p}$, and those of the irreducible modules $M_{r,s;p}=L_\mathrm{Vir}(c_{1,p},h_{r,s}^{1,p})$ are
\begin{equation*}
\ch^*_{M_{r,s;p}}(q)=(1-q^{rs})\frac{q^{h_{r,s}^{1,p}}}{\prod_{n=1}^\infty(1-q^n)}=(1-q^{rs})\frac{q^{h_{r,s}^{1,p}}}{(q;q)_\infty}
\end{equation*}
for $r\in\Ns$ and $1\leq s\leq p$.

%%%%%%%%%%%%%%%%%%%%%%%%%%%%%%%
%%%%%%%%%%%%%%%%%%%%%%%%%%%%%%%

\section{Character Identities}

In this section, we prove the main results of this text, certain character identities between (irreducible) modules for the simple affine \voa{} $L_k(\sl_2)$ and for the simple Virasoro \voa{} $L_\mathrm{Vir}(c,0)$. The character identities can also be understood as grading-respecting vector-space isomorphisms.

First, we establish a relation on the level of Verma modules, which is rather straightforward. However, specialising to the level $k=k_{q,p}=-2+q/p$ for the affine and to the central charge $c=c_{q,3p}=13-18p/q-2q/p$ for the Virasoro \voa{} (also allowing $q=1$), we observe that the correspondence respects the weights of the singular vectors. This implies that the character identities descend to irreducible quotients, which is a more subtle result.

Note that we do not relate the level $k=-2+q/p$ to the central charge $c_{q,p}$, like under quantum Hamiltonian reduction \cite{FF90c,FF92}, but rather to $c_{q,3p}$. Concretely, we relate irreducible modules of the admissible-level $L_{-2+q/p}(\sl_2)$ to those of rational Virasoro \voa{}s $L_\mathrm{Vir}(c_{q,3p})$, and near-admissible levels with ``$q=1$'' to logarithmic $(1,3p)$-minimal models. We study some aspects of these two cases in more detail in \autoref{sec:integral} and \autoref{sec:gvsi}.

%%%%%%%%%%%%%%%%%%%%%%%%%%%%%%%

\subsection{Verma Modules and Singular Vectors}\label{Verma-modules}

First, let the level $k\in\C\setminus\{-2\}$ and $c\in\C$ be arbitrary. We consider Verma modules $M_k(\mu)$, $\mu\in\C$, for the universal affine \voa{} $V^k(\sl_2)$ and Verma modules $M_\mathrm{Vir}(c,h)$, $h\in\C$, for the universal Virasoro \voa{} $V_\mathrm{Vir}(c,0)$.

By the Poincaré-Birkhoff-Witt theorem, the Verma module $M_k(\mu)$ has a basis consisting of vectors of the form
\begin{equation*}
\prod_{i=1}^\infty(e_{-i})^{k_i}(h_{-i})^{l_i}(f_{-i})^{m_i}(f_0)^nv
\end{equation*}
where $k=(k_i)_{i=1}^\infty$, $l=(l_i)_{i=1}^\infty$ and $m=(m_i)_{i=1}^\infty$ are finite sequences with values in $\N$, $n\in\N$ and $v$ is the highest-weight vector. On the other hand, the Verma module $M_\mathrm{Vir}(c,h)$ has a basis consisting of vectors of the form
\begin{equation*}
\prod_{i=1}^\infty(L_{-i})^{m_i}w
\end{equation*}
where $m=(m_i)_{i=1}^\infty$ is a finite sequence with values in $\N$ and $w$ is the highest-weight vector. Hence, we can define a vector-space isomorphism
\begin{equation}\label{eq:phiplus}
\varphi^+\colon M_k(\mu)\overset{\sim}{\longrightarrow}M_\mathrm{Vir}(c,h)
\end{equation}
by the assignment
\begin{align*}
e_{-i}\mapsto L_{-3i+1},\quad h_{-i}\mapsto L_{-3i},\quad f_{-i}\mapsto L_{-3i-1},\quad f_0\mapsto L_{-1},\quad v\mapsto w
\end{align*}
for $i\in\Ns$. By definition, the map $\varphi^+$ maps a vector of $(h_0,L_0)$-weight $(f,n)$ to a vector of $L_0$-weight $n'$ satisfying
\begin{equation*}
n'=-f/2+3n+\Bigl(h-\frac{\mu(3\mu+2-2k)}{4(k+2)}\Bigr).
\end{equation*}

The graded vector-space isomorphism $\varphi^+$ immediately implies the following character identity for Verma modules
{\allowdisplaybreaks
\begin{align*}
&\ch^*_{M_k(\mu)}(w,q)\Big|_{w^fq^n\mapsto q^{n'}}\\
&=\frac{(q^{-1/2})^\mu(q^3)^{\mu(\mu+2)/(4(k+2))}}{((q^{-1/2})^2q^3;q^3)_\infty(q^3;q^3)_\infty((q^{-1/2})^{-2};q^3)_\infty}q^{h-\mu(3\mu+2-2k)/(4(k+2))}\\
&=\frac{q^h}{(q^2;q^3)_\infty(q^3;q^3)_\infty(q^1;q^3)_\infty}=\frac{q^h}{(q;q)_\infty}=\ch^*_{M_\mathrm{Vir}(c,h)}(q)
\end{align*}
}%
under the substitution $w^fq^n\mapsto q^{n'}=q^{-f/2+3n+(h-\mu(3\mu+2-2k)/(4(k+2)))}$ of formal variables. After all, this character identity simply records the way $\varphi^+$ is compatible with the $h_0$- and $L_0$-grading.

\begin{rem}
The vector-space isomorphism $\varphi^+$ in \eqref{eq:phiplus} and its analogues in \autoref{sec:char}, \autoref{sec:universal}, \autoref{sec:relaxedrelaxed} and \autoref{sec:whittaker} are based on the isomorphism as graded vector spaces of the $\Z$-graded Lie algebras
\begin{equation*}
\hat{\sl}_2\cong\mathcal{L},
\end{equation*}
which respects the natural triangular decompositions. Here, the Virasoro Lie algebra $\smash{\mathcal{L}=\bigoplus_{s\in\Z}\C L_s\oplus\C C}$ is equipped with the standard $\Z$-grading (given by $-s$ in the above direct sum) and $\smash{\hat\sl_2}=\sl_2\otimes\C[t,t^{-1}]\oplus\C K$ with the $\Z$-grading defined by $\wt(e_i)=-1-3i$, $\wt(h_i)=-3i$, $\wt(f_i)=1-3i$ and $\wt(K)=0$ for $i\in\Z$. In fact, the $\Z$-grading on $\smash{\hat\sl_2}$ is uniquely determined by the condition that it extends the principal grading on $\sl_2$ (i.e.\ $\wt(e)=-1$, $\wt(h)=0$ and $\wt(f)=1$) and that $\smash{\hat{\sl}_2\cong\mathcal{L}}$ as graded vector spaces, preserving the natural triangular decomposition.

The above vector-space isomorphism (see also \autoref{prop:vermacorr} below) shows that the isomorphism $\smash{\hat{\sl}_2\cong\mathcal{L}}$ extends to functorial isomorphisms between Verma modules for these Lie algebras as graded vector spaces.
\end{rem}

\smallskip

So far, the vector-space isomorphism $\varphi^+$ or the corresponding character identity is perhaps not very enlightening (but see \autoref{rem:block_equivalences} below). However, we now study situations in which $\varphi^+$ respects the weights of singular vectors. Also note that, so far, we have kept the level $k$ and the central charge $c$ and moreover the highest weights $\mu$ and $h$ completely general.

Supposing that $k=k_{q,p}=-2+q/p$, i.e.\ the admissible case (or analogously the near-admissible case with $q=1$), recall from \autoref{sec:reps} that the Verma modules on both sides of $\varphi^+$ admit (at most) two families of singular vectors $\raisebox{.5pt}{\textcircled{\raisebox{-.9pt}{1}}}$ and $\raisebox{.5pt}{\textcircled{\raisebox{-.9pt}{2}}}$ such that these singular vectors generate the maximal ideals in these Verma modules.

Let $\mu=\smash{\mu_{r,s}=r-qs/p}$ so that both singular vectors in $M_k(\mu)$ are present. We recall that they have $(h_0,L_0)$-weights $(r-qs/p,\ell_{r,s}^{q,p})$ plus
\begin{equation*}
\raisebox{.5pt}{\textcircled{\raisebox{-.9pt}{1}}}~(-2(r+1),s(r+1))\quad\text{and}\quad\raisebox{.5pt}{\textcircled{\raisebox{-.9pt}{2}}}~(2(q-r-1),(p-s)(q-r-1)).
\end{equation*}
A direct computation shows that these will be mapped under $\varphi^+$ to vectors in $M_\mathrm{Vir}(c,h)$ of the following, surprisingly simple, $L_0$-weights
\begin{equation*}
\raisebox{.5pt}{\textcircled{\raisebox{-.9pt}{1}}}~h+(r+1)(3s+1)\quad\text{and}\quad\raisebox{.5pt}{\textcircled{\raisebox{-.9pt}{2}}}~h+(q-r-1)(3p-3s+1),
\end{equation*}
respectively. Comparing this with the formulae for the weights of the two families of singular vectors in the Verma modules $M_\mathrm{Vir}(c,h)$, we are prompted to set
\begin{equation*}
c\coloneqq c_{q,3p}=13-2(k_{q,p}+2)-\frac{18}{k_{q,p}+2}\quad\text{and}\quad h\coloneqq h_{r+1,3s+1}^{q,3p},
\end{equation*}
which only makes sense if $3\nmid q$. Assuming this, $\varphi^+$ maps the singular vectors $\raisebox{.5pt}{\textcircled{\raisebox{-.9pt}{1}}}$ and $\raisebox{.5pt}{\textcircled{\raisebox{-.9pt}{2}}}$ in $M_{-2+q/p}(\mu_{r,s})$ exactly to vectors in $M_\mathrm{Vir}(c_{q,3p},\smash{h_{r+1,3s+1}^{q,3p}})$ that have the same $L_0$-weights as the singular vectors $\raisebox{.5pt}{\textcircled{\raisebox{-.9pt}{1}}}$ and $\raisebox{.5pt}{\textcircled{\raisebox{-.9pt}{2}}}$ in $M_\mathrm{Vir}(c_{q,3p},\smash{h_{r+1,3s+1}^{q,3p}})$. But note that, as we discuss below, it is not generally true that $\varphi^+$, as defined in \eqref{eq:phiplus}, actually maps singular vectors to singular vectors.

There is a positive side effect of fixing
\begin{align*}
h&=h_{r+1,3s+1}^{q,3p}=\frac{(q(3s+1)-3p(r+1))^2-(3p-q)^2}{12pq}\\
&=\frac{\mu_{r,s}(3\mu_{r,s}+2-2k_{q,p})}{4(k_{q,p}+2)},
\end{align*}
as described above. Indeed, $h-\mu(3\mu+2-2k)/(4(k+2))$ is the constant offset in the grading relation of $\varphi^+$. Hence, with the above choice of $h$ (depending on the level $k$ and weight $\mu$), $\varphi^+$ maps a vector of $(h_0,L_0)$-weight $(f,n)$ to a vector of $L_0$-weight
\begin{align*}
n'=-f/2+3n.
\end{align*}
The corresponding character identity is then based on the simple substitution $(w,q)\mapsto(q^{-1/2},q^3)$. We summarise the discussion so far:
\begin{prop}\label{prop:vermacorr}
Consider the Verma module $M_k(\mu)$ for the universal affine \voa{} $V^k(\sl_2)$ with level $k\in\C\setminus\{-2\}$ and weight $\mu\in\C$, as well as the Verma module $M_\mathrm{Vir}(c,h)$ for the universal Virasoro \voa{} $V_\mathrm{Vir}(c,0)$ with central charge $c\in\C$ and weight $h\in\C$. We set $k'=(k+2)/3-2$,
\begin{equation*}
c=c(k')=13-2(k+2)-\frac{18}{k+2}\quad\text{and}\quad h=h(k',\mu)=\frac{\mu(3\mu+2-2k)}{4(k+2)}.
\end{equation*}
Then $\varphi^+\colon M_k(\mu)\overset{\sim}{\longrightarrow} M_\mathrm{Vir}(c,h)$ in \eqref{eq:phiplus} defines a vector-space isomorphism that maps a vector of $(h_0,L_0)$-weight $(f,n)$ to a vector of $L_0$-weight $n'=-f/2+3n$. It implies the character identity
\begin{equation*}
\ch^*_{M_k(\mu)}(q^{-1/2},q^3)=\ch^*_{M_\mathrm{Vir}(c,h)}(q).
\end{equation*}
\end{prop}
Recall that $c(k)=13-6(k+2)-6/(k+2)$ is the central charge of the Virasoro \voa{} obtained by quantum Drinfeld-Sokolov reduction of the affine \voa{} for $\sl_2$ at level~$k$; and $h(k,\mu)=\mu(\mu+2-2k)/(4(k+2))$ is the highest weight of the Verma module $\smash{M_\mathrm{Vir}(c,h)=H_\mathrm{DS}^0(M_k(\mu))}$ obtained by the reduction of the Verma module $M_k(\mu)$; see also \autoref{rem:block_equivalences} below.

\medskip

We now describe an analogous map $\varphi^-$ that relates lowest-weight Verma modules for $V^k(\sl_2)$ and Verma modules for $V_\mathrm{Vir}(c,0)$. By repeating the construction of $\varphi^+$ but with the roles of $e$ and $f$ reversed we obtain a vector-space isomorphism
\begin{equation}\label{eq:phiminus}
\varphi^-\colon M^-_k(-\mu)\overset{\sim}{\longrightarrow}M_\mathrm{Vir}(c,h)
\end{equation}
for all $\mu\in\C$ and $h\in\C$. By definition, the map $\varphi^-$ maps a homogeneous vector of $(h_0,L_0)$-weight $(f,n)$ to a vector of $L_0$-weight $n'$ satisfying
\begin{equation*}
n'=f/2+3n+\Bigl(h-\frac{\mu(3\mu+2-2k)}{4(k+2)}\Bigr).
\end{equation*}

In the admissible case, i.e.\ for $k=k_{q,p}=-2+q/p$ and $\mu=\mu_{r,s}=r-qs/p$ so that both singular vectors $\raisebox{.5pt}{\textcircled{\raisebox{-.9pt}{1}}}$ and $\raisebox{.5pt}{\textcircled{\raisebox{-.9pt}{2}}}$ in $\smash{M^-_k(-\mu)}$ are present, we recall that they have $(h_0,L_0)$-weights $\smash{(-r+qs/p,\ell_{r,s}^{q,p})}$ plus
\begin{equation*}
\raisebox{.5pt}{\textcircled{\raisebox{-.9pt}{1}}}~(2(r+1),s(r+1))\quad\text{and}\quad\raisebox{.5pt}{\textcircled{\raisebox{-.9pt}{2}}}~(-2(q-r-1),(p-s)(q-r-1)).
\end{equation*}
These will be mapped under $\varphi^-$ to vectors in $M_\mathrm{Vir}(c,h)$ of $L_0$-weights
\begin{equation*}
\raisebox{.5pt}{\textcircled{\raisebox{-.9pt}{1}}}~h+(r+1)(3s+1)\quad\text{and}\quad\raisebox{.5pt}{\textcircled{\raisebox{-.9pt}{2}}}~h+(q-r-1)(3p-3s+1),
\end{equation*}
respectively, just like in the highest-weight Verma modules under $\varphi^+$. Again, we compare this with the formulae for the weights of the two families of singular vectors in the Verma modules $M_\mathrm{Vir}(c,h)$ and set $c\coloneqq c_{q,3p}$ and $\smash{h\coloneqq h_{r+1,3s+1}^{q,3p}}$, supposing that $3\nmid q$. Then $\varphi^-$ maps the singular vectors $\raisebox{.5pt}{\textcircled{\raisebox{-.9pt}{1}}}$ and $\raisebox{.5pt}{\textcircled{\raisebox{-.9pt}{2}}}$ in $\smash{M^-_{-2+q/p}(-\mu_{r,s})}$ exactly to vectors in $\smash{M_\mathrm{Vir}(c_{q,3p},h_{r+1,3s+1}^{q,3p})}$ that have the same $L_0$-weights as the singular vectors $\raisebox{.5pt}{\textcircled{\raisebox{-.9pt}{1}}}$ and $\raisebox{.5pt}{\textcircled{\raisebox{-.9pt}{2}}}$ in $\smash{M_\mathrm{Vir}(c_{q,3p},h_{r+1,3s+1}^{q,3p})}$.

Again, fixing $\smash{h=h_{r+1,3s+1}^{q,3p}}$ also implies that $\varphi^-$ maps a vector of $(h_0,L_0)$-weight $(f,n)$ simply to a vector of $L_0$-weight
\begin{align*}
n'=f/2+3n.
\end{align*}
Overall, we obtain:
\begin{prop}\label{prop:vermacorr2}
Consider the lowest-weight Verma module $M^-_k(-\mu)$ for the universal affine \voa{} $V^k(\sl_2)$ with level $k\in\C\setminus\{-2\}$ and weight $\mu\in\C$, as well as the Verma module $M_\mathrm{Vir}(c,h)$ for the universal Virasoro \voa{} $V_\mathrm{Vir}(c,0)$ with central charge $c\in\C$ and weight $h\in\C$. We set $k'=(k+2)/3-2$, $c=c(k')$ and $h=h(k',\mu)$, exactly as in \autoref{prop:vermacorr}.

Then $\varphi^-\colon M_k(-\mu)\overset{\sim}{\longrightarrow} M_\mathrm{Vir}(c,h)$ in \eqref{eq:phiminus} defines a vector-space isomorphism that maps a vector of $(h_0,L_0)$-weight $(f,n)$ to a vector of $L_0$-weight $n'=f/2+3n$. It implies the character identity
\begin{equation*}
\ch^*_{M^-_k(-\mu)}(q^{1/2},q^3)=\ch^*_{M_\mathrm{Vir}(c,h)}(q).
\end{equation*}
\end{prop}

\smallskip

Finally, we note that the character identities in this section are compatible with known results for the affine category $\mathcal{O}$ \cite{Fie06,Ara14}. We thank the referee for pointing us to these references.
\begin{rem}\label{rem:block_equivalences}
Equivalences between suitable blocks of the category $\mathcal{O}$ for symmetrisable Kac-Moody algebras \cite{Fie06}, here applied to $\smash{\hat{\sl}_2}$, combined with quantum Drinfeld-Sokolov reduction (and two-sided BGG resolutions \cite{Ara14}) relate highest-weight Verma modules (and their simple quotients).

Quantum Drinfeld-Sokolov reduction provides a functor
\begin{equation*}
\smash{H_\mathrm{DS}^0\colon\mathcal{O}^k(\hat{\sl}_2)\to\mathcal{O}^c(\mathcal{L})},\quad M_k(\mu)\mapsto H_\mathrm{DS}^0(M_k(\mu))=M_\mathrm{Vir}(c,h)
\end{equation*}
between the categories $\mathcal{O}$ for the affine Lie algebra $\hat{\sl}_2$ at level $k$ and for the Virasoro Lie algebra $\mathcal{L}$ at central charge $c=c(k)=13-6(k+2)-6/(k+2)$. The functor maps the Verma module $M_k(\mu)$ of highest weight $\mu$ to the Verma module $M_\mathrm{Vir}(c,h)$ of highest weight $h=h(k,\mu)=\mu(\mu+2-2k)/(4(k+2))$. This reduction does not preserve the characters.

On the other hand, in this paper, we consider for levels $k$ and $k'$ satisfying $k'+2=(k+2)/3$, at the level of characters, the composition
\begin{equation*}
\mathcal{O}^k(\hat{\sl}_2)\dashrightarrow\mathcal{O}^{k'}(\hat{\sl}_2)\overset{H_\mathrm{DS}^0}{\longrightarrow}\mathcal{O}^{c(k')}(\mathcal{L}),
\end{equation*}
which maps
\begin{equation*}
M_k(\mu)\mapsto M_{k'}(\mu)\mapsto H_\mathrm{DS}^0(M_{k'}(\mu))=M_\mathrm{Vir}(c,h),
\end{equation*}
where $c=c(k')=13-6(k'+2)-6/(k'+2)=13-2(k+2)-18/(k+2)$ and $h=h(k',\mu)=\mu(\mu+2-2k')/(4(k'+2))=\mu(3\mu+2-2k)/(4(k+2))$, just as in \autoref{prop:vermacorr}.

The dashed arrow is not an honest functor. However, upon restriction to suitable blocks $[\mu]$ and $[\mu']$, there are equivalences of categories $\smash{\mathcal{O}_{[\mu]}^k(\hat{\sl}_2)\to\mathcal{O}_{[\mu']}^{k'}(\hat{\sl}_2)}$ mapping Verma modules to Verma modules \cite{Fie06}. Overall, this can be seen as a more categorical interpretation of our character identities between Verma modules in \autoref{prop:vermacorr}.

\smallskip

One obtains similar character identities for non-Verma modules in such blocks, in particular for irreducible highest-weight modules. For instance, for admissible levels $k=-2+q/p$ and $k'=-2+q/3p$, when the block $\smash{\mathcal{O}_{[\mu]}^{-2+q/p}(\hat{\sl}_2)}$ contains the irreducible highest-weight modules $L_{-2+q/p}(\mu_{r,s})$, there is an equivalence
\begin{equation*}
\mathcal{O}_{[\mu]}^{-2+q/p}(\hat{\sl}_2)\overset{\cong}{\longrightarrow}\mathcal{O}_{[\mu]}^{-2+q/3p}(\hat{\sl}_2),
\end{equation*}
and the BGG-type resolutions of the irreducible highest-weight modules in \cite{Ara14}, combined with quantum Drinfeld-Sokolov reduction, then provide another way to prove the character identities in \autoref{cor:charid} below. Indeed, the identities for the irreducible characters would follow from the corresponding identities for Verma module characters, applied term by term to the alternating sums.

Similar statements would also hold for the irreducible modules at the near-admissible levels; see \autoref{cor:charid2} and \autoref{prop:gvsi} below.

We believe that this alternative approach may be used to prove higher-rank generalisations of the results in this paper.
\end{rem}

%%%%%%%%%%%%%%%%%%%%%%%%%%%%%%%

\subsection{Irreducible Modules: Admissible Case}\label{sec:admissible}

We now prove character identities between irreducible modules of the simple affine \voa{}s $L_{-2+q/p}(\sl_2)$ at admissible levels $k=k_{q,p}\coloneqq-2+q/p$ with $q,p\in\Z$, $q\geq2$, $p\geq1$ and $(p,q)=1$ and irreducible modules for the rational Virasoro minimal models $L_\mathrm{Vir}(c_{q,3p},0)$, assuming that $3\nmid q$. The correspondence becomes particularly nice for integral levels, which we discuss in detail in \autoref{sec:integral}.

On the level of Verma modules, we have seen:
\begin{prop}\label{prop:vermasingadm}
Let $q,p\in\Z$ with $q\geq2$, $p\geq1$ and $(p,q)=1$. Moreover, assume that $3\nmid q$. Let $r,s\in\Z$ with $0\leq r\leq q-2$ and $0\leq s\leq p-1$. Then there is a vector-space isomorphism \eqref{eq:phiplus}
\begin{equation*}
\varphi^+\colon M_{-2+q/p}(\mu_{r,s})\overset{\sim}{\longrightarrow}M_\mathrm{Vir}(c_{q,3p},h_{r+1,3s+1}^{q,3p})
\end{equation*}
between Verma modules that maps a vector of $(h_0,L_0)$-weight $(f,n)$ to a vector of $L_0$-weight $n'=-f/2+3n$. In particular, the singular vectors in $\smash{M_{-2+q/p}(\mu_{r,s})}$ of $(h_0,L_0)$-weights $\smash{(r-qs/p,\ell_{r,s}^{q,p})}$ plus
\begin{equation*}
\raisebox{.5pt}{\textcircled{\raisebox{-.9pt}{1}}}~(-2(r+1),s(r+1))\quad\text{and}\quad\raisebox{.5pt}{\textcircled{\raisebox{-.9pt}{2}}}~(2(q-r-1),(p-s)(q-r-1))
\end{equation*}
are mapped to vectors in $\smash{M_\mathrm{Vir}(c_{q,3p},h_{r+1,3s+1}^{q,3p})}$ of $L_0$-weights
\begin{equation*}
\raisebox{.5pt}{\textcircled{\raisebox{-.9pt}{1}}}~h_{r+1,3s+1}^{q,3p}+(r+1)(3s+1)\quad\text{and}\quad\raisebox{.5pt}{\textcircled{\raisebox{-.9pt}{2}}}~h_{r+1,3s+1}^{q,3p}+(q-r-1)(3p-3s+1),
\end{equation*}
which are exactly the $L_0$-weights of the two singular vectors $\raisebox{.5pt}{\textcircled{\raisebox{-.9pt}{1}}}$ and $\raisebox{.5pt}{\textcircled{\raisebox{-.9pt}{2}}}$, respectively, in $\smash{M_\mathrm{Vir}(c_{q,3p},h_{r+1,3s+1}^{q,3p})}$.
\end{prop}

Because the weights of the singular vectors are preserved (although not necessarily the singular vectors themselves), this suggests that the above character identity for the Verma modules descends to a character identity for the irreducible quotients. Rather than verifying this by studying the graded structure of the maximal ideals, we directly compare the characters of the irreducible quotients (see \autoref{sec:reps}).
\begin{prop}\label{cor:charid}
Let $q,p\in\Z$ with $q\geq2$, $p\geq1$ and $(p,q)=1$. Moreover, assume that $3\nmid q$. Let $r,s\in\Z$ with $0\leq r\leq q-2$ and $0\leq s\leq p-1$. Then
\begin{equation*}
\ch^*_{L_{-2+q/p}(\mu_{r,s})}(q^{-1/2},q^3)=\ch^*_{L_\mathrm{Vir}(c_{q,3p},h_{r+1,3s+1}^{q,3p})}(q).
\end{equation*}
\end{prop}
We note that since $0\leq r\leq q-2$ and $0\leq s\leq p-1$, it follows that $1\leq r+1\leq q-1$ and $1\leq3s+1\leq3p-2\leq 3p-1$; so the new indices $r+1$ and $3s+1$ are in the correct ranges.
\begin{proof}
We consider the replacement $(w,q)\mapsto(q^{-1/2},q^3)$ and obtain
{\allowdisplaybreaks
\begin{align*}
&\ch^*_{L_{-2+q/p}(\mu_{r,s})}(q^{-1/2},q^3)=q^{3/8-3p/4q}\frac{\theta_{b_+,2a}(q^{-1/2p},q^3)-\theta_{b_-,2a}(q^{-1/2p},q^3)}{q^{-1/2}q^{3/8}(q^3;q^3)_\infty(q^2;q)_\infty(q;q)_\infty}\\
&=q^{1/2-3p/4q}\frac{\sum_{j\in b_+/2a+\Z}q^{3pqj^2-qj}-\sum_{j\in b_-/2a+\Z}q^{3pqj^2-qj}}{(q;q)_\infty}\\
&=q^{\frac{6p^2r-2pqr+3p^2r^2-6pqs+2q^2s-6pqrs+3q^2s^2}{4pq}}\\&\quad\frac{\sum_{j\in\Z}q^{(3p-q+3pr-3qs)j+3pqj^2}-\sum_{j\in\Z}q^{(1+r)(1+3s)+(-3p-q-3pr-3qs)j+3pqj^2}}{(q;q)_\infty}\\
&=q^{\frac{(q\tilde{s}-3p\tilde{r})^2-(3p-q)^2}{12pq}}\frac{\sum_{j\in\Z}q^{(3p\tilde{r}-q\tilde{s})j+3pqj^2}-\sum_{j\in\Z}q^{\tilde{r}\tilde{s}-(3p\tilde{r}+q\tilde{s})j+3 pqj^2}}{(q;q)_\infty}\\
&=q^{h_{\tilde{r},\tilde{s}}^{q,3p}}\frac{\sum_{j\in\Z}q^{(3p\tilde{r}-q\tilde{s})j+3pqj^2}-\sum_{j\in\Z}q^{\tilde{r}\tilde{s}-(3p\tilde{r}+q\tilde{s})j+3 pqj^2}}{(q;q)_\infty}\\
&=\ch^*_{L_\mathrm{Vir}(c_{q,3p},h_{\tilde{r},\tilde{s}}^{q,3p})}(q)
\end{align*}
}%
with $\tilde{r}=r+1$ and $\tilde{s}=3s+1$. This is the asserted character identity.
\end{proof}

In summary, we obtain a correspondence from irreducible $L_{-2+q/p}(\sl_2)$-modules to irreducible $L_\mathrm{Vir}(c_{q,3p},0)$-modules that, on the level of indices, is given by
\begin{equation*}
(q,p)\mapsto (q,3p),\quad (r,s)\mapsto(r+1,3s+1).
\end{equation*}
Note that, except for $p=1$ (see also \autoref{sec:integral}), the correspondence misses some (roughly one third) of the $L_\mathrm{Vir}(c_{q,3p},0)$-modules on the right-hand side.

For $r=s=0$, we obtain the character identity
\begin{equation*}
\ch^*_{L_{-2+q/p}(\sl_2)}(q^{-1/2},q^3)=\ch^*_{L_\mathrm{Vir}(c_{q,3p},0)}(q)
\end{equation*}
for the simple \voa{}s themselves.

\medskip

The analogous result to \autoref{prop:vermasingadm} for lowest-weight Verma modules is:
\begin{prop}\label{prop:lowestvermasingadm}
Let $q,p\in\Z$ with $q\geq2$, $p\geq1$ and $(p,q)=1$. Moreover, assume that $3\nmid q$. Let $r,s\in\Z$ with $0\leq r\leq q-2$ and $0\leq s\leq p-1$. Then there is a vector-space isomorphism \eqref{eq:phiminus}
\begin{equation*}
\varphi^-\colon M^-_{-2+q/p}(-\mu_{r,s})\overset{\sim}{\longrightarrow}M_\mathrm{Vir}(c_{q,3p},h_{r+1,3s+1}^{q,3p})
\end{equation*}
between Verma modules that maps a vector of $(h_0,L_0)$-weight $(f,n)$ to a vector of $L_0$-weight $n'=f/2+3n$. In particular, the singular vectors in $\smash{M^-_{-2+q/p}(-\mu_{r,s})}$ of $(h_0,L_0)$-weights $\smash{(-r+qs/p,\ell_{r,s}^{q,p})}$ plus
\begin{equation*}
\raisebox{.5pt}{\textcircled{\raisebox{-.9pt}{1}}}~(2(r+1),s(r+1))\quad\text{and}\quad\raisebox{.5pt}{\textcircled{\raisebox{-.9pt}{2}}}~(-2(q-r-1),(p-s)(q-r-1))
\end{equation*}
are mapped to vectors in $\smash{M_\mathrm{Vir}(c_{q,3p},h_{r+1,3s+1}^{q,3p})}$ of $L_0$-weights
\begin{equation*}
\raisebox{.5pt}{\textcircled{\raisebox{-.9pt}{1}}}~h_{r+1,3s+1}^{q,3p}+(r+1)(3s+1)\quad\text{and}\quad\raisebox{.5pt}{\textcircled{\raisebox{-.9pt}{2}}}~h_{r+1,3s+1}^{q,3p}+(q-r-1)(3p-3s+1),
\end{equation*}
which are exactly the $L_0$-weights of the two singular vectors $\raisebox{.5pt}{\textcircled{\raisebox{-.9pt}{1}}}$ and $\raisebox{.5pt}{\textcircled{\raisebox{-.9pt}{2}}}$, respectively, in $\smash{M_\mathrm{Vir}(c_{q,3p},h_{r+1,3s+1}^{q,3p})}$.
\end{prop}
Analogously to \autoref{cor:charid}, we then obtain the following result, which we again verify directly on the level of the characters:
\begin{prop}\label{cor:lowestcharid}
Let $q,p\in\Z$ with $q\geq2$, $p\geq1$ and $(p,q)=1$. Moreover, assume that $3\nmid q$. Let $r,s\in\Z$ with $0\leq r\leq q-2$ and $0\leq s\leq p-1$. Then
\begin{equation*}
\ch^*_{L^-_{-2+q/p}(-\mu_{r,s})}(q^{1/2},q^3)=\ch^*_{L_\mathrm{Vir}(c_{q,3p},h_{r+1,3s+1}^{q,3p})}(q).
\end{equation*}
\end{prop}
We note that for $s=0$, this is not a contradiction to \autoref{cor:charid}, given that $\smash{L^-_{-2+q/p}(-\mu_{r,0})}=\smash{L_{-2+q/p}(\mu_{r,0})}$. Indeed, take note of the symmetry
\begin{equation*}
\ch^*_{L_{-2+q/p}(\mu_{r,0})}(w,q)=\ch^*_{L_{-2+q/p}(\mu_{r,0})}(w^{-1},q).
\end{equation*}
Specialising further to $r=s=0$, we obtain the character identity
\begin{equation*}
\ch^*_{L_{-2+q/p}(\sl_2)}(q^{1/2},q^3)=\ch^*_{L_\mathrm{Vir}(c_{q,3p},0)}(q)
\end{equation*}
for the simple \voa{}s themselves.

\smallskip

\begin{rem}\label{rem:inducedGVSI}
The character identities in \autoref{cor:charid} and \autoref{cor:lowestcharid} are equivalent to the existence of graded vector-space isomorphisms
\begin{equation}\label{eq:chiplusminus}
\chi^\pm\colon L^\pm_{-2+q/p}(\pm\mu_{r,s})\overset{\sim}{\longrightarrow}L_\mathrm{Vir}(c_{q,3p},h_{r+1,3s+1}^{q,3p}),
\end{equation}
respectively, between irreducible modules. However, besides not being unique, they would be difficult to write down explicitly, in general.

In \autoref{sec:freefermion}, in the case of $(q,p)=(4,1)$, we shall write down choices of the vector-space isomorphisms $\chi^\pm$ explicitly (for two of the three irreducible modules) by embedding the modules into free-fermion \svoa{}s.
\end{rem}

%%%%%%%%%%%%%%%%%%%%%%%%%%%%%%%

\subsubsection*{Boundary Admissible Case}

We saw in \autoref{sec:reps} that the characters of the irreducible modules for $L_{-2+q/p}(\sl_2)$ and $L_\mathrm{Vir}(c_{q,p},0)$ take a particularly nice form in terms of the Jacobi theta function $\vartheta_{11}(w,q)$ in the boundary (admissible) case, i.e.\ for $q=2$. This can be used to obtain a simple proof of the character identities.

Indeed, let $p\in\Ns$ be odd and $s\in\Z$ with $0\leq s\leq p-1$. Then, recalling that
\begin{align*}
\ch^*_{L^\pm_{-2+2/p}(\pm\mu_{0,s})}(w,q)&=q^{(1-p)/8}w^{\mp2s/p}q^{s^2/2p}\frac{\vartheta_{11}(w^{\mp2}q^s,q^p)}{\vartheta_{11}(w^{\mp2},q)},
\end{align*}
we obtain
\begin{align*}
\ch^*_{L^\pm_{-2+2/p}(\pm\mu_{0,s})}(q^{\mp1/2},q^3)&=q^{3(1-p)/8}\frac{q^{s/p}q^{3s^2/2p}\vartheta_{11}(qq^{3s},q^{3p})}{\vartheta_{11}(q,q^3)}\\
&=q^{c_{2,3p}/24}\frac{q^{(3s+1)^2/6p}\vartheta_{11}(q^{3s+1},q^{3p})}{q^{1/6}\vartheta_{11}(q,q^3)}\\
&=\ch^*_{L_\mathrm{Vir}(c_{2,3p},h_{1,3s+1}^{2,3p})}(q),
\end{align*}
proving \autoref{cor:charid} and \autoref{cor:lowestcharid} in the special case of $q=2$ almost instantaneously, as desired.

%%%%%%%%%%%%%%%%%%%%%%%%%%%%%%%

\subsection{Irreducible Modules: Near-Admissible Case \texorpdfstring{$q=1$}{q=1}}\label{sec:char}

We prove character identities between the irreducible modules for the simple affine \voa{}s $\smash{L_{-2+1/p}(\sl_2)}=\smash{V^{-2+1/p}(\sl_2)}$ at nonadmissible levels $k=k_{1,p}=-2+1/p$ for $p\in\Ns$, called near-admissible levels, and irreducible modules for the $(1,3p)$-minimal models $\smash{L_\mathrm{Vir}(c_{1,3p},0)}=\smash{V_\mathrm{Vir}(c_{1,3p},0)}$ at central charge $c_{1,3p}=13-18p-2/p$, assuming that $3\nmid q$. Some interesting applications will be studied in greater detail in \autoref{sec:gvsi}.

The main difference to the admissible case in \autoref{sec:admissible} is that there is now only one family~$\raisebox{.5pt}{\textcircled{\raisebox{-.9pt}{1}}}$ of singular vectors in the Verma modules (which will actually allow us to write down an explicit graded vector-space isomorphism between the irreducible modules). Analogously to \autoref{prop:vermasingadm}, we obtain:
\begin{prop}\label{prop:vermasingnear}
Let $p\in\Ns$. Let $\mu\in\C$. Then there exists a vector-space isomorphism \eqref{eq:phiplus}
\begin{equation*}
\varphi^+\colon M_{-2+1/p}(\mu)\overset{\sim}{\longrightarrow}M_\mathrm{Vir}(c_{1,3p},\mu(-2+3p(\mu+2))/4)
\end{equation*}
between Verma modules that maps a vector of $(h_0,L_0)$-weight $(f,n)$ to a vector of $L_0$-weight $n'=-f/2+3n$. If $\mu\notin\N$, then both sides of this correspondence are already irreducible.

On the other hand, if $\mu=\mu_{r,0}=r$ for $r\in\N$, this becomes a vector-space isomorphism \eqref{eq:phiplus}
\begin{equation*}
\varphi^+\colon M_{-2+1/p}(\mu_{r,0})\overset{\sim}{\longrightarrow}M_\mathrm{Vir}(c_{1,3p},h_{r+1,1}^{1,3p})
\end{equation*}
that maps the singular vector in $\smash{M_{-2+1/p}(\mu_{r,0})}$ of $(h_0,L_0)$-weight
\begin{equation*}
\raisebox{.5pt}{\textcircled{\raisebox{-.9pt}{1}}}~(r,\ell_{r,0}^{1,p})+(-2(r+1),0)
\end{equation*}
to a vector in $\smash{M_\mathrm{Vir}(c_{1,3p},h_{r+1,1}^{1,3p})}$ of $L_0$-weight
\begin{equation*}
\raisebox{.5pt}{\textcircled{\raisebox{-.9pt}{1}}}~h_{r+1,1}^{1,3p}+r+1,
\end{equation*}
which is exactly the $L_0$-weight of the singular vector $\raisebox{.5pt}{\textcircled{\raisebox{-.9pt}{1}}}$ in $\smash{M_\mathrm{Vir}(c_{1,3p},h_{r+1,1}^{1,3p})}$.
\end{prop}

Similarly to \autoref{cor:charid}, because the weights of the (at most one) singular vectors are preserved, we obtain a character identity for the irreducible quotients.
\begin{prop}\label{cor:charid2}
Let $p\in\Ns$. Let $\mu\in\C\setminus\N$. Then
\begin{equation*}
\ch^*_{L_{-2+1/p}(\mu)}(q^{-1/2},q^3)=\ch^*_{L_\mathrm{Vir}(c_{1,3p},\mu(-2+3p(\mu+2))/4)}(q)
\end{equation*}
where $L_\mathrm{Vir}(c_{1,3p},\mu(-2+3p(\mu+2))/4)=M_\mathrm{Vir}(c_{1,3p},\mu(-2+3p(\mu+2))/4)$ and $L_{-2+1/p}(\mu)=M_{-2+1/p}(\mu)$ are irreducible Verma modules.

On the other hand, let $r\in\N$. Then
\begin{equation*}
\ch^*_{L_{-2+1/p}(\mu_{r,0})}(q^{-1/2},q^3)=\ch^*_{M_{r+1,1;3p}}(q)
\end{equation*}
where $\smash{L_{-2+1/p}(\mu_{r,0})=V^{-2+1/p}(r)}$ is an irreducible Weyl module and moreover $\smash{M_{r+1,1;3p}=L_\mathrm{Vir}(c_{1,3p},h_{r+1,1}^{1,3p})}$ is an irreducible non-Verma module.
\end{prop}
\begin{proof}
The Jacobi character of the Weyl module $\smash{L_{-2+1/p}(\mu_{r,0})=V^{-2+1/p}(r)}$ for $\smash{V^{-2+1/p}(\sl_2)}$ is
\begin{equation*}
\ch^*_{L_{-2+1/p}(\mu_{r,0})}(w,q)=(w^r-w^{-(r+2)})\frac{q^{\frac{pr^2+2pr}{4}}}{(w^2q;q)_\infty(q;q)_\infty(w^{-2};q)_\infty}.
\end{equation*}
Applying $(w,q)\mapsto(q^{1/2},q^3)$ we obtain
\begin{align*}
&\ch^*_{L_{-2+1/p}(\mu_{r,0})}(q^{1/2},q^3)\\
&=(q^{-r/2}-q^{r/2+1})\frac{q^{\frac{3pr^2+6pr}{4}}}{(q^2;q^3)_\infty(q^3;q^3)_\infty(q;q^3)_\infty}\\
&=(1-q^{r+1})\frac{q^{\frac{3pr^2+6pr-2r}{4}}}{(q;q)_\infty}.
\end{align*}
On the other hand, the character for one of the Virasoro modules $M_{\tilde{r},\tilde{s};3p}$ is
\begin{equation*}
\ch^*_{M_{\tilde{r},\tilde{s};3p}}(q)=(1-q^{\tilde{r}\tilde{s}})\frac{q^{h_{\tilde{r},\tilde{s}}^{1,3p}}}{(q;q)_\infty},
\end{equation*}
which specialises to
\begin{equation*}
\ch^*_{M_{r+1,1;3p}}(q)=(1-q^{r+1})\frac{q^{\frac{3pr^2+6pr-2r}{4}}}{(q;q)_\infty}
\end{equation*}
for $(\tilde{r},\tilde{s})=(r+1,1)$. It then follows that, as asserted,
\begin{equation*}
\ch^*_{L_{-2+1/p}(\mu_{r,0})}(q^{1/2},q^3)=\ch^*_{M_{r+1,1;3p}}(q).\qedhere
\end{equation*}
\end{proof}
\autoref{cor:charid2} will also follow from the graded vector-space isomorphism constructed in \autoref{prop:gvsi} below.

We note that the character correspondence only maps to the irreducible modules $\smash{M_{r+1,\tilde{s};3p}}$ for $\tilde{s}=1$ and not for larger $1\leq\tilde{s}\leq3p$, i.e.\ we miss the majority of the $\smash{L_\mathrm{Vir}(c_{1,3p},0)}$-modules on the right-hand side.

For $r=0$ we obtain the character identity
\begin{equation*}
\ch^*_{L_{-2+1/p}(\sl_2)}(q^{-1/2},q^3)=\ch^*_{L_\mathrm{Vir}(c_{1,3p},0)}(q)
\end{equation*}
between the \voa{}s themselves.

\medskip

The corresponding result for lowest-weight modules is:
\begin{prop}\label{prop:lowestvermasingnear}
Let $p\in\Ns$. Let $\mu\in\C$. Then there is a vector-space isomorphism \eqref{eq:phiminus}
\begin{equation*}
\varphi^-\colon M^-_{-2+1/p}(-\mu)\overset{\sim}{\longrightarrow}M_\mathrm{Vir}(c_{1,3p},\mu(-2+3p(\mu+2))/4)
\end{equation*}
between lowest-weight Verma modules that maps a vector of $(h_0,L_0)$-weight $(f,n)$ to a vector of $L_0$-weight $n'=f/2+3n$. If $\mu\notin\N$, both sides of this correspondence are already irreducible.

On the other hand, if $\mu=\mu_{r,0}=r$ for $r\in\N$, this becomes a vector-space isomorphism \eqref{eq:phiminus}
\begin{equation*}
\varphi^-\colon M^-_{-2+1/p}(-\mu_{r,0})\overset{\sim}{\longrightarrow}M_\mathrm{Vir}(c_{1,3p},h_{r+1,1}^{1,3p})
\end{equation*}
that maps the singular vector in $\smash{M^-_{-2+1/p}(-\mu_{r,0})}$ of $(h_0,L_0)$-weight
\begin{equation*}
\raisebox{.5pt}{\textcircled{\raisebox{-.9pt}{1}}}~(-r,\ell_{r,0}^{1,p})+(2(r+1),0)
\end{equation*}
to a vector in $\smash{M_\mathrm{Vir}(c_{1,3p},h_{r+1,1}^{1,3p})}$ of $L_0$-weight
\begin{equation*}
\raisebox{.5pt}{\textcircled{\raisebox{-.9pt}{1}}}~h_{r+1,1}^{1,3p}+r+1,
\end{equation*}
which is exactly the $L_0$-weight of the singular vector $\raisebox{.5pt}{\textcircled{\raisebox{-.9pt}{1}}}$ in $\smash{M_\mathrm{Vir}(c_{1,3p},h_{r+1,1}^{1,3p})}$.
\end{prop}

Again, we obtain a character identity for the irreducible quotients.
\begin{prop}\label{cor:lowestcharid2}
Let $p\in\Ns$. Let $\mu\in\C\setminus\N$. Then
\begin{equation*}
\ch^*_{L^-_{-2+1/p}(-\mu)}(q^{1/2},q^3)=\ch^*_{L_\mathrm{Vir}(c_{1,3p},\mu(-2+3p(\mu+2))/4)}(q)
\end{equation*}
where $L_\mathrm{Vir}(c_{1,3p},\mu(-2+3p(\mu+2))/4)=M_\mathrm{Vir}(c_{1,3p},\mu(-2+3p(\mu+2))/4)$ and $L^-_{-2+1/p}(-\mu)=M^-_{-2+1/p}(-\mu)$ are irreducible Verma modules.

On the other hand, let $r\in\N$. Then
\begin{equation*}
\ch^*_{L^-_{-2+1/p}(-\mu_{r,0})}(q^{1/2},q^3)=\ch^*_{M_{r+1,1;3p}}(q)
\end{equation*}
where $\smash{L^-_{-2+1/p}(-\mu_{r,0})=V^{-2+1/p}(r)}$ is an irreducible Weyl module and moreover $\smash{M_{r+1,1;3p}=L_\mathrm{Vir}(c_{1,3p},h_{r+1,1}^{1,3p})}$.
\end{prop}
For $r=0$ we obtain the character identity
\begin{equation*}
\ch^*_{L_{-2+1/p}(\sl_2)}(q^{1/2},q^3)=\ch^*_{L_\mathrm{Vir}(c_{1,3p},0)}(q)
\end{equation*}
between the \voa{}s themselves.

\subsubsection*{Vector-Space Isomorphism}\label{sec:proof}

Analogously to \autoref{rem:inducedGVSI} in the admissible case, the character identities in \autoref{cor:charid2} and \autoref{cor:lowestcharid2} are equivalent to the existence of graded vector-space isomorphisms
\begin{equation}\label{eq:psiplusminus}
\psi^\pm\colon L^\pm_{-2+1/p}(\pm\mu_{r,0})\overset{\sim}{\longrightarrow}M_{r+1,1;3p},
\end{equation}
respectively, for $r\in\N$. In contrast to the admissible case in \autoref{sec:admissible}, there is now only one family of singular vectors, which actually allows us in the following to describe such isomorphisms $\psi^+$ and $\psi^-$ explicitly.

But note that they are not unique; it is not clear why the map that we describe in the following is the one induced from the original map $\varphi^\pm$ on the Verma modules by taking the quotient by the singular vectors. (Of course, for $\mu\notin\N$, the Verma modules are already irreducible and $\psi^\pm$ and $\varphi^\pm$ are identical.)

We note that in the special case of $p=2$, the map $\psi^-$ was first described in \cite{BN22}, which partially served as motivation for this paper.

\medskip

The Weyl module $\smash{L_{-2+1/p}(r)}=\smash{L^-_{-2+1/p}(-r)=V^{-2+1/p}(r)}$ (see \autoref{sec:affine}) is, by the Poincaré-Birkhoff-Witt theorem, as a vector space of the form
\begin{equation*}
V^{-2+1/p}(r)\cong U(\smash{\hat{\sl}}_2^{<0})\otimes U_r
\end{equation*}
where $\smash{\hat{\sl}}_2^{<0}=\sl_2\otimes t^{-1}\C[t^{-1}]\subseteq\hat\sl_2$ and $U_r$ denotes the irreducible $(r+1)$-dimensional $\sl_2$-module of highest weight $r\Lambda_1$. In other words, the Weyl module $V^{-2+1/p}(r)$ has a basis consisting of the vectors
\begin{equation*}
\prod_{i=1}^\infty(e_{-i})^{k_i}(h_{-i})^{l_i}(f_{-i})^{m_i}v_s,
\end{equation*}
where $k=(k_i)_{i=1}^\infty$, $l=(l_i)_{i=1}^\infty$ and $m=(m_i)_{i=1}^\infty$ are finite sequences with values in $\N$ and $s\in\{-r,-r+2,\ldots,r\}$ labels the vectors $v_s$ of $h_0$-weight $s$ (and $L_0$-weight $(pr^2+2pr)/4$) that span the $(r+1)$-dimensional, irreducible $\sl_2$-module $U_r$ from which $V^{-2+1/p}(r)$ is induced.

On the other hand, as it is the quotient of a Verma module by one singular vector of $L_0$-weight $h_{r+1,1}^{1,3p}+(r+1)$, the Virasoro module $M_{r+1,1;3p}$ has a basis consisting of vectors of the form
\begin{equation*}
\prod_{i=2}^\infty(L_{-i})^{m_i} L_{-1}^s v_{h_{r+1,1}^{1,3p}},
\end{equation*}
where $m=(m_i)_{i=2}^\infty$ is a finite sequence with values in $\N$, $s\in\{0,1,\ldots,r\}$ and $\smash{v_{h_{r+1,1}^{1,3p}}}$ is the vector of $L_0$-weight $\smash{h_{r+1,1}^{1,3p}=r(-2+3p(2+r))/4}$ from which the Verma module is induced.

The linear maps $\psi^\pm\colon V^{-2+1/p}(r)\overset{\sim}{\longrightarrow}M_{r+1,1;3p}$ in \eqref{eq:psiplusminus} are then defined by
\begin{equation*}
\prod_{i=1}^\infty(e_{-i})^{k_i}(h_{-i})^{l_i}(f_{-i})^{m_i}v_s\mapsto\prod_{i=1}^\infty(L_{-3i\pm1})^{k_i}(L_{-3i})^{l_i}(L_{-3i\mp1})^{m_i}L_{-1}^{(r\mp s)/2} v_{h_{r+1,1}^{1,3p}},
\end{equation*}
i.e.\ by mapping
\begin{equation*}
e_{-i}\mapsto L_{-3i\pm1},\quad h_{-i}\mapsto L_{-3i},\quad f_{-i}\mapsto L_{-3i\mp1}\quad\text{and}\quad v_s\mapsto L_{-1}^{(r\mp s)/2} v_{h_{r+1,1}^{1,3p}}
\end{equation*}
for $i\in\Ns$ and $s\in\{-r,-r+2,\ldots,r\}$. As we explicitly described a basis of the domain and the codomain of the maps $\psi^\pm$, it follows:
\begin{prop}\label{prop:gvsi}
Let $p\in\Ns$. The maps $\psi^\pm\colon L_{-2+1/p}(\mu_{r,0})\overset{\sim}{\longrightarrow}M_{r+1,1;3p}$ in \eqref{eq:psiplusminus} for $r\in\N$ are vector-space isomorphisms that map a homogeneous vector of $h_0$-weight $f$ and $L_0$-weight~$n$ to a vector of $L_0$-weight $n'=\mp f/2+3n$.
\end{prop}
This gives another proof of \autoref{cor:charid2} and \autoref{cor:lowestcharid2}.

%%%%%%%%%%%%%%%%%%%%%%%%%%%%%%%

\subsection{Universal \VOA{}s}\label{sec:universal}

As a digression, we note that the vector-space isomorphisms $\psi^\pm$ constructed in the previous section can be applied verbatim to obtain a correspondence between the universal affine \voa{} $V^k(\sl_2)$ (i.e.\ the trivial Weyl module) and the universal Virasoro \voa{} $V_\mathrm{Vir}(c,0)$. This holds at any level $k\in\C\setminus\{-2\}$ and central charge $c\in\C$; the only relevant point is that they both can be obtained from Verma modules by taking the quotient by one singular vector, namely the one belonging to family~$\raisebox{.5pt}{\textcircled{\raisebox{-.9pt}{1}}}$.

For the $V^k(\sl_2)$-Verma module $M_k(0)$ of highest weight $0\Lambda_1$, the singular vector~$\raisebox{.5pt}{\textcircled{\raisebox{-.9pt}{1}}}$ exists and has $(h_0,L_0)$-weight $(-2,0)$, independent of the level~$k$; or weight $(2,0)$ if we consider the lowest-weight Verma module $\smash{M^-_k(0)}$ instead. Similarly, the $V_\mathrm{Vir}(c,0)$-Verma module $M_\mathrm{Vir}(c,0)$ has the singular vector~$\raisebox{.5pt}{\textcircled{\raisebox{-.9pt}{1}}}$ of $L_0$-weight $1$, independent of the central charge~$c$. Hence, arguing as before, we obtain on the level of the characters:
\begin{prop}\label{prop:chariduniversal}
For all $k\in\C\setminus\{-2\}$ and $c\in\C$ the character identities
\begin{equation*}
\ch^*_{V^k(\sl_2)}(q^{\mp1/2},q^3)=\ch^*_{V_\mathrm{Vir}(c,0)}(q)
\end{equation*}
for the universal \voa{}s hold.
\end{prop}

This result is equivalent to the existence (not uniquely) of graded vector-space isomorphisms
\begin{equation}\label{eq:psiplusminus2}
\psi^\pm\colon V^k(\sl_2)\overset{\sim}{\longrightarrow}V_\mathrm{Vir}(c,0),
\end{equation}
which we construct in the following.

The universal affine \voa{} $V^k(\sl_2)$ is the trivial Weyl module $V^k(0)$ for $\smash{\hat\sl_2}$. Specialising the discussion in \autoref{sec:proof} to $r=0$, we recall that it is as a vector space of the form
\begin{equation*}
V^k(\sl_2)\cong U(\smash{\hat{\sl}}_2^{<0})\otimes U_0
\end{equation*}
where $U_0$ denotes the irreducible $1$-dimensional $\sl_2$-module of highest weight $0\Lambda_1$. In other words, it has a basis consisting of the vectors $\prod_{i=1}^\infty(e_{-i})^{k_i}(h_{-i})^{l_i}(f_{-i})^{m_i}\vac$, where $k=(k_i)_{i=1}^\infty$, $l=(l_i)_{i=1}^\infty$ and $m=(m_i)_{i=1}^\infty$ are finite sequences with values in $\N$ and the vacuum vector $\vac$ is the unique vector with $h_0$-weight $0$ and $L_0$-weight~$0$ that spans the $1$-dimensional, irreducible $\sl_2$-module $U_0$ from which $V^k(\sl_2)$ is induced.

On the other hand, the universal Virasoro \voa{} $V_\mathrm{Vir}(c,0)$ is a quotient of the Verma module $M_\mathrm{Vir}(c,0)$ by the singular vector $L_{-1}\vac$ of $L_0$-weight~$1$. $V_\mathrm{Vir}(c,0)$ has a basis consisting of vectors of the form $\prod_{i=2}^\infty(L_{-i})^{m_i}\vac$, where $m=(m_i)_{i=2}^\infty$ is a finite sequence with values in $\N$ and the vacuum vector~$\vac$ is the unique vector of $L_0$-weight $0$ from which the Verma module is induced.

We can then define the linear maps $\psi^\pm\colon V^k(\sl_2)\overset{\sim}{\longrightarrow}V_\mathrm{Vir}(c,0)$ in \eqref{eq:psiplusminus2} via
\begin{equation*}
\prod_{i=1}^\infty(e_{-i})^{k_i}(h_{-i})^{l_i}(f_{-i})^{m_i}\vac\mapsto\prod_{i=1}^\infty(L_{-3i\pm1})^{k_i}(L_{-3i})^{l_i}(L_{-3i\mp1})^{m_i}\vac,
\end{equation*}
i.e.\ by mapping $e_{-i}\mapsto L_{-3i\pm1}$, $h_{-i}\mapsto L_{-3i}$, $f_{-i}\mapsto L_{-3i\mp1}$ for $i\in\Ns$ and $\vac\mapsto\vac$. We have shown:
\begin{prop}\label{prop:gvsiuniversal}
For $k\in\C\setminus\{-2\}$, $c\in\C$ the maps $\psi^\pm\colon V^k(\sl_2)\overset{\sim}{\longrightarrow}V_\mathrm{Vir}(c,0)$ in \eqref{eq:psiplusminus2} are vector-space isomorphisms that map a homogeneous vector of $h_0$-weight~$f$ and $L_0$-weight $n$ to a vector of $L_0$-weight $n'=\mp f/2+3n$, respectively.
\end{prop}

%%%%%%%%%%%%%%%%%%%%%%%%%%%%%%%
%%%%%%%%%%%%%%%%%%%%%%%%%%%%%%%

\section{Integral Case \texorpdfstring{$p=1$}{p=1}}\label{sec:integral}

In the following, we specialise the character identities in the case of admissible levels $k=-2+q/p$ (see \autoref{sec:admissible}) to the integral case $p=1$. Here, we shall see that they induce a Galois conjugation between the representation categories; and for small values of $q$, the characters are also related by certain Hecke operators. Further specialising to $q=4$, we give an explicit graded vector-space isomorphism between the irreducible modules by embedding them into free-fermion \svoa{}s.

%%%%%%%%%%%%%%%%%%%%%%%%%%%%%%%

\subsection{Galois Conjugation and Hecke Operators}\label{sec:galois}

Consider the simple affine \voa{} $L_{-2+q}(\sl_2)$ for $q\geq2$, with central charge $c=3-6/q$. It is \strat{} and the exactly $q-1$ irreducible modules are $L_{-2+q}(\mu_{r,0})$ for $r\in\Z$ with $0\leq r\leq q-2$. Their lowest $L_0$-weights are given by
\begin{equation*}
\ell_{r,0}^{q,1}=\frac{r(r+2)}{4q}
\end{equation*}
and we recall that the fusion rules are
\begin{equation*}
L_{-2+q}(\mu_{r_1,0})\boxtimes L_{-2+q}(\mu_{r_2,0})\cong\bigoplus_{\substack{r_3=|r_1-r_2|\\r_3\in|r_1-r_2|+2\Z}}^{\min(r_1+r_2,-4+2q-r_1-r_2)}L_{-2+q}(\mu_{r_3,0}).
\end{equation*}

On the other hand, the \strat{} $L_\mathrm{Vir}(c_{q,3},0)$ with $3\nmid q$ has central charge $c=13-18/q-2q$ and the exactly $q-1$ irreducible $L_\mathrm{Vir}(c_{q,3},0)$-modules are given by $\smash{L_\mathrm{Vir}(c_{q,3},h_{r+1,1}^{q,3})}$ for $1\leq r+1\leq q-1$. Their lowest $L_0$-weights are
\begin{equation*}
h_{r+1,1}^{q,3}=\frac{(6-2q+3r)r}{4q},
\end{equation*}
and the fusion rules are
\begin{equation*}
L_\mathrm{Vir}(c_{q,3},h_{r_1+1,1}^{q,3})\boxtimes L_\mathrm{Vir}(c_{q,3},h_{r_2+1,1}^{q,3})\cong\bigoplus_{\substack{r_3=|r_1-r_2|\\r_3\in|r_1-r_2|+2\Z}}^{\min(r_1+r_2,-4+2q-r_1-r_2)}L_\mathrm{Vir}(c_{q,3},h_{r_3+1,1}^{q,3}).
\end{equation*}

\medskip

Recall from \autoref{cor:charid} and \autoref{cor:lowestcharid} that for $3\nmid q$ the vector-space isomorphisms $\varphi^\pm$ in \eqref{eq:phiplus} and \eqref{eq:phiminus} induce the character identities
\begin{equation*}
\ch^*_{L_{-2+q}(\mu_{r,0})}(q^{\mp1/2},q^3)=\ch^*_{L_\mathrm{Vir}(c_{q,3},h_{r+1,1}^{q,3})}(q)
\end{equation*}
for $r\in\Z$ with $0\leq r\leq q-2$. This defines a bijection between the irreducible modules on both sides, which on the level of indices is given by
\begin{equation*}
(q,1)\mapsto(q,3),\quad(r,0)\mapsto(r+1,1).
\end{equation*}

Comparing the fusion rules, we observe:
\begin{prop}\label{prop:fusionrings}
Let $q\in\Z$ with $q\geq 2$ and $3\nmid q$. Then the correspondence $L_{-2+q}(\mu_{r,0})\mapsto L_\mathrm{Vir}(c_{q,3},h_{r+1,1}^{q,3})$ for $r\in\Z$ with $0\leq r\leq q-2$, stemming from the vector-space isomorphisms $\varphi^{\pm}$ for Verma modules, defines an isomorphism of fusion rings of the modular tensor categories $\Rep(L_{-2+q}(\sl_2))$ and $\Rep(L_\mathrm{Vir}(c_{q,3},0))$.
\end{prop}
For $q=2$, both \voa{}s are the trivial \voa{} $\C\vac$ of central charge~$0$ and the modular tensor categories are just $\operatorname{Vect}$.

For $q>2$, it is apparent that $\Rep(L_{-2+q}(\sl_2))$ and $\Rep(L_\mathrm{Vir}(c_{q,3},0))$ are inequivalent ribbon fusion categories. Indeed, the ribbon twists are different:
\begin{equation*}
h_{r+1,1}^{q,3}-\ell_{r,0}^{q,1}=\frac{(2-q+r)r}{2q}
\end{equation*}
vanishes if and only if $r=0$ or $r=q-2$, i.e.\ for the extreme values of $r$.

$\Rep(L_{-2+q}(\sl_2))$ and $\Rep(L_\mathrm{Vir}(c_{q,3},0))$ are most likely not even equivalent as fusion categories (i.e.\ ignoring the ribbon structure). Indeed, while $\Rep(L_{-2+q}(\sl_2))$ is always pseudounitary, $\Rep(L_\mathrm{Vir}(c_{q,3},0))$ is probably not (except for $q=2,4$) as these are minimal models away from the discrete series.

For $q=4$, both $\Rep(L_2(\sl_2))$ and $\Rep(L_\mathrm{Vir}(c_{4,3},0))$ are pseudounitary fusion categories. They have the well-known Ising fusion rules. It follows from \cite{TY98} and Proposition~8.32 in \cite{ENO05} that there are exactly two nonequivalent fusion categories with Ising fusion rules (and both of them are pseudounitary). They can be distinguished by the Frobenius-Schur indicator of the simple object with dimension $\sqrt{2}$. One can show that the value of the Frobenius-Schur indicator for $\Rep(L_2(\sl_2))$ is $-1$, while for $\Rep(L_\mathrm{Vir}(c_{4,3},0))$ it is $+1$. We shall discuss the case $q=4$ further in \autoref{sec:freefermion} below.

\medskip

We want to make more precise the relation between the representation categories of the \strat{} \voa{}s $L_{-2+q}(\sl_2)$ and $L_\mathrm{Vir}(c_{q,3},0)$.

We briefly recall the notion of \emph{Galois symmetry} \cite{BG91,CG94} (see also \cite{RSW09,DLN15}). Given a modular tensor category, one can show that certain categorical data (in particular, the fusion matrices and the $S$-matrix), if normalised appropriately, have entries in a Galois extension $K$ of $\Q$, with abelian Galois group $\Aut(K/\Q)$. Roughly speaking, by acting on these data with an element of the Galois group, we obtain a new, possibly different, modular tensor category with the same fusion rules. For example, the Fibonacci and Yang-Lee modular tensor categories $\Rep(L_1(G_2))$ and $\Rep(L_\mathrm{Vir}(c_{2,5},0))$, respectively, are Galois conjugates. In particular, Galois symmetry may relate pseudounitary and nonpseudounitary modular tensor categories.

Most likely, the following is true:
\begin{conj}\label{conj:galois}
For $q\in\Z$ with $q\geq 2$ and $3\nmid q$, the modular tensor categories $\Rep(L_{-2+q}(\sl_2))$ and $\Rep(L_\mathrm{Vir}(c_{q,3},0))$ are Galois conjugates.
\end{conj}
The assertion is true when $q$ is odd, as noticed in Section~6 of \cite{Gan03}. There, the former category is called a unitarisation of the latter. It should not be difficult to verify the assertion in general.

\medskip

We suspect that there is an even deeper relation between the representation theories of $L_{-2+q}(\sl_2)$ and $L_\mathrm{Vir}(c_{q,3},0)$. We thank Jeffrey Harvey and Kaiwen Sun for related discussions.
\begin{rem}\label{rem:Hecke}
Let $q\in\Z$ with $q\geq 2$ and $3\nmid q$. When $q\leq 8$, the vector-valued characters of $L_{-2+q}(\sl_2)$ and $L_\mathrm{Vir}(c_{q,3},0)$ are related by a certain Hecke operator $T_3$ if $3$ divides the conductor of $L_\mathrm{Vir}(c_{q,3},0)$, and by another Hecke operator $T_p$ or a generalised Hecke operator $T_3$ otherwise. This is trivially true for $q=2$, shown in \cite{DLS22} for $q=4$ (using a generalised Hecke operator $T_3$), in \cite{HHW20,Wu20} for $q=5$, in \cite{Wu20} for $q=7$ (but using $T_{59}$) and in \cite{Wu20,DLS22} for $q=8$.
\end{rem}

It would be interesting to investigate this phenomenon further, which we leave to future work. Moreover, one should explain the precise relation to the character identities in \autoref{cor:charid} and \autoref{cor:lowestcharid}, which stem from the vector-space isomorphisms $\varphi^{\pm}$ in \eqref{eq:phiplus} and \eqref{eq:phiminus} on the level of Verma modules.

\medskip

We finish this section with some remarks on the role of the conductor~$N$ and the prime~$p$ in the Hecke operator $T_p$; see \cite{Gan03,HHW20,DLS22} for details. First recall that the central charge of $L_{-2+q}(\sl_2)$ is
\begin{equation*}
c=3(q-2)/q.
\end{equation*}
On the other hand, the \emph{effective} central charge of $L_\mathrm{Vir}(c_{q,3},0)$ is
\begin{equation*}
c_{q,3}^\mathrm{eff}=c_{q,3}-24h_\mathrm{min}=(q-2)/q,
\end{equation*}
which is exactly $1/3$ of the central charge of $L_{-2+q}(\sl_2)$. This suggests that these two theories could be connected by a Hecke operator $T_3$, at least as long as $3$ does not divide the conductor $N$ of $L_\mathrm{Vir}(c_{q,3},0)$.

In general, the conductor of a \strat{} \voa{} is the order of the $T$-matrix \cite{Ban03,DLN15}. For the simple Virasoro \voa{}s $L_\mathrm{Vir}(c_{q,p},0)$, the conductor $N$ is hence the smallest $N\in\Ns$ such that
\begin{equation*}
N\Bigl(h_{r,s}^{q,p}-\frac{c_{q,p}}{24}\Bigr)=N\Bigl(\frac{(qs-pr)^2}{4pq}-\frac{1}{24}\Bigr)\in\Z
\end{equation*}
for all $r,s\in\Z$ with $1\leq r\leq q-1$ and $1\leq s\leq p-1$.

For the $(q,3)$-minimal models, one can show that $3\mid N$ for $q\equiv 1\pmod{3}$ and $3\nmid N$ for $q\equiv 2\pmod{3}$. So, in the latter case, the relation in \autoref{rem:Hecke} can hold using the Hecke operator $T_3$, while it needs to be modified in the former case.

%%%%%%%%%%%%%%%%%%%%%%%%%%%%%%%

\subsection{Special Case \texorpdfstring{$q=4$}{q=4}: Free Fermions}\label{sec:freefermion}

We consider the case $(q,p)=(4,1)$ in more detail. Our character identities in \autoref{cor:charid} and \autoref{cor:lowestcharid} relate the irreducible modules for $L_2(\sl_2)$, with central charge $c=3/2$, to the irreducible modules for $L_\mathrm{Vir}(c_{4,3},0)$ at central charge $c_{4,3}=1/2$ as follows:
\begin{alignat*}{3}
L_2(\mu_{0,0}),&\;\ell_{0,0}^{4,1}=0&&\;\mapsto\;&L_\mathrm{Vir}(c_{4,3},h_{1,1}^{4,3}),&\;h_{1,1}^{4,3}=0,\\
L_2(\mu_{1,0}),&\;\ell_{1,0}^{4,1}=3/16&&\;\mapsto\;&L_\mathrm{Vir}(c_{4,3},h_{2,1}^{4,3}),&\;h_{2,1}^{4,3}=1/16,\\
L_2(\mu_{2,0}),&\;\ell_{2,0}^{4,1}=1/2&&\;\mapsto\;&L_\mathrm{Vir}(c_{4,3},h_{3,1}^{4,3}),&\;h_{3,1}^{4,3}=1/2.
\end{alignat*}
Recall (see, e.g., \cite{Hoe95,HM23}) that the \svoa{}s of three and one free fermions (with the standard choices of conformal structures such that they have central charge $c=3/2$ and $c=1/2$) decompose as
\begin{align*}
F^{3/2}&=(F^{1/2})^{\otimes3}\cong L_2(\mu_{0,0})\oplus L_2(\mu_{2,0}),\\
F^{1/2}&\cong L_\mathrm{Vir}(c_{4,3},h_{1,1}^{4,3})\oplus L_\mathrm{Vir}(c_{4,3},h_{3,1}^{4,3}),
\end{align*}
respectively, into the even and odd part. Hence, we obtain a character correspondence between $\smash{F^{3/2}}$ and $\smash{F^{1/2}}$:
\begin{prop}\label{prop:charidFF}
The vector-space isomorphisms $\varphi^\pm$ in \eqref{eq:phiplus} and \eqref{eq:phiminus} on the level of Verma modules induce the (super)character identities
\begin{equation*}
\operatorname{(s)ch}^*_{F^{3/2}}(q^{\mp1/2},q^3)=\operatorname{(s)ch}^*_{F^{1/2}}(q),
\end{equation*}
which restrict to the character identities in \autoref{cor:charid} and \autoref{cor:lowestcharid} for $(q,p)=(4,1)$, respectively.
\end{prop}

The proposition is equivalent to the existence (nonuniquely) of graded vector-space isomorphisms $\smash{\chi^\pm\colon F^{3/2}\overset{\sim}{\longrightarrow}F^{1/2}}$ that map a homogeneous vector of $h_0$-weight $f$ and $L_0$-weight $n$ to one of $L_0$-weight $n'=\mp f/2+3n$. They restrict to the vector-space isomorphisms $\chi^\pm$ in \eqref{eq:chiplusminus}. In the following, we construct such isomorphisms explicitly; cf.\ \autoref{rem:inducedGVSI}.

\medskip

The free-fermion \svoa{} $F^{1/2}$ is strongly generated by one odd field $\Phi(z)=Y(\Phi,z)=\sum_{i\in\Z}\Phi_iz^{-i-1}$ whose modes satisfy the anti-commutation relations
\begin{equation*}
[\Phi_i,\Phi_j]=\delta_{i+j,-1}
\end{equation*}
for $i,j\in\Z$. As a vector space, $F^{1/2}$ is the exterior algebra
\begin{equation*}
F^{1/2}\cong\Lambda(\{\Phi_{-i}\mid i\in\Ns\}).
\end{equation*}

The \svoa{} $F^{3/2}$ of three free fermions is $(F^{1/2})^{\otimes3}$. Hence, $F^{3/2}$ is strongly generated by the odd fields $\smash{\Phi^{(r)}(z)=Y(\smash{\Phi^{(r)}},z)=\sum_{i\in\Z}\Phi^{(r)}_iz^{-i-1}}$ for $r=0,1,2$ whose modes satisfy the anti-commutation relations
\begin{equation*}
[\Phi^{(r)}_i,\Phi^{(s)}_j]=\delta_{r,s}\delta_{i+j,-1}
\end{equation*}
for $r,s=0,1,2$ and $i,j\in\Z$. As a vector space, $F^{3/2}$ is the exterior algebra
\begin{equation*}
F^{3/2}\cong\Lambda(\{\Phi^{(0)}_{-i},\Phi^{(1)}_{-i},\Phi^{(2)}_{-i}\mid i\in\Ns\}).
\end{equation*}
It is convenient to choose a slightly different basis, namely
\begin{equation*}
\Psi^+\coloneqq(\Phi^{(1)}+i\Phi^{(2)})/\sqrt{2},\quad
\Psi^-\coloneqq(\Phi^{(1)}-i\Phi^{(2)})/\sqrt{2},\quad\Phi^{(0)}\coloneqq\Phi^{(0)}.
\end{equation*}
Then these odd fields satisfy the anti-commutation relations
\begin{equation*}
[\Psi^+_i,\Psi^-_j]=\delta_{i+j,-1}\quad\text{and}\quad[\Phi^{(0)}_i,\Phi^{(0)}_j]=\delta_{i+j,-1}
\end{equation*}
for all $i,j\in\Z$ and all other anti-commutators are zero. As a vector space, $F^{3/2}$ is then the exterior algebra
\begin{equation*}
F^{3/2}\cong\Lambda(\{\Psi^+_{-i},\Psi^-_{-i},\Phi^{(0)}_{-i}\mid i\in\Ns\}).
\end{equation*}
In order to fix the $\sl_2$-weight grading, we need to explicitly realise $L_2(\sl_2)\cong(F^{3/2})^{\bar0}$ inside $F^{3/2}$. The even subalgebra $(F^{3/2})^{\bar0}$ is strongly generated by, e.g., the fields
\begin{equation}\label{eq:sl2triple}
\begin{split}
e(z)&=\i\sqrt{2}{:}\Phi^{(0)}(z)\Psi^+(z){:},\\
f(z)&=\i\sqrt{2}{:}\Phi^{(0)}(z)\Psi^-(z){:},\\
h(z)&=2{:}\Psi^+(z)\Psi^-(z){:},
\end{split}
\end{equation}
which satisfy the standard operator product expansions for the $\sl_2$-triple $\{e,f,h\}$ at level~$2$. One can show that they generate the simple quotient $L_2(\sl_2)$. Relevant for us is in the following that, if we extend the action of $h_0$ from $L_2(\sl_2)$ to all of $F^{3/2}$, then $\Psi^\pm$ has $h_0$-weight $\pm2$ and $\Phi^{(0)}$ has weight~$0$.

Similarly to \autoref{sec:universal}, we now define the graded vector-space isomorphisms
\begin{equation}\label{eq:chifermion}
\chi^\pm\colon F^{3/2}\overset{\sim}{\longrightarrow}F^{1/2}
\end{equation}
by mapping
\begin{equation*}
\Psi^+_{-i}\mapsto\Phi_{-3i+1\pm1},\quad\Psi^-_{-i}\mapsto\Phi_{-3i+1\mp1},\quad\Phi^{(0)}_{-i}\mapsto\Phi_{-3i+1}
\end{equation*}
for $i\in\Ns$ and $\vac\mapsto\vac$. Then, indeed, a vector of $(h_0,L_0)$-weight $(f,n)$ is mapped to a vector of $L_0$-weight $n'=\mp f/2+3n$.

We can strengthen the above statement as follows, noting that the maps $\chi^\pm$ equip $F^{1/2}$ with the structure of a module for $F^{3/2}$ (cf.\ \cite{AW23}):
\begin{prop}\label{prop:FFcorrespondence}
The \svoa{} $F^{1/2}$ has the structure of an irreducible $F^{3/2}$-module via the module map $\smash{Y_{F^{1/2}}(\cdot,z)\colon F^{3/2}\to\End(F^{1/2})[[z^{\pm1}]]}$,
\begin{align*}
Y_{F^{1/2}}(\Psi^+,z)&=\sum_{i\in\Z}\Phi_{3i+1\pm1}z^{-i-1},\\
Y_{F^{1/2}}(\Psi^-,z)&=\sum_{i\in\Z}\Phi_{3i+1\mp1}z^{-i-1},\\
Y_{F^{1/2}}(\Phi^{(0)},z)&=\sum_{i\in\Z}\Phi_{3i+1}z^{-i-1},
\end{align*}
and the map $\smash{\chi^\pm\colon F^{3/2}\overset{\sim}{\longrightarrow}F^{1/2}}$ in \eqref{eq:chifermion} is an isomorphism of $\smash{F^{3/2}}$-modules.

Conversely, the \svoa{} $F^{3/2}$ is an irreducible $F^{1/2}$-module via the module vertex operator $\smash{Y_{F^{3/2}}(\cdot,z)\colon F^{1/2}\to\End(F^{3/2})[[z^{\pm1}]]}$,
\begin{equation*}
Y_{F^{3/2}}(\Phi,z)=\sum_{i\in\Z}\bigl(\Psi^+_iz^{-3i-2\mp1}+\Psi^-_iz^{-3i-2\pm1}+\Phi^{(0)}_iz^{-3i-2}\bigr),
\end{equation*}
and the map $\smash{(\chi^\pm)^{-1}\colon F^{1/2}\overset{\sim}{\longrightarrow}F^{3/2}}$ in \eqref{eq:chifermion} is an isomorphism of $\smash{F^{1/2}}$-modules.
\end{prop}

Restricting the result to the even parts of $F^{1/2}$ and $F^{3/2}$ we obtain:
\begin{cor}\label{cor:dual}
The \voa{} $L_\mathrm{Vir}(c_{4,3},0)$ has the structure of an irreducible $L_2(\sl_2)$-module, which is obtained by setting
\begin{equation*}
Y_{L_\mathrm{Vir}(c_{4,3},0)}(u,z)=Y_{F^{1/2}}(u,z)\quad\text{for } u\in(F^{3/2})^{\bar0}=L_2(\sl_2).
\end{equation*}
There is an isomorphism of $L_2(\sl_2)$-modules $\smash{\chi^\pm\colon L_2(\sl_2)\overset{\sim}{\longrightarrow}L_\mathrm{Vir}(c_{4,3},0)}$, which is the even part of the isomorphism $\smash{\chi^\pm\colon F^{3/2}\overset{\sim}{\longrightarrow}F^{1/2}}$ in \eqref{eq:chifermion}.

Conversely, the \voa{} $L_2(\sl_2)$ has the structure of an irreducible $L_\mathrm{Vir}(c_{4,3},0)$-module, which is obtained by setting
\begin{equation*}
Y_{L_2(\sl_2)}(u,z)=Y_{F^{3/2}}(u,z)\quad\text{for } u\in(F^{1/2})^{\bar0}=L_\mathrm{Vir}(c_{4,3},0).
\end{equation*}
There is an isomorphism of $\smash{L_\mathrm{Vir}(c_{4,3},0)}$-modules $\smash{(\chi^\pm)^{-1}\colon L_\mathrm{Vir}(c_{4,3},0)\overset{\sim}{\longrightarrow}L_2(\sl_2)}$, which is the even part of the isomorphism $\smash{(\chi^\pm)^{-1}\colon F^{1/2}\overset{\sim}{\longrightarrow}F^{3/2}}$ in \eqref{eq:chifermion}.
\end{cor}
These results solve the problem of \autoref{rem:inducedGVSI} in the special case of $(q,p)=(4,1)$ and for the irreducible modules with $r=0,q-2$. The question is whether there are further admissible cases where vector-space isomorphisms $\chi^\pm$ can be defined.

\begin{rem}\label{rem:AWduality}
In \cite{AW23}, the authors introduced a new kind of duality between simple vertex (super)algebras $V_1$ and $V_2$. We say that $V_1$ and $V_2$ are \emph{dual} to each other if $V_2$ can be equipped with the structure of a $V_1$-module isomorphic to $V_1$ and vice versa. Moreover, we demand that $V_1$ can be realised as a subalgebra of the vertex algebra of local fields acting on $V_2$-modules and vice versa.

The example studied in \cite{AW23} is the duality between the vertex superalgebra $\smash{\mathcal{V}^{(2)}}$ (see also \autoref{sec:gvsi}) and the affine vertex superalgebra $\smash{L_{-h^{\vee}}(\mathfrak{osp}_{1,2})}$ at the critical level. \autoref{cor:dual} shows that $\smash{L_2(\sl_2)}$ and $\smash{L_\mathrm{Vir}(c_{4,3},0)}$ are dual in the same sense.
\end{rem}

We thank Ching Hung Lam for pointing the following out to us:
\begin{rem}\label{rem:twisted}
\autoref{prop:FFcorrespondence} and \autoref{cor:dual} are closely related to the Barron-Dong-Mason construction \cite{BDM02}, or more precisely to its super version \cite{Bar16}, which provides a construction of $g$-twisted modules for tensor-product vertex operator (super)algebras $V^{\otimes k}$ for an automorphism $g$ given by a permutation of the tensor factors. In particular, when $g=(12\cdots k)$ is a $k$-cycle, the category of $g$-twisted $V^{\otimes k}$-modules is equivalent to the category of untwisted $V$-modules.

Applying this to $V=F^{1/2}$ and $k=3$, one obtains a relation between untwisted $F^{1/2}$-modules and $g$-twisted $F^{3/2}$-modules, where $\smash{F^{3/2}\cong(F^{1/2})^{\otimes 3}}$. We note that \cite{Bar16} is applicable because $g=(123)$ has odd order. This permutation automorphism can also be written as
\begin{equation*}
g=\e^{2\pi\i H_0}\qquad\text{with}\qquad H=\frac{\i}{3\sqrt{3}}\bigl(\Phi^{(0)}_{-1}\Phi^{(1)}+\Phi^{(1)}_{-1}\Phi^{(2)}+\Phi^{(2)}_{-1}\Phi^{(0)}\bigr),
\end{equation*}
which means that it is inner. Hence, the category of $g$-twisted $F^{3/2}$-modules is also equivalent to the category of untwisted $F^{3/2}$-modules via Li's $\Delta$-operator construction \cite{Li96}.

Overall, we thus arrive at the same type of conclusion as \autoref{prop:FFcorrespondence}. However, our approach gives a direct correspondence between untwisted $F^{1/2}$- and $F^{3/2}$-modules, realised by an explicit re-indexing of fermionic modes modulo~$3$ rather than by combining two different constructions of twisted modules.

\smallskip

Since $g$ preserves the even subalgebra $\smash{(F^{3/2})^{\bar{0}}\cong L_2(\sl_2)}$, it also induces an automorphism $\smash{\bar{g}\coloneqq g|_{L_2(\sl_2)}}$ of $L_2(\sl_2)$. For instance, with the choice of affine generators $e,f,h$ in \eqref{eq:sl2triple}, this inner automorphism is given by
\begin{equation*}
\bar{g}=\e^{2\pi\i H_0}\qquad\text{with}\qquad H=\frac{\i}{6\sqrt{3}}\bigl((1+\i)e+(1-\i)f-h\bigr).
\end{equation*}
Thus, the permutation construction from \cite{BDM02,Bar16} restricted to the even subalgebras, relates $L_{\mathrm{Vir}}(c_{4,3},0)$-modules to $\bar{g}$-twisted $L_2(\sl_2)$-modules; and because $\bar{g}$ is an inner automorphism of $L_2(\sl_2)$, the category of $\bar{g}$-twisted $L_2(\sl_2)$-modules is equivalent to the category of untwisted $L_2(\sl_2)$-modules via Li's $\Delta$-operator construction \cite{Li96}. We then arrive at the same kind of statement as \autoref{cor:dual}.
\end{rem}

\smallskip

It may be possible to extend the above results to any integral level $k=-2+q$ with $q\geq4$ and $3\nmid q$ (and $p=1$). To this end, we restrict the correspondence in \autoref{sec:galois} between the irreducible modules for $L_{-2+q}(\sl_2)$ and those for $L_\mathrm{Vir}(c_{q,3},0)$ to the first and last module:
\begin{alignat*}{3}
L_2(\mu_{0,0}),&\;\ell_{0,0}^{q,1}=0&&\;\mapsto\; &L_\mathrm{Vir}(c_{q,3},h_{1,1}^{q,3}),&\;h_{1,1}^{q,3}=0,\\
L_2(\mu_{q-2,0}),&\;\ell_{q-2,0}^{q,1}=(q-2)/4&&\;\mapsto\; &L_\mathrm{Vir}(c_{q,3},h_{q-1,1}^{q,3}),&\;h_{q-1,1}^{q,3}=(q-2)/4.
\end{alignat*}
These modules are simple currents, which allows us to define the following simple generalised \voa{}s (\aia{}s, to be precise) as simple-current extensions
\begin{align*}
A^{(q)}&\coloneqq L_{-2+q}(\sl_2)\oplus L_{-2+q}(\mu_{q-2,0}),\\
B^{(q)}&\coloneqq L_\mathrm{Vir}(c_{q,3},0)\oplus L_\mathrm{Vir}(c_{q,3},h_{q-1,1}^{q,3}).
\end{align*}
For $q\equiv 0\pmod{4}$ these are \svoa{}s, and for $q\equiv 2\pmod{4}$ these are ordinary \voa{}s.

Now, according to the character identities in \autoref{cor:charid} and \autoref{cor:lowestcharid}, there are (not uniquely) graded vector-space isomorphisms $\chi^\pm\colon A^{(q)}\overset{\sim}{\longrightarrow}B^{(q)}$. The question is whether they can be written in a nice way, as in the case of $q=4$.
\begin{conj}
Assume that $q\geq2$ with $3\nmid q$. Then the \voa{}s $L_{-2+q}(\sl_2)$ and $L_\mathrm{Vir}(c_{q,3},0)$ are dual in the sense of \cite{AW23} (see \autoref{rem:AWduality}), and similarly the generalised \voa{}s $A^{(q)}$ and $B^{(q)}$. In particular, $L_{-2+q}(\sl_2)$ can be equipped with the structure of an irreducible $L_\mathrm{Vir}(c_{q,3},0)$-module isomorphic to $L_\mathrm{Vir}(c_{q,3},0)$ and vice versa.
\end{conj}

%%%%%%%%%%%%%%%%%%%%%%%%%%%%%%%
%%%%%%%%%%%%%%%%%%%%%%%%%%%%%%%

\section{Near-Admissible Case \texorpdfstring{$q=1$}{q=1}}\label{sec:gvsi}

We return to the near-admissible case in \autoref{sec:char} and consider the character correspondence between irreducible modules for $\smash{L_{-2+1/p}(\sl_2)=V^{-2+1/p}(\sl_2)}$ and $\smash{L_\mathrm{Vir}(c_{1,3p},0)=V_\mathrm{Vir}(c_{1,3p},0)}$ for $p\in\Ns$. It turns out that it has a number of interesting applications, some related to the 4d/2d-correspondence \cite{BLLPRR15}.

%%%%%%%%%%%%%%%%%%%%%%%%%%%%%%%

\subsection{Extension to \texorpdfstring{$\mathcal{V}^{(p)}$}{V(p)} and \texorpdfstring{$\mathcal{A}^{(p)}$}{A(p)}}\label{sec:VpAp}

Let $p\in\Ns$. Recall that in \autoref{prop:gvsi} we established graded vector-space isomorphisms \eqref{eq:psiplusminus}
\begin{equation*}
\psi^\pm\colon V^{-2+1/p}(\mu_{r,0})\overset{\sim}{\longrightarrow}M_{r+1,1;3p}
\end{equation*}
for all $r\in\N$. In the case of $p=2$, $\psi^-$ was already given in \cite{BN22}. Here, $\smash{V^{-2+1/p}(\mu_{r,0})}$ for $r\in\N$ are the irreducible Weyl modules for $\smash{V^{-2+1/p}(\sl_2)}$ and the $\smash{M_{\tilde{r},\tilde{s};3p}= L_\mathrm{Vir}(c_{1,3p},h_{\tilde{r},\tilde{s}}^{1,3p})}$ are the irreducible non-Verma modules for the Virasoro \voa{} at central charge $c=c_{1,3p}$.

The irreducible modules appearing in the above correspondence appear in the decomposition of certain well-known (generalised) \voa{}s. First, recall that $\mathcal{V}^{(p)}$ for $p\in\Ns$ is a family of simple \aia{}s of central charge $c=3k/(k+2)=3-6p$, conformally extending the affine \voa{} $\smash{L_{-2+1/p}(\sl_2)=V^{-2+1/p}(\sl_2)}$ at level $k=-2+1/p$, considered in \cite{Ada16,ACGY21}. In particular, for $p\equiv2\pmod4$ this is a simple \svoa{} with ``correct statistics'', and for $p\equiv0\pmod4$ it is a simple \voa{}. $\mathcal{V}^{(p)}$ is obtained as kernel of a certain screening operator $\tilde{Q}$ inside the tensor product of a $\beta\gamma$-system and a rank-$1$ lattice \aia{} generated by a lattice vector of squared norm $p/2$. Importantly for us, they decompose as
\begin{equation*}
\mathcal{V}^{(p)}=\bigoplus_{r=0}^\infty(r+1)V^{-2+1/p}(\mu_{r,0})
\end{equation*}
into a direct sum of Weyl modules for $V^{-2+1/p}(\sl_2)$.

On the other hand,
\begin{equation*}
\mathcal{A}^{(p)}=\bigoplus_{r=0}^\infty(r+1)M_{r+1,1;p}
\end{equation*}
for $p\in\Z$ with $p\geq2$ is the doublet, a simple \aia{} of central charge $c=c_{1,p}=13-6p-6/p$ \cite{FFT11,FT11,AM13}, again defined as kernel of a certain screening operator in a lattice \aia{}. It is a simple $\Z$-graded \svoa{} for $p\equiv2\pmod4$ and a simple $(1/2)\Z$-graded \voa{} for $p\equiv0\pmod4$.

\medskip

It is an immediate consequence of these decompositions and \autoref{prop:gvsi} that there is also a graded vector-space isomorphism between $\mathcal{V}^{(p)}$ and $\mathcal{A}^{(3p)}$:
\begin{prop}\label{prop:gvsispecial}
Let $p\in\Ns$. The maps $\psi^\pm\colon L_{-2+1/p}(\mu_{r,0})\overset{\sim}{\longrightarrow}M_{r+1,1;3p}$ in \eqref{eq:psiplusminus} for all $r\in\N$ form vector-space isomorphisms
\begin{equation*}
\psi^\pm\colon\mathcal{V}^{(p)}\overset{\sim}{\longrightarrow}\mathcal{A}^{(3p)}
\end{equation*}
that map a homogeneous vector of $(h_0,L_0)$-weight $(f,n)$ to a vector of $L_0$-weight $n'=\mp f/2+3n$, respectively. Moreover, if $p\equiv2\pmod4$, then both sides are \svoa{}s and $\psi^\pm$ is a parity-preserving map.
\end{prop}
\begin{proof}
We only need to prove the last assertion. To this end, note that for $p\equiv2\pmod4$ the parity in $\mathcal{V}^{(p)}$ and $\mathcal{A}^{(p)}$ is simply the parity of $r$ in $L_{-2+1/p}(\mu_{r,0})$ and $M_{r+1,1;p}$, respectively.
\end{proof}
The special case of $p=2$ of this result is stated in \cite{BN22} for $\psi^-$. In the subsequent section, we shall study this case in more detail.

There is also the corresponding statement on the level of characters, which follows from \autoref{cor:charid2} and \autoref{cor:lowestcharid2}.
\begin{prop}\label{prop:charid2special}
Let $p\in\Ns$. Then the following character identities hold:
\begin{equation*}
\ch^*_{\mathcal{V}^{(p)}}(q^{\mp1/2},q^3)=\ch^*_{\mathcal{A}^{(3p)}}(q).
\end{equation*}
Moreover, if $p\equiv2\pmod4$, then both sides are \svoa{}s and also the supercharacters satisfy
\begin{equation*}
\sch^*_{\mathcal{V}^{(p)}}(q^{\mp1/2},q^3)=\sch^*_{\mathcal{A}^{(3p)}}(q).
\end{equation*}
\end{prop}

We remark that, in a more standard setting, the \aia{}s $\mathcal{V}^{(p)}$ and $\mathcal{A}^{(p)}$ for the same $p$ are related by quantum Hamiltonian reduction \cite{ACGY21}, but here we relate them for different $p$.

%%%%%%%%%%%%%%%%%%%%%%%%%%%%%%%

\subsection{Special Case \texorpdfstring{$p=2$}{p=2}}

We consider the special case of $p=2$. The \svoa{} $\mathcal{V}^{(2)}$ is the simple quotient $\smash{\mathcal{V}^{(2)}=L_{c=-9}^{\mathcal{N}=4}}$ of the small $\mathcal{N}=4$ superconformal algebra of central charge $c=-9$ \cite{Kac98,Ada16}, which appears in many interesting applications (see, e.g., \cite{BLLPRR15,BMR19,AKM23}).

In \cite{BN22}, the authors explicitly propose a graded vector-space isomorphism $\smash{\psi^-\colon\mathcal{V}^{(2)}\overset{\sim}{\longrightarrow}\mathcal{A}^{(6)}}$ in terms of the strong generators of these \svoa{}s. This isomorphism $\psi^-$ maps a vector in $\mathcal{V}^{(2)}$ of $(h_0,L_0)$-weight $(f,n)$ to a vector in $\mathcal{A}^{(6)}$ of $L_0$-weight $n'=f/2+3n$ and has a number of further nice properties. However, it is not completely clear to us that this $\psi^-$ would restrict exactly to the $\smash{\psi^-\colon V^{-3/2}(\mu_{r,0})\overset{\sim}{\longrightarrow}M_{r+1,1;6}}$ that we constructed above.

We remark that it should be possible to write down a vector-space isomorphism $\smash{\psi^+\colon\mathcal{V}^{(2)}\overset{\sim}{\longrightarrow}\mathcal{A}^{(6)}}$ in a similar way. Recall that the isomorphisms $\psi^+$ and $\psi^-$ are related by the inner automorphism $e\mapsto f$, $f\mapsto e$, $h\mapsto -h$ of $\smash{V^{-3/2}(\sl_2)}$. This automorphism extends to $\mathcal{V}^{(2)}$ by interchanging the roles of $G^-$ and $G^+$ and those of $\smash{\tilde{G}^-}$ and $\smash{\tilde{G}^+}$. This would allow us to define $\psi^+$ in terms of the strong generators of $\mathcal{V}^{(2)}$ and $\mathcal{A}^{(6)}$. This isomorphism $\psi^+$ then maps a homogeneous vector in $\mathcal{V}^{(2)}$ of $(h_0,L_0)$-weight $(f,n)$ to a vector in $\mathcal{A}^{(6)}$ of $L_0$-weight $n'=-f/2+3n$.

\medskip

We demonstrate the relation between the characters and supercharacters
\begin{equation*}
\operatorname{(s)ch}^*_{\mathcal{V}^{(2)}}(q^{\mp1/2},q^3)=\operatorname{(s)ch}^*_{\mathcal{A}^{(6)}}(q)
\end{equation*}
from \autoref{prop:charid2special} for $p=2$ concretely. The Jacobi (super)character of $\mathcal{V}^{(2)}$ is given by (see, e.g., \cite{AKM23})
{\allowdisplaybreaks
\begin{align*}
\ch_{\mathcal{V}^{(2)}}(w,q)&=\frac{\sum_{n=0}^\infty(n+1)(w^{n+1}-w^{-n-1})q^{(n+1)^2/2}}{q^{1/8}(q;q)_\infty(w^2q;q)_\infty(w^{-2}q;q)_\infty(w-w^{-1})}\\
&=\frac{w\partial_w(\sum_{n\in\Z}w^nq^{n^2/2})}{wq^{1/8}(q;q)_\infty(w^2q;q)_\infty(w^{-2};q)_\infty}\\
&=\frac{w\partial_w\vartheta_{00}(w,q)}{-\i\vartheta_{11}(w^2,q)},\\
\sch_{\mathcal{V}^{(2)}}(w,q)&=\frac{\sum_{n=0}^\infty(-1)^n(n+1)(w^{n+1}-w^{-n-1})q^{(n+1)^2/2}}{q^{1/8}(q;q)_\infty(w^2q;q)_\infty(w^{-2}q;q)_\infty(w-w^{-1})}\\
&=\frac{-w\partial_w(\sum_{n\in\Z}(-1)^nw^nq^{n^2/2})}{wq^{1/8}(q;q)_\infty(w^2q;q)_\infty(w^{-2};q)_\infty}\\
&=\frac{w\partial_w\vartheta_{01}(w,q)}{\i\vartheta_{11}(w^2,q)}.
\end{align*}
}%
On the other hand, the character and supercharacter of $\mathcal{A}^{(6)}$ are given by
\begin{align*}
\ch^*_{\mathcal{A}^{(6)}}(q)&=\frac{\sum_{n=0}^\infty(n+1)(1-q^{n+1})q^{(3n^2+5n)/2}}{(q;q)_\infty},\\
\sch^*_{\mathcal{A}^{(6)}}(q)&=\frac{\sum_{n=0}^\infty(n+1)(-1)^n(1-q^{n+1})q^{(3n^2+5n)/2}}{(q;q)_\infty},
\end{align*}
as can be seen from the decomposition of $\mathcal{A}^{(6)}$ into irreducible modules for the Virasoro \voa{} at central charge $c=-24$. It is not difficult to verify \autoref{prop:charid2special} directly now.

\medskip

We also sketch how the character identities between $\mathcal{V}^{(2)}$ and $\mathcal{A}^{(6)}$ can be extended to certain modules of these \svoa{}s, as suggested in \cite{PY25}. We explain how they arise as graded vector-space isomorphisms of free-field realisations compatible with certain nonsplit extensions.

Recall that $\mathcal{V}^{(2)}$ has two irreducible modules in the category $\mathcal{O}$, which are denoted in \cite{Ada16} by $\textsl{SC}\Lambda(1)=\mathcal{V}^{(2)}$ and $\textsl{SC}\Pi(1)$. The \svoa{} $\mathcal{V}^{(2)}$ is realised as a subalgebra of the \svoa{} $M\otimes F$, where $M$ is the Weyl \voa{} (or $\beta\gamma$-system) with central charge $c=2$ and $F$ the Clifford \svoa{} (or $bc$-system) with central charge $c=-11$. Moreover, $M\otimes F$ is a nonsplit extension of $\textsl{SC}\Lambda(1)$ and $\textsl{SC}\Pi(1)$,
\begin{equation*}
0\rightarrow \textsl{SC}\Lambda(1)\rightarrow M\otimes F\rightarrow \textsl{SC}\Pi(1)\rightarrow 0.\label{eq-ext-N4}
\end{equation*}

The doublet \svoa{} $\mathcal{A}^{(6)}$, on the other hand, is realised as a vertex subalgebra of the lattice \svoa{} $\smash{V_{\Z\gamma}}$ with bilinear form $\langle\gamma,\gamma\rangle=3$. It has three irreducible modules, denoted in \cite{AM13} by $\Lambda\Pi(1)=\mathcal{A}^{(6)}$, $\Lambda\Pi(3)$ and $\Lambda\Pi(5)$. Moreover, $\smash{V_{\Z\gamma}}$ is a nonsplit extension of $\Lambda\Pi(1)$ and $\Lambda\Pi(5)$,
\begin{equation*}
0\rightarrow\Lambda\Pi(1)\rightarrow V_{\Z\gamma}\rightarrow\Lambda\Pi(5)\rightarrow 0.\label{eq-ext-doublet}
\end{equation*}

One can show that there is a vector-space isomorphism between $M \otimes F$ and $V_{\Z \gamma}$. Indeed, we first find vector-space isomorphisms between Wakimoto modules for $V^{-3/2}(\sl_2)$ and Fock spaces for $V_\mathrm{Vir}(c_{1,6},0)$, similar to the isomorphisms between Verma modules, and then extend these to an isomorphism between $M\otimes F$ and $V_{\Z\gamma}$. The corresponding character identity is
\begin{equation*}
\ch^*_{M\otimes F}(q^{\mp1/2},q^3)=\ch^*_{V_{\Z\gamma}}(q).
\end{equation*}
Then, using the two nonsplit extensions, we obtain a vector-space isomorphism between $\textsl{SC}\Pi(1)$ and $\Lambda\Pi(5)$, and moreover between certain indecomposable modules. The corresponding character identity is
\begin{equation*}
\ch^*_{\textsl{SC}\Pi(1)}(q^{\mp1/2},q^3)=\ch^*_{\Lambda\Pi(5)}(q).
\end{equation*}

%%%%%%%%%%%%%%%%%%%%%%%%%%%%%%%

\subsection{Physical Interpretation}\label{sec:physics}

We briefly mention an interesting physical context (see \cite{BN22}) for the character identity in \autoref{prop:charid2special} for $p=2$, which we discussed above.

\medskip

The 4d/2d-correspondence \cite{BLLPRR15,BR18}
\begin{equation*}
\mathbb{V}\colon\{\text{4d $\mathcal{N}=2$ SCFTs}\}\longrightarrow \{\text{VOSAs}\}
\end{equation*}
associates a \svoa{} $\mathbb{V}(\T)$ with each 4d $\mathcal{N}=2$ superconformal field theory $\T$ in physics. Then the Schur index $\I_{\T}(q)$ of $\T$ equals the supercharacter
\begin{equation*}
\I_{\T}(q)=\sch^*_{\mathbb{V}(\T)}(q)
\end{equation*}
of $\mathbb{V}(\T)$.

The \svoa{} $\mathbb{V}(\T)$ is expected to be simple, $(1/2)\Z$-graded, quasi-lisse and of CFT-type; see \cite{RR24} for a nice survey. Based on the results in \cite{AK18,Li23}, it is expected that the supercharacter $\sch_{\mathbb{V}(\T)}(q)$ of $\mathbb{V}(\T)$ sits inside a vector-valued quasimodular form of weight~$0$.

\medskip

The work in \cite{BN22} is motivated by an interesting relation between two 4d $\mathcal{N}=2$ superconformal field theories, in particular with regard to their Schur sectors. These two theories are the $\mathcal{N}=4$ supersymmetric Yang-Mills theory with gauge group $\SU(2)$, denoted by $\T_{\SU(2)}$, and a certain $\mathcal{N}=2$ theory $\T_{(3,2)}$ obtained as marginal diagonal $\SU(2)$ gauging of three copies of the $D_3(\SU(2))\cong(A_1,A_3)$ Argyres-Douglas theory \cite{BN16,DVX15}. Their associated \svoa{}s in the 4d/2d-correspondence are given by
\begin{equation*}
\mathbb{V}(\T_{\SU(2)})\cong\mathcal{V}^{(2)}\quad\text{and}\quad\mathbb{V}(\T_{(3,2)})\cong\mathcal{A}^{(6)};
\end{equation*}
see \cite{BMR19,AKM23} and \cite{BN16}, respectively. This is summarised in \autoref{fig:4d2d}.

\begin{figure}[ht]
\begin{tikzcd}
a=c=3/4
&
\T_{\SU(2)}
\arrow{rr}{\mathbb{V}}
&&
\mathcal{V}^{(2)}
\arrow{d}{\psi^-}
&
c=-9
\\
a=c=2
&
\T_{(3,2)}
\arrow{rr}{\mathbb{V}}
&&
\mathcal{A}^{(6)}
&c=-24
\end{tikzcd}
\caption{Graded vector-space isomorphism between the two \svoa{}s $\mathcal{V}^{(2)}$ and $\mathcal{A}^{(6)}$ in the image of the 4d/2d-correspondence.}
\label{fig:4d2d}
\end{figure}

On the level of the Schur indices, it is shown that
\begin{equation*}
\I_{\SU(2)}(q^{1/2},q^3)=\I_{(3,2)}(q),
\end{equation*}
which then implies the relation for the supercharacters of the corresponding \svoa{}s (see \autoref{prop:charid2special})
\begin{equation*}
\sch^*_{\mathcal{V}^{(2)}}(q^{1/2},q^3)=\sch^*_{\mathcal{A}^{(6)}}(q).
\end{equation*}
This then motivates the existence of the map $\smash{\psi^-\colon\mathcal{V}^{(2)}\overset{\sim}{\longrightarrow}\mathcal{A}^{(6)}}$. Of course, one could have equally well stated $\smash{\I_{\SU(2)}(q^{-1/2},q^3)=\I_{(3,2)}(q)}$, leading to the map $\smash{\psi^+\colon\mathcal{V}^{(2)}\overset{\sim}{\longrightarrow}\mathcal{A}^{(6)}}$, by interchanging the roles of $e$ and $f$ in $\sl_2$.

\medskip

As somewhat of a digression, we mention that the relation between the Schur indices of $\T_{(3,2)}$ and $\T_{\SU(2)}$ can be generalised considerably. Indeed, a large class of 4d $\mathcal{N}=2$ superconformal field theories (with $a=c$) can be related to 4d $\mathcal{N}=4$ supersymmetric Yang-Mills theories via their Schur indices \cite{BN22,KLS21} (see also \cite{Jia24,LPY25}).

For instance, one can consider the generalised Argyres-Douglas theories $\T_{(n,N)}$ with $n=2,3,4,6$ and $N\in\Z_{\geq2}$ such that $(n,N)=1$. They are obtained in \cite{BN22,KLS21} by taking a certain collection of $D_{n_i}(\SU(N))$ theories \cite{CD13,CDG13} and gauging a diagonal exactly marginal $\SU(N)$ symmetry. It is expected that
\begin{equation*}
\I_{(n,N)}(q)=\I_{\SU(N)}(q^{n/2-1},q^n),
\end{equation*}
where the right-hand side is the Schur index of the 4d $\mathcal{N}=4$ supersymmetric Yang-Mills theory $\T_{\SU(N)}$ with gauge group $\SU(N)$.

The Schur index $\I_{\SU(N)}$ of $\T_{\SU(N)}$, and in particular its modular properties, are well-studied in the literature; see, e.g., \cite{BDF15,PP22,PWZ22,BSR22}. The modular properties for the Schur indices $\I_{(n,N)}$ of the Argyres-Douglas theories were recently studied in \cite{Jia24}.

It is a natural question to ask:
\begin{prob}\label{prob:generalschur}
Is there a graded vector-space isomorphism between the \svoa{}s $\mathbb{V}(\T_{(n,N)})$ and $\mathbb{V}(\T_{\SU(N)})$, as depicted in \autoref{fig:4d2dconj}?
\end{prob}
\begin{figure}[ht]
\begin{tikzcd}[column sep=small]
a=c=\frac{1}{4}(N^2-1)
&
\T_{\SU(N)}
\arrow{rr}{\mathbb{V}}
&&
\mathbb{V}(\T_{\SU(N)})
\arrow{d}{?}
&
c=-3(N^2-1)
\\
a=c=\frac{n-1}{n}(N^2-1)
&
\T_{(n,N)}
\arrow{rr}{\mathbb{V}}
&&
\mathbb{V}(\T_{(n,N)})
&c=-12\frac{n-1}{n}(N^2-1)
\end{tikzcd}
\caption{Conjectural graded vector-space isomorphisms between certain \svoa{}s in the image of the 4d/2d-correspondence (for each $n=2,3,4,6$).}
\label{fig:4d2dconj}
\end{figure}

We note that such an isomorphism should map a vector of $L_0$-weight $k$ and Lie algebra weight $f$ to a vector of $L_0$-weight $k'=nk+(n/2-1)f$ and preserve parity. However, it is apparent that unless $n=3$, such a vector-space isomorphism will not directly come from the correspondence between modules for the affine \voa{} for $\sl_2$ and the Virasoro \voa{} in this text.

The \svoa{}s $\mathbb{V}(\T_{\SU(N)})=W_{S_N}$ and many of their properties were described in \cite{BLLPRR15,BMR19,AKM23}. As mentioned above, for $N=2$, one recovers the simple quotient $W_{S_2}=\mathcal{V}^{(2)}$ of the small $\mathcal{N}=4$ superconformal algebra at central charge $c=-9$ \cite{Ada16}. However, the \voa{}s corresponding to the Argyres-Douglas theories $\T_{(n,N)}$ are in general expected to be complicated and not yet explicitly known.

\medskip

Returning to the contents of this paper, we recall that we established graded vector-space isomorphisms
\begin{equation*}
\psi^\pm\colon\mathcal{V}^{(p)}\overset{\sim}{\longrightarrow}\mathcal{A}^{(3p)}
\end{equation*}
between the \aia{}s
\begin{equation*}
\mathcal{V}^{(p)}\text{ with }c=3-6p\qquad\text{and}\qquad\mathcal{A}^{(3p)}\text{ with }c=c_{1,3p}=13-18p-2/p
\end{equation*}
for $p\in\Ns$; see \autoref{prop:gvsispecial}. While we did not connect these to 4d superconformal field theories, we suspect that they still play a role in this context.
\begin{prob}\label{prob:VpA3p}
Does the character identity between $\mathcal{V}^{(p)}$ and $\mathcal{A}^{(3p)}$ for $p\in\Ns$ in \autoref{prop:charid2special} have a physical interpretation in the context of the 4d/2d-correspondence?
\end{prob}

%%%%%%%%%%%%%%%%%%%%%%%%%%%%%%%
%%%%%%%%%%%%%%%%%%%%%%%%%%%%%%%

\section{Boundary Admissible Case \texorpdfstring{$q=2$}{q=2}}\label{sec:boundary}

Inspired by the relation between the Schur indices of $\T_{\SU(2)}$ and $\T_{(3,2)}$ discussed in \autoref{sec:physics}, we present further identities between Schur indices of 4d $\mathcal{N}=2$ superconformal field theories based on the character identities in the admissible case in \autoref{sec:admissible}. We specialise to the boundary admissible case $q=2$.

\medskip

For $q=2$, the character identities in \autoref{cor:charid} and \autoref{cor:lowestcharid} for the vacuum modules are
\begin{equation*}
\ch^*_{L_{-2+2/p}(\sl_2)}(q^{\mp1/2},q^3)=\ch^*_{L_\mathrm{Vir}(c_{2,3p},0)}(q)
\end{equation*}
for $p\in\Ns$ odd. On the other hand, we recall (see, e.g., \cite{BR18}) that these \voa{}s correspond to certain 4d $\mathcal{N}=2$ superconformal field theories. Indeed, they are the associated \voa{}s
\begin{equation*}
\mathbb{V}(\T_{(A_1,D_{2n+1})})\cong L_{-2+2/(2n+1)}(\sl_2)
\quad\text{and}\quad
\mathbb{V}(\T_{(A_1,A_{2n})})\cong L_\mathrm{Vir}(c_{2,2n+3},0),
\end{equation*}
$n\in\Ns$, of Argyres-Douglas theories of type $(A_1,D_{2n+1})$ and $(A_1,A_{2n})$, respectively. This is summarised in \autoref{fig:4d2dboundary}.

\begin{figure}[ht]
\begin{tikzcd}[column sep=0]
\begin{array}{l}
a=\frac{n(8n+3)}{8(2n+1)}\\c=\frac{n}{2}
\end{array}
&
\T_{(A_1,D_{2n+1})}
\arrow{rr}{\mathbb{V}}
&~&
L_{-2+2/(2n+1)}(\sl_2)
\arrow{d}{\psi^\pm}
&
c=\frac{3k_{2,2n+1}}{2+k_{2,2n+1}}=-6n
\\
\begin{array}{l}
a=\frac{n(72n+19)}{24(2n+1)}\\
c=\frac{n(18n+5)}{6(2n+1)}
\end{array}
&
\T_{(A_1,A_{6n})}
\arrow{rr}{\mathbb{V}}
&~&
L_\mathrm{Vir}(c_{2,6n+3},0)
&c=c_{2,6n+3}=-\frac{2n(18n+5)}{2n+1}
\end{tikzcd}
\caption{Graded vector-space isomorphism between the \voa{}s $\smash{L_{-2+2/(2n+1)}(\sl_2)}$ and $\smash{L_\mathrm{Vir}(c_{2,6n+3},0)}$ in the image of the 4d/2d-correspondence.}
\label{fig:4d2dboundary}
\end{figure}

On the level of Schur indices, it then follows:
\begin{cor}\label{cor:schur}
For $n\in\Ns$, the following identity of Schur indices holds:
\begin{equation*}
\I_{(A_1,D_{2n+1})}(q^{\mp1/2},q^3)=\I_{(A_1,A_{6n})}(q).
\end{equation*}
\end{cor}
It is natural to ask the following, but note that one difference is that these theories do not satisfy $a=c$ any longer.
\begin{prob}\label{prob:argyresdouglas}
Is there a physical interpretation of the Schur index identity in \autoref{cor:schur}, in the sense of \cite{BN22,KLS21} (see \autoref{sec:physics})?
\end{prob}

\smallskip

The boundary admissible case is also interesting from a mathematical point of view, namely in the context of classical freeness.
\begin{rem}\label{rem:classicallyfree}
It is expected that the \voa{}s $L_{-2+2/p}(\sl_2)$ for $p\in\Ns$ odd are classically free in the sense of \cite{EH21,AEH23}. On the other hand, it was proved in \cite{EH21} that the \voa{}s $L_\mathrm{Vir}(c_{2,p},0)$ are classically free. We believe that our vector-space isomorphism between the \voa{}s $L_{-2+2/p}(\sl_2)$ and $L_\mathrm{Vir}(c_{2,3p},0)$ can be used to prove the classical freeness of affine \voa{}s. In particular, it would be interesting to find an $\sl_2$-version of the basis of $L_\mathrm{Vir}(c_{2,3p},0)$ from \cite{FF93}. We plan to study this problem in forthcoming work.
\end{rem}

%%%%%%%%%%%%%%%%%%%%%%%%%%%%%%%
%%%%%%%%%%%%%%%%%%%%%%%%%%%%%%%

\section{Outlook: Relaxed and Whittaker Modules}\label{sec:relaxed}

So far, we only considered modules for $V^k(\sl_2)$ and $V_\mathrm{Vir}(c,0)$ that have finite-dimensional weight spaces and hence well-defined characters, for the former in the sense of a Jacobi character; see the introduction of \autoref{sec:reps}. In the following, we sketch how to extend the (graded) vector-space isomorphisms to modules not necessarily satisfying this condition, in particular relaxed highest-weight modules and Whittaker modules.

For the former, the formal Jacobi character on the $V^k(\sl_2)$-side still makes sense, while there is no character on the $V_\mathrm{Vir}(c,0)$-side anymore. For the Whittaker modules, neither side has well-defined formal characters.

%%%%%%%%%%%%%%%%%%%%%%%%%%%%%%%

\subsection{Relaxed Highest-Weight Modules}\label{sec:relaxedrelaxed}

Let $k\in\C\setminus\{-2\}$. Let $U$ be any module for $\sl_2$ such that the Casimir element $\Omega=ef+fe+h^2/2$ of $U(\sl_2)$ acts as $\chi\id$. As explained in \autoref{sec:affine}, $U$ is a module for the Lie subalgebra
\begin{equation*}
P=\smash{\hat{\sl}}_2^{\geq0}=\sl_2\otimes\C[t]\oplus\C K\subseteq\hat{\sl}_2
\end{equation*}
of $\smash{\hat\sl_2}$ such that $\sl_2\otimes t\C[t]$ acts trivially, $K$ as $k\id$ and $L_0=\Omega/(2(k+2))$ as $\chi\id/(2(k+2))$. There is the induced module $\smash{V^k(U)=U(\hat{\sl}_2)\otimes_{U(P)}U}$ for $\smash{\hat\sl_2}$ and for $V^k(\sl_2)$. For example, if $U$ is a Verma module or lowest-weight Verma module for $\sl_2$, then $V^k(U)$ is a Verma module or lowest-weight Verma module for $V^k(\sl_2)$, respectively. We found graded vector-space isomorphisms $\smash{\psi^\pm}$ between these modules and Verma modules for the Virasoro algebra (see \autoref{Verma-modules}).

However, as indicated in \autoref{sec:affine}, we can also consider irreducible weight modules for $\sl_2$ that are neither highest- nor lowest-weight. These modules still have $1$-dimensional weight spaces. We explain their construction. For $(\lambda,\chi)\in\C^2$, there is an $\sl_2$-module
\begin{equation*}
E_{\lambda,\chi}=\bigoplus_{n\in\Z}\C v_{\lambda+2n}
\end{equation*}
with the $\sl_2$-action given by
\begin{align*}
hv_{\lambda+2n}&=(\lambda+2n)v_{\lambda+2n},\\
ev_{\lambda+2n}&=v_{\lambda+2n+2},\\
fv_{\lambda+2n}&=\frac{1}{2}\Bigl(\chi-\frac{1}{2}(\lambda+2n-2)^2-(\lambda+2n-2) \Bigr)v_{\lambda+2n-2}
\end{align*}
for $n\in\Z$. Then, indeed, $\Omega$ acts as $\chi\id$. One sees that $E_{\lambda,\chi}$ is reducible if and only if there is a $\mu\in\lambda+2\Z$ such that $\chi=\mu^2/2+\mu$.

Then, we consider the induced $\smash{\hat\sl_2}$-module
\begin{equation}\label{eq:real-relaxed}
\mathcal{E}^{\sl_2}_{\lambda,\chi}=V^k(E_{\lambda,\chi})=U(\hat{\sl}_2)\otimes_{U(P)}E_{\lambda,\chi}.
\end{equation}
This module is a $\N$-gradable module for $V^k(\sl_2)$ called a \emph{relaxed highest-weight module}; see \cite{AM95,RW15,CMY24}. It has a basis consisting of vectors
\begin{equation*}
\prod_{i=1}^\infty(e_{-i})^{k_i}(h_{-i})^{l_i}(f_{-i})^{m_i}v_{\lambda+2n}
\end{equation*}
where $k=(k_i)_{i=1}^\infty$, $l=(l_i)_{i=1}^\infty$ and $m=(m_i)_{i=1}^\infty$ are finite sequences with values in $\N$ and $n\in\Z$.

\medskip

We shall now construct the Virasoro algebra counterpart of relaxed highest-weight modules. It is well known that
\begin{equation*}
e=-L_1,\quad h=-2L_0\quad\text{and}\quad f=L_{-1}
\end{equation*}
generate a Lie subalgebra $\g_1$ of the Virasoro algebra $\mathcal{L}$ isomorphic to $\sl_2$. We then consider $E_{\lambda,\chi}=\bigoplus_{n\in\Z}\C v_{\lambda+2n}$ as a $\g_1$-module (in particular, $L_0$ acts on $v_{\lambda+2n}$ by multiplication with $-\lambda/2-n$), and as a module for the Lie subalgebra
\begin{equation*}
P_\mathcal{L}\coloneqq\g_1\oplus\bigoplus_{s=2}^\infty\C L_s\oplus\C C=\bigoplus_{s=-1}^\infty\C L_s\oplus\C C\subseteq\mathcal{L}
\end{equation*}
of the Virasoro algebra such that $\bigoplus_{s=2}^\infty\C L_s$ acts trivially on it and $C$ as $c\in\C$. Consider now the induced module for $\mathcal{L}$,
\begin{equation}\label{eq:Vir-relaxed}
\mathcal{E}_{\lambda,\chi}^\mathrm{Vir}\coloneqq U(\mathcal{L})\otimes_{U(P_\mathcal{L})}E_{\lambda,\chi}.
\end{equation}
It is a $\N$-gradable module for the Virasoro \voa{} $V_\mathrm{Vir}(c,0)$. It has a basis of vectors of the form
\begin{equation*}
\prod_{i=2}^\infty(L_{-i})^{m_i}v_{\lambda+2n}
\end{equation*}
where $m=(m_i)_{i=2}^\infty$ is a finite sequence with values in $\N$ and $n\in\Z$.

\medskip

We can then define the linear map $\varphi^+\colon\mathcal{E}_{\lambda,\chi}^{\sl_2} \overset{\sim}{\longrightarrow}\mathcal{E}_{\lambda,\chi}^\mathrm{Vir}$ via
\begin{equation*}
\prod_{i=1}^\infty(e_{-i})^{k_i}(h_{-i})^{l_i}(f_{-i})^{m_i}v_{\lambda+2n} \mapsto\prod_{i=1}^\infty(L_{-3i + 1})^{k_i}(L_{-3i})^{l_i}(L_{-3i- 1})^{m_i} v_{\lambda + 2n},
\end{equation*}
i.e.\ by mapping $e_{-i}\mapsto L_{-3i+1}$, $h_{-i}\mapsto L_{-3i}$, $f_{-i}\mapsto L_{-3i-1}$ for $i\in\Ns$ and $v_{\lambda+2n}\mapsto v_{\lambda+2n}$ for $n\in\Z$. Then:
\begin{prop}\label{prop:relaxed}
Let $k\in\C\setminus\{-2\}$ and $c\in\C$. Assume that $(\lambda,\chi)\in\C^2$. Then the map $\smash{\varphi^+\colon\mathcal{E}_{\lambda,\chi}^{\sl_2}\overset{\sim}{\longrightarrow}\mathcal{E}_{\lambda,\chi}^\mathrm{Vir}}$ is a graded vector-space isomorphism. It maps a vector of $(h_0,L_0)$-weight $(f,n)$ to a vector of $L_0$-weight $n'=-f/2+3n-3\chi$.
\end{prop}

While the relaxed highest-weight modules $\smash{\mathcal{E}_{\lambda,\chi}^{\sl_2}}$ for $V^k(\sl_2)$ still have well-defined formal characters $\ch^*_M(w,q)=\tr_Mq^{L_0}w^{h_0}$ (see, e.g., \cite{KR19}, though convergence has to be treated more carefully now), the Virasoro analogues $\smash{\mathcal{E}_{\lambda,\chi}^\mathrm{Vir}}$ do not have well-defined characters $\ch^*_M(q)=\tr_Mq^{L_0}$ any longer. Therefore, we cannot formulate \autoref{prop:relaxed} as a character identity.

We also remark that we can repeat the definition of $\smash{\mathcal{E}_{\lambda,\chi}^\mathrm{Vir}}$ by choosing the Lie subalgebra $\g_1\cong\sl_2$ of the Virasoro algebra $\mathcal{L}$ by $f=-L_1$, $h=2L_0$ and $e=L_{-1}$, i.e.\ by composing with the automorphism $e\mapsto f$, $f\mapsto e$, $h\mapsto -h$ of $\sl_2$. Then, we can define an analogous version $\varphi^-$ of the graded vector-space isomorphism, which now maps a homogeneous vector of $(h_0,L_0)$-weight $(f,n)$ to a vector of $L_0$-weight $n'=f/2+3n-3\chi$.

%%%%%%%%%%%%%%%%%%%%%%%%%%%%%%%

\subsection{Whittaker Modules}\label{sec:whittaker}

We also consider classical nondegenerate Whittaker modules for $\smash{\hat\sl_2}$. They are also smooth modules for $\smash{\hat\sl_2}$ and hence (weak) modules for the universal affine \voa{} $V^k(\sl_2)$. Fix $\lambda,\mu\in\C\setminus\{0\}$ and the level $k\in\C$. Let $\C=\C v_{\lambda,\mu}$ be a $1$-dimensional module for the subalgebra
\begin{equation*}
N\coloneqq\langle e_n,f_{n+1},h_{n+1},K\mid n\in\N\rangle_\C\subseteq\hat{\sl}_2
\end{equation*}
generated by the vector $v_{\lambda,\mu}$ such that
\begin{align*}
e_0 v_{\lambda,\mu}&=\lambda v_{\lambda,\mu},\quad f_1 v_{\lambda,\mu}=\mu v_{\lambda,\mu},\\
K v_{\lambda,\mu}&=kv_{\lambda,\mu},\\
e_{n+1}v_{\lambda,\mu}&=f_{n+2}v_{\lambda,\mu}=h_{n+1}v_{\lambda,\mu}=0\quad\text{for }n\in\N.
\end{align*}
The \emph{universal Whittaker module} $\operatorname{Wh}^k _{\hat{\sl}_2}(\lambda,\mu)$ for $V^k(\sl_2)$ is the induced module
\begin{equation*}
\operatorname{Wh}^k_{\hat{\sl}_2}(\lambda,\mu)=U(\hat{\sl}_2)\otimes_{U(N)}\C v_{\lambda,\mu}.
\end{equation*}
The Whittaker vector is $w^{\sl_2}_{\lambda,\mu}=1\otimes v_{\lambda,\mu}$. It is proved in \cite{ALZ16} that $ \operatorname{Wh}^k_{\hat{\sl}_2}(\lambda,\mu)$ is irreducible when $k\ne-2$. It has a basis consisting of vectors of the form
\begin{equation*}
\prod_{i=0}^\infty(e_{-i-1})^{k_{i+1}}(h_{-i})^{l_i}(f_{-i})^{m_i}w^{\sl_2}_{\lambda,\mu}
\end{equation*}
where $k=(k_i)_{i=1}^\infty$, $l=(l_i)_{i=0}^\infty$, $m=(m_i)_{i=0}^\infty$ are finite sequences with values in $\N$.

\medskip

Consider now Virasoro Whittaker modules. To this end, fix $c\in\C$ and define the Lie subalgebra
\begin{equation*}
N_\mathcal{L}\coloneqq\bigoplus_{s=1}^\infty\C L_s\oplus\C C\subseteq\mathcal{L}
\end{equation*}
of the Virasoro algebra $\mathcal{L}$. Assume that $a,b\in\C\setminus\{0\}$. Let $\C=\C u_{a,b}$ be the $1$-dimensional $N_\mathcal{L}$-module such that
\begin{align*}
L_1u_{a,b}&=a u_{a,b},\quad L_2u_{a,b}=b u_{a,b},\\
Cu_{a,b}&=c u_{a,b},\\
L_su_{a,b}&=0, s\geq3.
\end{align*}
The \emph{universal Whittaker module} $\operatorname{Wh}^c_\mathrm{Vir}(a,b)$ is defined as the induced $\mathcal{L}$-module
\begin{equation*}
\operatorname{Wh}^c_\mathrm{Vir}(a,b)=U(\mathcal{L})\otimes _{U(N_\mathcal{L})}\C u_{a,b}.
\end{equation*}
The Whittaker vector is $w^\mathrm{Vir}_{a,b}=1\otimes u_{a,b}$. $\operatorname{Wh}^c_\mathrm{Vir}(a,b)$ is a weak module for the vertex algebra $V_\mathrm{Vir}(c,0)$. It is proved in \cite{OW09} (see also \cite{LGZ11}) that $\operatorname{Wh}^c_\mathrm{Vir}(a,b)$ is irreducible.

$\operatorname{Wh}^c_\mathrm{Vir}(a,b)$ has a basis consisting of the vectors of the form
\begin{equation*}
\prod_{i=0}^\infty(L_{-i})^{m_i}w^\mathrm{Vir}_{a,b}
\end{equation*}
where $m=(m_i)_{i=0}^\infty$ is a finite sequence with values in $\N$.

\medskip

We can then define the linear map $\smash{\varphi^+\colon\operatorname{Wh}^k _{\hat{\sl}_2}(\lambda,\mu)\overset{\sim}{\longrightarrow}\operatorname{Wh}^c _\mathrm{Vir}(a,b)}$ via
\begin{equation*}
\prod_{i=0}^\infty(e_{-i-1})^{k_{i+1}}(h_{-i})^{l_i}(f_{-i})^{m_i}w^{\sl_2}_{\lambda,\mu}\mapsto\prod_{i=0}^\infty(L_{-3i + 1})^{k_i}(L_{-3i})^{l_i}(L_{-3i- 1})^{m_i}w^\mathrm{Vir}_{a,b},
\end{equation*}
i.e.\ by mapping $e_{-i-1}\mapsto L_{-3i + 1}$, $h_{-i}\mapsto L_{-3i}$, $f_{-i}\mapsto L_{-3i-1}$ for $i\in\N$ and $\smash{w^{\sl_2}_{\lambda,\mu}\mapsto w^\mathrm{Vir}_{a,b}}$. Then:
\begin{prop}\label{prop:isomorphism-whittaker}
Let $k\in\C$ and $c\in\C$. Assume further that $\lambda,\mu\in\C\setminus\{0\}$ and $a,b\in\C\setminus\{0\}$. Then the map $\smash{\varphi^+\colon\operatorname{Wh}^k_{\hat{\sl}_2}(\lambda,\mu)\overset{\sim}{\longrightarrow}\operatorname{Wh}^c_\mathrm{Vir}(a,b)}$ is a vector-space isomorphism.
\end{prop}

%%%%%%%%%%%%%%%%%%%%%%%%%%%%%%%
%%%%%%%%%%%%%%%%%%%%%%%%%%%%%%%

\bibliographystyle{alpha_noseriescomma}
\bibliography{quellen}{}

\end{document}